\documentclass{article}
\usepackage[T1]{fontenc}
\usepackage[utf8]{inputenc}
\usepackage{lmodern}
\usepackage[margin=1in]{geometry}
\usepackage{amsmath,amssymb,amsthm,mathtools}
\usepackage{enumitem}
\usepackage{tikz}
\usepackage{float}
\usepackage{pgfplots}
\pgfplotsset{compat=1.18}
\usetikzlibrary{arrows.meta,calc,patterns,positioning}
\usepackage{xcolor}
\usepackage{PRIMEarxiv}
\usepackage{algorithm}
\usepackage{algpseudocode}
\title{Layered Hypervolume and Magnitude Indicator Ascent}

\definecolor{darkblue}{RGB}{11,27,205}
\definecolor{darkgreen}{RGB}{11,27,05}
\definecolor{darkorange}{RGB}{253,18,11}
\definecolor{layerblue}{RGB}{138,182,255}
\definecolor{layergreen}{RGB}{141,197,135}
\definecolor{layerorange}{RGB}{253,188,142}
\usepackage{hyperref}
\hypersetup{colorlinks=true,linkcolor=blue!50!black,urlcolor=blue!50!black,citecolor=blue!50!black}

\title{Nonsmooth Set-Gradient Ascent to the Pareto Front via Layered Hypervolume and Magnitude Indicators}
\author{
Michael T. M. Emmerich\\
Faculty of Information Technology, University of Jyväskylä, Finland\\
\texttt{michael.t.m.emmerich@jyu.fi}\\
\href{https://orcid.org/0000-0002-7342-2090}{ORCID: 0000-0002-7342-2090}
}
\date{Faculty of Information Technology, University of Jyväskylä, Finland}

\newtheorem{definition}{Definition}
\newtheorem{proposition}{Proposition}
\newtheorem{theorem}{Theorem}
\newtheorem{remark}{Remark}
\newtheorem{example}{Example}

\newcommand{\Mag}{\operatorname{Mag}}
\newcommand{\HV}{\operatorname{HV}}
\newcommand{\clarke}{\partial^{\!C}}
\newcommand{\R}{\mathbb{R}}
\newcommand{\Ginf}{\text{\textcircled{\scriptsize 1}}}
\newcommand{\front}[1]{A^{(#1)}}
\newcommand{\Dom}{\operatorname{D}}

\begin{document}
\maketitle

\begin{abstract}
A nonsmooth set-gradient ascent method is developed for moving finite approximation sets toward the Pareto front in multiobjective optimization. The method optimizes layered set indicators: a base indicator is evaluated on successive nondomination layers, and the layer values are combined with rapidly decreasing weights. This gives ascent directions to nondominated and dominated points while preventing deeper layers from compensating for deterioration of the first front. Two base indicators are treated: the hypervolume indicator and the magnitude indicator of the dominated set, whose expansion over coordinate projections contains extent, projected-area, and volume terms. The scalar objectives are nonsmooth because nondomination layers change combinatorially and the active orthogonal-union geometry changes piecewise. On fixed strata, where layer assignments and active geometry remain unchanged, the indicators are piecewise smooth and chamberwise continuous. For the magnitude indicator, an exact gradient formula is derived as a linear combination of hypervolume gradients of projected shadow sets. Thus, for fixed objective dimension, magnitude gradients have the same asymptotic time complexity as hypervolume gradients. Lexicographic layer aggregation is related to a unary infinitesimal encoding. For finite-\(\epsilon\) surrogates, the main nonsmoothness mechanisms are isolated and chamberwise Lipschitz continuity on bounded sets is proved; a two-point counterexample shows that hard-layer scalarization is not globally continuous across layer switches. The theory motivates a projected finite-difference implementation with repulsion and recovery from stagnation. Numerical examples and reproducible code cover two- and three-objective settings, including objective-space tests, curved fronts, a supersphere benchmark, and traces comparing layered magnitude and hypervolume ascent.
\end{abstract}


\section{Introduction}
A finite approximation set is often the practical output of a multiobjective optimization algorithm.  Instead of returning a single point, the algorithm returns a set of objective vectors that should represent the trade-off surface.  This naturally leads to set-valued search dynamics: the variables are the coordinates of all approximation points, and the objective is a quality indicator of the set as a whole.

We use the following notation for the underlying Pareto-front approximation problem.  The decision space is denoted by $\mathcal X$, and a feasible decision $x\in\mathcal X$ is mapped into objective space by an objective vector $F(x)=(F_1(x),\dots,F_m(x))$.  The component functions are denoted by lowercase symbols $f_i$, so that $F_i(x)=f_i(x)$ is the corresponding objective value.  In figures and tables, $F_i$ denotes an objective-space coordinate or value, while $f_i$ denotes the function itself.  For maximization, an objective vector $y\in\mathbb R^m$ weakly Pareto-dominates $z\in\mathbb R^m$ if $y_i\ge z_i$ for all $i$, and strictly dominates it if, in addition, $y\ne z$.  A decision is efficient if no other feasible decision maps to an objective vector that strictly dominates its objective vector.  The set of all efficient decisions is the efficient set, and its image in objective space is the Pareto front.  The Pareto-front approximation problem is to construct a finite set of decision vectors, or directly a finite set of objective vectors, that represents this front well while keeping the approximation set small.

The most prominent set indicator is the hypervolume indicator, also known as the S-metric in the evolutionary multiobjective optimization literature~\cite{zitzlerthiele1999}.  For maximization with anchor point $r$, the anchored dominated set of a finite objective-space set $B$ is
\[
\Dom_r(B):=\bigcup_{b\in B}[r_1,b_1]\times\cdots\times[r_m,b_m],
\]
with boxes of negative side length omitted when $b_i<r_i$.  Unless explicitly stated otherwise, all dominated sets in this paper are anchored dominated sets; for the origin anchor we write simply $\Dom(B)=\Dom_0(B)$.  The top-dimensional Lebesgue measure of the anchored dominated set,
\[
\HV(B;r):=\lambda_m(\Dom_r(B)),
\]
is the hypervolume indicator.  Its theoretical and computational properties have been studied extensively, including foundations of hypervolume-based multiobjective optimization and algorithms for hypervolume-related computational problems~\cite{augerbaderbrockhoffzitzler2012,guerreirofonsecapaquete2021}.  The other base indicator used here is the magnitude indicator, defined as the \(\ell_1\)-magnitude of the same anchored dominated set.  It extends hypervolume through metric geometry: the top-dimensional volume term recovers the hypervolume of the dominated region, while additional boundary terms record lower-dimensional geometric features such as extents and projected faces~\cite{leinster2013,emmerich2026magnitudedominatedsetspareto}.  Thus the common notation below is a base indicator $\mathcal I$, instantiated either as $\mathcal I(B)=\HV(B;r)$ or as $\mathcal I(B)=\Mag(\Dom_r(B))$.  Since boxes generated by dominated points are contained in boxes generated by nondominated points, both choices depend only on the first nondominated front when applied without layering.  A natural remedy is to decompose the approximation set by nondominated sorting~\cite{deb2001multiobjective}: the first front consists of all points not dominated by any other point in the current set; after removing it, the same operation gives the second front, and iteration gives the deeper fronts.  This paper evaluates $\mathcal I$ on each such layer and studies the resulting \emph{layered indicators} and the nonsmooth ascent dynamics that they induce.

The word \emph{set-gradient} is used here in a concrete computational sense: we differentiate, or approximate by finite differences, a scalar layered indicator with respect to the coordinates of all points in the approximation set.  The resulting direction field moves the set as a whole.  The idea of differentiating an entire finite approximation set arose in hypervolume-gradient work on S-metric maximization~\cite{emmerichdeutzbeume2007}; it is not limited to hypervolume, since set-based derivatives and Newton steps have also been developed for distance-based indicators such as the averaged Hausdorff distance~\cite{uribeetal2020hausdorffnewton}.  If the objective functions are available in closed form and the current configuration lies in a smooth stratum, the chain rule applies: the derivative with respect to decision variables is obtained by multiplying the set-gradient in objective space by the corresponding objective Jacobians.  For hypervolume-gradient fields this relationship and the associated zero structure have been studied in~\cite{emmerichdeutz2014}.  The word \emph{nonsmooth} is essential because the layer structure changes when dominance relations change, and because the orthogonal-union geometry behind the chosen base indicator $\mathcal I$ changes when different boxes, faces, or extrema become active.

\begin{figure}[t]
\centering
\begin{tikzpicture}[scale=0.72]
  \draw[step=1cm,gray!45,thin] (0,0) grid (8,8);
  \draw[->,thick] (0,0) -- (8.35,0) node[right] {$F_1$};
  \draw[->,thick] (0,0) -- (0,8.35) node[above] {$F_2$};
  \foreach \x in {0,...,8} \draw (\x,0) -- +(0,-0.08) node[below] {\x};
  \foreach \y in {0,...,8} \draw (0,\y) -- +(-0.08,0) node[left] {\y};

  \foreach \x/\y in {
    0/0,0/1,0/2,0/3,0/4,0/5,0/6,0/7,
    1/0,1/1,1/2,1/3,
    2/0,2/1,2/2,2/3,
    3/0,3/1,3/2,3/3,
    4/0,4/1,4/2,4/3,
    5/0,5/1,5/2,
    6/0,6/1,6/2}
    {\fill[layerblue,opacity=1] (\x,\y) rectangle ++(1,1);}

  \foreach \x/\y in {
    0/0,0/1,0/2,0/3,0/4,0/5,0/6,
    1/0,1/1,1/2,1/3,
    2/0,2/1,2/2,2/3,
    3/0,3/1,
    4/0,4/1,
    5/0,5/1}
    {\fill[layergreen,opacity=1] (\x,\y) rectangle ++(1,1);}

  \foreach \x/\y in {
    0/0,0/1,0/2,0/3,1/0,2/0,3/0}
    {\fill[layerorange,opacity=1] (\x,\y) rectangle ++(1,1);}

  \fill[darkblue] (1,8) circle (3.2pt);
  \fill[darkblue] (5,4) circle (3.2pt);
  \fill[darkblue] (7,3) circle (3.2pt);

  \fill[darkgreen] (1,7) circle (3.2pt);
  \fill[darkgreen] (3,4) circle (3.2pt);
  \fill[darkgreen] (6,2) circle (3.2pt);

  \fill[darkorange] (1,4) circle (3.2pt);
  \fill[darkorange] (4,1) circle (3.2pt);

  \node[above right] at (1,8) {$(1,8)$};
  \node[above right] at (5,4) {$(5,4)$};
  \node[above right] at (7,3) {$(7,3)$};

  \node[above right] at (1.14,7.14) {$(1,7)$};
  \node[above right] at (3.14,4.14) {$(3,4)$};
  \node[above right] at (6.14,2.14) {$(6,2)$};

  \node[above right] at (1.08,4.12) {$(1,4)$};
  \node[above right] at (4.08,1.12) {$(4,1)$};

  \begin{scope}[shift={(5.0,7.6)}]
    \fill[layerblue,opacity=0.40] (0,0) rectangle +(0.35,0.25);
    \draw[black] (0,0) rectangle +(0.35,0.25);
    \node[anchor=west] at (0.45,0.125) {Layer 1};

    \fill[layergreen,opacity=0.40] (0,-0.45) rectangle +(0.35,0.25);
    \draw[black] (0,-0.45) rectangle +(0.35,0.25);
    \node[anchor=west] at (0.45,-0.325) {Layer 2};

    \fill[layerorange,opacity=0.40] (0,-0.90) rectangle +(0.35,0.25);
    \draw[black] (0,-0.90) rectangle +(0.35,0.25);
    \node[anchor=west] at (0.45,-0.775) {Layer 3};
  \end{scope}
\end{tikzpicture}
\caption{A layered integer-grid example for a base indicator $\mathcal I$ evaluated on successive nondomination layers. The first, second, and third layers are shown in light blue, light green, and light orange, respectively, together with their origin-anchored dominated regions.}
\label{fig:intro-layered-example}
\end{figure}
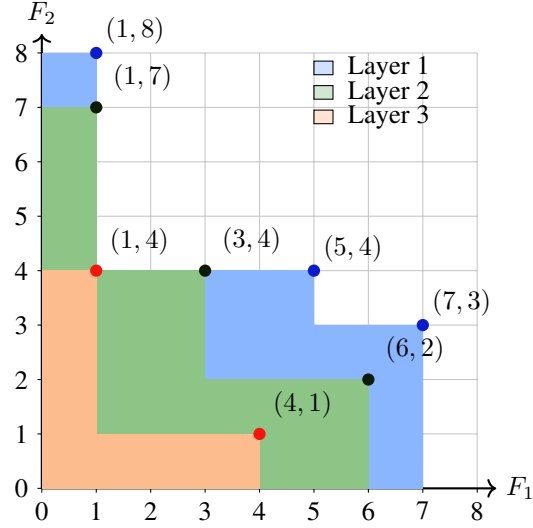

Figure~\ref{fig:intro-layered-example} shows the layer mechanism on a small integer grid.  The picture is drawn at the level of anchored dominated regions and therefore applies to either base indicator $\mathcal I$: the hypervolume indicator uses their top-dimensional Lebesgue measure, while the magnitude indicator uses their \(\ell_1\)-magnitude.  Each layer contributes its own anchored dominated region; the layer weights then enforce the intended hierarchy, so that improving the first nondominated front dominates improvements obtained only in deeper layers.

The paper develops this idea at two levels.  At the framework level, a layered indicator is obtained by choosing a base set indicator $\mathcal I$, here either anchored hypervolume or the magnitude indicator, applying it to each nondomination layer, and adding the weighted contributions.  At the method level, we study nonsmooth set-gradient ascent for these objectives.  On fixed combinatorial chambers the chosen indicator $\mathcal I$ is differentiable or piecewise differentiable; at dominance switches the appropriate interpretation is Clarke-motivated and nonsmooth.

The main theoretical details separate statements that hold for any finite-valued base indicator $\mathcal I$ from statements that are specific to the magnitude indicator.  The lexicographic layered form, the infinitesimal-style unary encoding, and the basic nonsmooth regularity statements apply to both anchored hypervolume and magnitude indicators.  The projected-gradient formula in Section~\ref{subsec:magnitude-gradient-projections}, by contrast, is specific to the magnitude indicator.  We identify two sources of nonsmoothness for the finite-$\varepsilon$ scalar surrogate: changing nondomination layers and changing active orthogonal-union geometry.  On chambers with fixed layer membership we prove Lipschitz continuity on bounded sets for both indicators, whereas a two-point example shows that the hard-layer scalarization is not continuous across layer switches.

The computational part instantiates the same set-gradient idea for both layered magnitude and layered hypervolume indicators.  We use projected finite-difference ascent in small biobjective examples and exact or sweep-based indicator gradients in three-objective examples.  A short-range repulsion term prevents coincident points, and longer convergence runs include a perturbation-and-recovery step after stagnation.  The numerical experiments include a triangle front, a smooth quadratic decision-space map with a curved Pareto front, and a three-objective supersphere benchmark with an analytically described Pareto-front surface.  These experiments are intended to show how the layered construction changes the motion of dominated points and how the magnitude and hypervolume indicators differ once they are placed in the same nonsmooth layered-indicator framework.
\section{Related work and positioning}
The nonsmooth-optimization background for the present paper is classical Clarke analysis; for general definitions and foundational results we refer to Clarke's monograph~\cite{clarke1990}.  In multiobjective optimization, nonsmooth methods have long been studied in settings where vector-valued structure and generalized derivatives must be handled simultaneously; an early example is the interactive nonsmooth method of Miettinen and M{\"a}kel{\"a}~\cite{miettinenmakela1993}.  The present paper does not develop a full nonsmooth convergence theory.  Instead, it uses the Clarke viewpoint to interpret the chamber structure of layered set indicators and to motivate the projected set-gradient ascent schemes used in the experiments.

From the computational side, our numerical scheme is closest in spirit to methods that update an approximation set as a whole.  Direct Multisearch~\cite{custodioetal2011} is a well-known derivative-free framework for multiobjective optimization.  The first hypervolume-indicator gradient method was introduced in the gradient-based/evolutionary relay hybrid for S-metric maximization~\cite{emmerichdeutzbeume2007}, and later hypervolume-indicator gradient ascent work developed step-size control and practical ascent mechanisms for multiobjective optimization~\cite{wangdeutzbackemmerich2017}.  Uncrowded hypervolume gradient ascent~\cite{deistetal2020} further shows how a performance indicator can induce a gradient field that distributes points along a front.  The structure of hypervolume gradient fields and their zeros was studied in~\cite{emmerichdeutz2014}.  Hypervolume also appears in deterministic and exact optimization methods that compute Pareto-front representations through hypervolume scalarizations, including representation algorithms for biobjective optimization~\cite{paqueteschulzestiglmayrlourenco2022} and a greedy hypervolume polychotomic scheme for multiobjective combinatorial optimization~\cite{lopesklamrothpaquete2025}.  Our layered hypervolume construction keeps this volume-based viewpoint but extends the scalar objective to dominated layers as well as to the first front.

The second indicator used here is the magnitude indicator.  Its geometric part is tied to the notion of magnitude in metric geometry, introduced systematically by Leinster~\cite{leinster2013}.  The magnitude of anchored dominated sets used here is specialized to low-dimensional anchored regions and was developed in the companion preprint~\cite{emmerich2026magnitudedominatedsetspareto}.  The contribution of the present paper is to place the magnitude indicator next to hypervolume inside a common layered nonsmooth-indicator framework, and to study the corresponding set-gradient ascent dynamics on representative Pareto-front approximation problems.

\section{Layered set indicators}
For the theoretical development in this section, we specialize the notation from the introduction to the biobjective objective-space setting. Let $A=\{a^{(1)},\dots,a^{(\mu)}\}\subset \R^2$ be a finite approximation set for a biobjective maximization problem, ordered by Pareto dominance as defined above.

The nondomination layers of $A$ are denoted by
\[
\front{1},\front{2},\dots,\front{L},
\]
where $\front{1}$ is the first nondominated front of $A$, $\front{2}$ is the first nondominated front of $A\setminus \front{1}$, and so on.  Equivalently, nondominated sorting repeatedly peels off the current set of maximal points under Pareto dominance; every point receives the rank of the peeling step in which it is removed.  The layer construction is independent of the base indicator.  If $\mathcal I$ is any set indicator that can be evaluated on a finite layer, its finite-$\varepsilon$ layered version is
\[
\mathcal L_{\varepsilon}^{\mathcal I}(A)
 :=\sum_{\ell=1}^{L}\varepsilon^{\ell-1}\mathcal I(\front{\ell}),
 \qquad 0<\varepsilon\ll 1.
\]
In this paper the two main choices of $\mathcal I$ are the anchored hypervolume indicator and the magnitude indicator.  With anchor point $r$, the first choice is $\mathcal I(B)=\HV(B;r)=\lambda_d(\Dom_r(B))$, the top-dimensional Lebesgue measure of the anchored dominated set.  The second choice is $\mathcal I(B)=\Mag(\Dom_r(B))$, the \(\ell_1\)-magnitude of the same anchored dominated set.  When the anchor is the origin we suppress $r$ and write $\Dom(B)$, $\HV(B)$, and $\Mag(\Dom(B))$.

For a finite set $B\subseteq [0,\infty)^2$, let
\[
\Dom(B):=\bigcup_{b\in B}[0,b_1]\times[0,b_2]
\]
be its origin-anchored dominated set. This is the anchored case $r=0$. In two dimensions we use the low-dimensional magnitude formula
\[
\Mag(\Dom(B))=1+\frac{X(B)+Y(B)}{2}+\frac{\HV(B)}{4},
\]
where
\[
X(B)=\max_{b\in B} b_1,
\qquad
Y(B)=\max_{b\in B} b_2,
\]
and $\HV(B)=\operatorname{area}(\Dom(B))$ is the usual anchored hypervolume indicator, i.e. the two-dimensional Lebesgue measure of the anchored dominated set.  This formula makes the relationship between the two choices of $\mathcal I$ explicit: hypervolume is the area term, while magnitude additionally records coordinate extent.

\begin{remark}[Indicator values in Figure~\ref{fig:intro-layered-example}]
For the origin anchor and the maximization convention used in Figure~\ref{fig:intro-layered-example}, the layer values can be read off directly from the staircase areas.  The first layer consists of $(1,8)$, $(5,4)$, and $(7,3)$, hence
\[
\HV(\front{1})=1\cdot 8+(5-1)\cdot 4+(7-5)\cdot 3=30.
\]
With $X(\front{1})=7$ and $Y(\front{1})=8$, the magnitude indicator is
\[
\Mag(\Dom(\front{1}))=1+\frac{7+8}{2}+\frac{30}{4}=16.
\]
Similarly, for the second layer $(1,7)$, $(3,4)$, $(6,2)$ one obtains
\[
\HV(\front{2})=1\cdot 7+(3-1)\cdot 4+(6-3)\cdot 2=21,
\qquad
\Mag(\Dom(\front{2}))=1+\frac{6+7}{2}+\frac{21}{4}=\frac{51}{4}.
\]
For the third layer $(1,4)$, $(4,1)$,
\[
\HV(\front{3})=1\cdot 4+(4-1)\cdot 1=7,
\qquad
\Mag(\Dom(\front{3}))=1+\frac{4+4}{2}+\frac{7}{4}=\frac{27}{4}.
\]
Thus the same three layers have hypervolume values $30,21,7$ and magnitude-indicator values $16,51/4,27/4$, respectively.
\end{remark}

\begin{definition}[Lexicographic layered indicator]
For a base indicator $\mathcal I$, the lexicographic layered indicator of $A$ is
\[
I^{\mathcal I}_{\mathrm{lex}}(A):=
\bigl(\mathcal I(\front{1}),\mathcal I(\front{2}),\dots,\mathcal I(\front{L})\bigr),
\]
ordered lexicographically.  The cases used below are $\mathcal I(B)=\HV(B)$ and $\mathcal I(B)=\Mag(\Dom(B))$, with the origin anchor suppressed.
\end{definition}

\begin{definition}[Unary-encoded layered indicator]
For a base indicator $\mathcal I$, the same layer hierarchy may be formally encoded by
\[
I^{\mathcal I}_{\Ginf}(A):=
\sum_{i=1}^{L}\mathcal I(\front{i})\,\Ginf^{-i+1},
\]
where $\Ginf$ denotes Sergeyev's grossone, an infinite unit from the methodology of numerical infinities and infinitesimals~\cite{sergeyev2017numerical}.  Here the notation is used as a numerically convenient, but order-correct, scalar encoding of infinitesimals of different lexicographic orders associated with the contributions of the nondomination layers.  A worse dominance-rank layer is multiplied by a strictly lower-order infinitesimal and can therefore never compensate for deterioration in any better dominance-rank layer.  Thus the unary notation records the same strict priority hierarchy as the lexicographic formulation, without introducing finite scale-separation constants.  In this paper $\mathcal I$ is instantiated either as the anchored hypervolume indicator, the top-dimensional Lebesgue measure of the anchored dominated set, or as the magnitude indicator, the $\ell_1$-magnitude of the anchored dominated set.  For actual numerical experiments we replace $\Ginf^{-1}$ by a small positive parameter $\varepsilon$.
\end{definition}

\begin{proposition}[Equivalence of the two orderings]
For any fixed base indicator $\mathcal I$ with finite real values on all layers, the grossone-style unary layered indicator and the lexicographic layered indicator induce the same ordering on finite approximation sets. In other words, the unary form is a scalar encoding of the lexicographic hierarchy.
\end{proposition}

\begin{proof}
If two lexicographic vectors first differ in component $j$, then the corresponding difference in
\[
\sum_{i=1}^{L}\mathcal I(\front{i})\,\Ginf^{-i+1}
\]
has leading nonzero term of order $\Ginf^{-j+1}$. That leading term determines the sign of the difference, and therefore the induced ordering agrees with the lexicographic one.
\end{proof}

\begin{remark}
For theory, the lexicographic formulation is cleaner. For numerical work, a finite surrogate of the form
\[
\sum_{i=1}^{L}\varepsilon^{i-1}\mathcal I(\front{i})
\]
with $0<\varepsilon\ll 1$ is easier to implement.  In the magnitude case $\mathcal I(\front{i})=\Mag(\Dom(\front{i}))$; in the hypervolume case $\mathcal I(\front{i})=\HV(\front{i})$.  The practical effect is that dominated points still receive an ascent signal, but only after all higher-priority layers have been respected.  Thus deeper layers can be interpreted as providing a controlled \emph{pulling-back} contribution for dominated points, in the same broad spirit in which hypervolume-gradient methods produce improvement and distribution fields for approximation sets~\cite{emmerichdeutzbeume2007,deistetal2020}.
\end{remark}

\subsection{Magnitude gradients from projected hypervolume gradients}
\label{subsec:magnitude-gradient-projections}

The magnitude indicator also admits a useful gradient decomposition in arbitrary fixed dimension.  Let $B=\{b^{(1)},\dots,b^{(m)}\}\subset [0,\infty)^d$ and write $[d]:=\{1,\dots,d\}$.  For every nonempty coordinate subset $S\subseteq [d]$, let
\[
P_S:\R^d\to\R^{|S|}
\]
be the coordinate projection onto the coordinates in $S$, and let $P_S B:=\{P_S b:b\in B\}$ be the corresponding shadow of $B$.  Denote by $\HV_S(P_S B)$ the origin-anchored hypervolume of this shadow in $\R^{|S|}$.  The finite-dimensional magnitude expansion can then be written as
\[
\Mag_d(\Dom(B))
=
1+
\sum_{\varnothing\neq S\subseteq [d]}
2^{-|S|}\,\HV_S(P_S B).
\]
For $d=3$ this is exactly the decomposition into length, projected-area, and volume terms,
\[
\Mag_3(\Dom(B))
=
1+\frac12 L(B)+\frac14 A(B)+\frac18 V(B),
\]
where
\[
L(B)=\sum_{|S|=1}\HV_S(P_S B),
\qquad
A(B)=\sum_{|S|=2}\HV_S(P_S B),
\qquad
V(B)=\HV_{\{1,2,3\}}(B).
\]

On a differentiability chamber, the set-gradient with respect to the point $b^{(i)}$ is therefore obtained by differentiating the shadow hypervolumes and lifting their gradients back to the full coordinate space:
\[
\nabla_{b^{(i)}}\Mag_d(\Dom(B))
=
\sum_{\varnothing\neq S\subseteq [d]}
2^{-|S|}\,
P_S^{\top}
\nabla_{P_S b^{(i)}}\HV_S(P_S B).
\]
Here $P_S^{\top}:\R^{|S|}\to\R^d$ denotes the adjoint insertion map, i.e. it places the shadow-gradient components into the coordinates in $S$ and fills the remaining coordinates with zeros.  In three dimensions this says that the magnitude gradient is a linear combination of the lifted gradients of the one-dimensional length shadows, the two-dimensional area shadows, and the three-dimensional hypervolume term:
\[
\nabla_{b^{(i)}}\Mag_3(\Dom(B))
=
\frac12\nabla_{b^{(i)}}L(B)
+
\frac14\nabla_{b^{(i)}}A(B)
+
\frac18\nabla_{b^{(i)}}V(B).
\]
For the whole labelled set, the gradient is the block vector
\[
\nabla_B\Mag_d(\Dom(B))
=
\bigl(
\nabla_{b^{(1)}}\Mag_d(\Dom(B)),\dots,
\nabla_{b^{(m)}}\Mag_d(\Dom(B))
\bigr),
\]
and its scalar magnitude can be measured, for example, by the Frobenius norm
\[
\bigl\|\nabla_B\Mag_d(\Dom(B))\bigr\|_F
=
\left(
\sum_{i=1}^{m}
\bigl\|\nabla_{b^{(i)}}\Mag_d(\Dom(B))\bigr\|_2^2
\right)^{1/2}.
\]
At nonsmooth configurations, where the active orthogonal-union geometry changes or where a point becomes tied in a projected shadow, the formula is used chamberwise.  Neighboring one-sided gradients motivate the local direction field, while the numerical implementation uses finite-difference derivatives of the chosen finite-$\varepsilon$ scalar surrogate rather than an explicit generalized-derivative computation.

This representation is also computationally convenient.  If $T(m,k)$ denotes the time needed to compute a $k$-dimensional hypervolume gradient for $m$ points, then the projected formula requires
\[
\sum_{k=1}^{d}\binom{d}{k}T(m,k)
\]
work, since it evaluates one hypervolume-gradient computation for each of the $2^d-1$ nonempty coordinate shadows.  For fixed objective-space dimension $d$, this is only a constant-factor overhead relative to the ordinary $d$-dimensional hypervolume-gradient computation; the magnitude gradient is simply a linear composition of hypervolume gradients in dimensions $1,\dots,d$.  Thus, in the fixed-dimensional setting considered in this paper, computing the gradient of the magnitude indicator has the same asymptotic time complexity in the number of points as computing the hypervolume gradient, up to this projection factor.

For the layered magnitude indicator, this projected-gradient formula is applied layer by layer.  If $a^{(i)}$ belongs to layer $\front{\ell(i)}$ and the layer structure is fixed, then
\[
\nabla_{a^{(i)}}\mathcal L_{\varepsilon}^{\Mag}(A)
=
\varepsilon^{\ell(i)-1}
\nabla_{a^{(i)}}\Mag_d\bigl(\Dom(\front{\ell(i)})\bigr),
\]
with the right-hand side again evaluated through the projected hypervolume gradients of the shadows of that layer.

\section{Piecewise-smooth structure of layered indicators}
The layered indicator is only piecewise smooth. There are two distinct sources of nonsmoothness.

First, the layer structure itself changes combinatorially when two points exchange dominance status. Second, even within a fixed layer, the functions $X(B)$, $Y(B)$, and $\HV(B)$ are only piecewise smooth because different points become extremal or active in the staircase representation of the dominated region. Consequently, the layered $\mathcal I$-indicator, in lexicographic or finite-$\varepsilon$ scalar form, is naturally associated with nonsmooth optimization.  For $\mathcal I=\Mag$, the low-dimensional formula includes additional nonsmooth maximum or projected-extent terms such as $X(B)$ and $Y(B)$, while for $\mathcal I=\HV$ the nonsmoothness comes from the active-staircase geometry of the orthogonal union.

\begin{remark}[Role of Clarke-style nonsmooth analysis]
The role of Clarke-style nonsmooth analysis in the present work is primarily structural and interpretive.  The layered indicators are smooth only on regions where the nondomination stratification and the active orthogonal-union geometry of the corresponding dominated sets remain fixed.  Across boundaries of such regions, classical gradients may fail to exist, and neighboring smooth gradients need not agree.  For finite-$\varepsilon$ surrogates that are locally Lipschitz on fixed layer chambers, Clarke calculus provides a natural language for describing such transitions and for interpreting directions obtained from nearby smooth pieces.  The numerical method used below, however, should not be read as an algorithm that explicitly computes Clarke subdifferentials or Clarke normal cones.  It is a projected finite-difference set-gradient ascent scheme for layered set indicators on a stratified nonsmooth landscape, motivated by the local piecewise-smooth structure.
\end{remark}

\begin{remark}[Finite differences at layer and geometry boundaries]
The finite-$\varepsilon$ scalarization is the real-valued objective used in the finite-difference implementation,
\[
J_\varepsilon(A)=\sum_{\ell\ge 1}\varepsilon^{\ell-1}\mathcal I(L_\ell(A)),
\]
where $L_\ell(A)$ denotes the $\ell$-th nondomination layer and $\mathcal I$ is the chosen base indicator, with the repulsion term added in the numerical runs when stated.  On a smooth chamber, where both the layer assignment and the active dominated-set geometry are fixed, symmetric finite differences approximate the ordinary derivative, equivalently the ordinary gradient with respect to the point coordinates, of this scalar surrogate.  At boundaries where only the active geometry changes, for example when orthogonal-union cells or projected shadow cells change, the objective is typically continuous but nonsmooth; the finite-difference direction samples neighboring smooth formulas and acts as a local regularization of the piecewise-smooth gradient field.  At boundaries where the nondomination layers themselves change, the hard layer assignment can cause value jumps, because a point may suddenly be weighted by $1$, $\varepsilon$, $\varepsilon^2$, and so on.  Near such a boundary, a finite-difference quotient may therefore reflect the layer-switching jump rather than a classical derivative.  In the projected normalized implementation this does not produce an unbounded step, but it may create a strong directional signal toward the chamber with the better layer rank.  Thus the method is best understood as projected finite-difference ascent on a stratified, partly nonsmooth and partly discontinuous surrogate, rather than as an exact Clarke-subgradient method for the full hard-layer scalarization.
\end{remark}

To keep the notation simple in the theoretical biobjective analysis, let
\[
J^{\mathcal I}_{\varepsilon,\tau}(A):=\sum_{i=1}^{L}\varepsilon^{i-1}\mathcal I(\front{i})-\tau R_\sigma(A),
\]
where the base indicator is either
\[
\mathcal I(B)=\HV(B),
\qquad\text{or}\qquad
\mathcal I(B)=\Mag(\Dom(B)).
\]
The origin anchor is suppressed in this notation, and $R_\sigma$ is a short-range repulsion term.  When the choice of base indicator is clear, we abbreviate $J^{\mathcal I}_{\varepsilon,\tau}$ to $J_{\varepsilon,\tau}$.

\begin{definition}[Short-range repulsion]
For $\sigma>0$ define
\[
R_\sigma(A):=\sum_{1\le i<j\le \mu}\exp\!\Bigl(-\frac{\|a^{(i)}-a^{(j)}\|_2^2}{\sigma^2}\Bigr).
\]
The term $-\tau R_\sigma(A)$ penalizes nearly coincident points and prevents exact duplicates in the numerical iterations.
\end{definition}

The repulsion term is not part of the theoretical layered indicator. It is introduced only as a numerical regularizer.

\begin{proposition}[Local smoothness on fixed combinatorial strata]
Fix $\mu$ and consider configurations $A=(a^{(1)},\dots,a^{(\mu)})\in(\R^2)^\mu$. For either base indicator $\mathcal I\in\{\HV,\Mag\circ\Dom\}$, the map $A\mapsto J^{\mathcal I}_{\varepsilon,\tau}(A)$ is smooth on every region of configuration space on which
\begin{enumerate}[label=(\roman*)]
\item the membership of every point in every nondomination layer is fixed,
\item the staircase combinatorics defining $\HV(\front{\ell})$ are fixed for each layer, and
\item in the magnitude case, the points realizing the projected extents $X(\front{\ell})$ and $Y(\front{\ell})$ are fixed for each layer.
\end{enumerate}
\end{proposition}

\begin{proof}
On such a region, each layer $\front{\ell}$ is represented by a fixed subset of indices. In two dimensions, once the staircase order of the active points is fixed, the anchored hypervolume $\HV(\front{\ell})$ is given by a fixed polynomial expression in the relevant coordinates. Hence the layered-hypervolume part is smooth. In the magnitude case, the additional terms $X(\front{\ell})$ and $Y(\front{\ell})$ are coordinate projections of fixed points on the stated region and are therefore affine; thus $\Mag(\Dom(\front{\ell}))=1+\tfrac12(X+Y)+\tfrac14\HV$ is smooth as well. The repulsion term is smooth everywhere, so the weighted sum defining $J^{\mathcal I}_{\varepsilon,\tau}$ is smooth for either choice of $\mathcal I$.
\end{proof}

The preceding proposition is a stratified smoothness statement. For actual estimates it is useful to record a chamberwise Lipschitz theorem that does \emph{not} require the active staircase combinatorics to remain fixed.

\begin{theorem}[Lipschitz continuity on chambers with fixed layer membership]
Fix $M>0$ and let $K_M:=([0,M]^2)^\mu$. Let $\mathcal P=(I_1,\dots,I_L)$ be an ordered partition of $\{1,\dots,\mu\}$ and let $\mathcal C_{\mathcal P}\subset K_M$ denote the set of configurations whose nondomination layers are exactly indexed by $I_1,\dots,I_L$. For either base indicator $\mathcal I\in\{\HV,\Mag\circ\Dom\}$, the restriction of $J^{\mathcal I}_{\varepsilon,\tau}$ to $\mathcal C_{\mathcal P}$ is Lipschitz on $K_M\cap \mathcal C_{\mathcal P}$ with respect to the product $\ell^1$ norm
\[
\|A-B\|_{1,\mu}:=\sum_{i=1}^\mu \|a^{(i)}-b^{(i)}\|_1.
\]
More precisely, there exists a constant $L=L(M,\mu,\varepsilon,\tau,\sigma)$ such that
\[
|J^{\mathcal I}_{\varepsilon,\tau}(A)-J^{\mathcal I}_{\varepsilon,\tau}(B)|\le L\,\|A-B\|_{1,\mu}
\qquad\text{for all }A,B\in K_M\cap \mathcal C_{\mathcal P}.
\]
\end{theorem}

\begin{proof}
We first estimate both base indicators on a fixed labelled subset. Let $B=(b^{(1)},\dots,b^{(m)})$ and $C=(c^{(1)},\dots,c^{(m)})$ be two labelled $m$-point configurations in $([0,M]^2)^m$. For the magnitude terms, the maximum is $1$-Lipschitz with respect to the $\ell^1$ norm, so
\[
|X(B)-X(C)|\le \sum_{i=1}^m |b^{(i)}_1-c^{(i)}_1|,
\qquad
|Y(B)-Y(C)|\le \sum_{i=1}^m |b^{(i)}_2-c^{(i)}_2|.
\]
For the hypervolume term, write $Q(p):=[0,p_1]\times[0,p_2]$ for the anchored box generated by $p\in[0,M]^2$. Then
\[
\Dom(B)=\bigcup_{i=1}^m Q\bigl(b^{(i)}\bigr),
\qquad
\Dom(C)=\bigcup_{i=1}^m Q\bigl(c^{(i)}\bigr).
\]
The symmetric difference of the two unions is contained in the union of the individual symmetric differences, hence
\[
\bigl|\HV(B)-\HV(C)\bigr|
\le \operatorname{area}\bigl(\Dom(B)\,\triangle\,\Dom(C)\bigr)
\le \sum_{i=1}^m \operatorname{area}\Bigl(Q\bigl(b^{(i)}\bigr)\,\triangle\,Q\bigl(c^{(i)}\bigr)\Bigr).
\]
For each pair of anchored boxes one has
\[
\operatorname{area}\Bigl(Q\bigl(b^{(i)}\bigr)\,\triangle\,Q\bigl(c^{(i)}\bigr)\Bigr)
\le M\,\bigl\|b^{(i)}-c^{(i)}\bigr\|_1,
\]
because changing the $x$-extent by $\delta_x$ and the $y$-extent by $\delta_y$ can alter the box only by vertical and horizontal strips of total area at most $M(|\delta_x|+|\delta_y|)$. Therefore
\[
\bigl|\HV(B)-\HV(C)\bigr|\le M\sum_{i=1}^m \bigl\|b^{(i)}-c^{(i)}\bigr\|_1.
\]
The hypervolume estimate above already gives the required Lipschitz bound for $\mathcal I=\HV$. Combining it with the two maximum estimates gives the magnitude bound
\[
\bigl|\Mag(\Dom(B)) - \Mag(\Dom(C))\bigr|
\le \Bigl(\frac12+\frac{M}{4}\Bigr)\sum_{i=1}^m \bigl\|b^{(i)}-c^{(i)}\bigr\|_1.
\]
Now let $A,B\in K_M\cap\mathcal C_{\mathcal P}$. Since $A$ and $B$ have the same layer membership pattern, the $\ell$-th layers are obtained from the same index set $I_\ell$ in both configurations. Applying the corresponding estimate layer by layer and summing with weights $\varepsilon^{\ell-1}$ yields a Lipschitz bound for either the layered-hypervolume part or the layered-magnitude part. Finally, the repulsion term $R_\sigma$ is smooth on the compact set $K_M$, hence Lipschitz there. Adding the two bounds proves the claim.
\end{proof}

\begin{proposition}[The hard-layer scalarization is not continuous across layer switches]
Fix $\varepsilon>0$ and ignore the repulsion term for simplicity. For either base indicator $\mathcal I\in\{\HV,\Mag\circ\Dom\}$, the map
\[
A\mapsto \sum_{\ell}\varepsilon^{\ell-1}\mathcal I(\front{\ell}(A))
\]
need not be continuous at configurations where the nondomination layers change.
\end{proposition}

\begin{proof}
Consider the two-point configuration $A_0=\{a,b_0\}$ with
\[
a=(1,\tfrac12),
\qquad
b_0=(\tfrac{9}{10},\tfrac12).
\]
Here $a$ dominates $b_0$, so the first layer is $\{a\}$ and the second layer is $\{b_0\}$. Hence
\[
J^{\mathcal I}_{\varepsilon,0}(A_0)=\mathcal I(\{a\}) + \varepsilon\,\mathcal I(\{b_0\}).
\]
For $\eta>0$ let
\[
b_\eta=(\tfrac{9}{10},\tfrac12+\eta),
\qquad
A_\eta=\{a,b_\eta\}.
\]
Then $a$ and $b_\eta$ are nondominated, so the whole configuration forms a single first layer and
\[
J^{\mathcal I}_{\varepsilon,0}(A_\eta)=\mathcal I(\{a,b_\eta\}).
\]
As $\eta\downarrow 0$, the anchored dominated set of $\{a,b_\eta\}$ converges to the anchored dominated set of $\{a\}$, both in hypervolume and in $\ell_1$-magnitude. Therefore
\[
\lim_{\eta\downarrow 0} J^{\mathcal I}_{\varepsilon,0}(A_\eta)=\mathcal I(\{a\}).
\]
Since $\mathcal I(\{b_0\})>0$ for both the hypervolume and magnitude indicators, this limit is strictly smaller than $J^{\mathcal I}_{\varepsilon,0}(A_0)$. Thus the hard-layer finite-$\varepsilon$ scalarization has a downward jump at $A_0$ and is not continuous there for either choice of base indicator.
\end{proof}

\begin{remark}
The theorem and counterexample together identify the real regularity picture. On chambers with fixed layer membership one has standard Lipschitz control and may invoke Clarke calculus in the usual way. Across layer-switching boundaries, however, the hard-layer finite-$\varepsilon$ scalarization can jump and therefore is not a locally Lipschitz function on the full configuration space. A full-space Clarke analysis is therefore unavailable for this hard-layer scalar surrogate; the discontinuity arises precisely from the abrupt reassignment of points between layers when dominance relations change.
\end{remark}

\section{Nonsmooth set-gradient ascent}
Let $\Omega\subset \R^2$ be a closed feasible region in objective space and let $\Omega^\mu$ be the feasible set for approximation sets of size $\mu$. There are two closely related algorithmic viewpoints.

The first is an \emph{idealized} projected Clarke-ascent step for a layered scalar surrogate $J_{\varepsilon,\tau}$, either the magnitude version or the hypervolume version.  The term set-gradient refers to the gradient, generalized gradient, or finite-difference direction with respect to all point coordinates of the approximation set.

\begin{definition}[Ideal projected Clarke-ascent step]
Given a current approximation set $A^k=(a^{k,(1)},\dots,a^{k,(\mu)})\in\Omega^\mu$, choose a generalized gradient
\[
g^k=(g^{k,(1)},\dots,g^{k,(\mu)})\in \clarke J_{\varepsilon,\tau}(A^k).
\]
For the normalized pointwise variant, define
\[
\widehat g^{k,(i)}=
\begin{cases}
\dfrac{g^{k,(i)}}{\|g^{k,(i)}\|_2}, & g^{k,(i)}\neq 0,\\[1ex]
0, & g^{k,(i)}=0,
\end{cases}
\qquad i=1,\dots,\mu,
\]
and set
\[
\widehat g^k=(\widehat g^{k,(1)},\dots,\widehat g^{k,(\mu)}).
\]
The projected update is then
\[
A^{k+1}=\Pi_{\Omega^\mu}\bigl(A^k+\alpha_k \widehat g^k\bigr),
\]
where $\Pi_{\Omega^\mu}$ denotes Euclidean projection onto the feasible set and $\alpha_k>0$ is chosen either by backtracking or as a small diffusion step size.
\end{definition}

The second viewpoint is the one actually used in the small biobjective experiments: a projected \emph{finite-difference} set-gradient ascent scheme motivated by the preceding Clarke picture. In a full nonsmooth analysis one would work directly with the Clarke subdifferential. For the small experiments below, however, we use symmetric finite differences to construct a direction field, then normalize pointwise when indicated, and finally project back to the feasible set. In the curved-front experiment this pointwise normalization prevents a single large local gradient from dominating the motion of the whole approximation set and produces a small-step gradient-diffusion dynamics that is easier to visualize and numerically more robust in the presence of temporary near-collisions.

For later reproducibility it is useful to separate the numerical core from the surrounding plotting or interactive code. The method used throughout the paper is summarized in Appendix~\ref{app:algorithm}, where we give commented pseudocode, explain all symbols before the algorithm, and provide a step-by-step walkthrough of one iteration.

\medskip
\noindent\textbf{Implemented finite-difference ascent scheme.}
The numerical method associated with the figures is the following: start from a feasible configuration, evaluate the chosen layered surrogate, build a set-gradient direction field either by symmetric finite differences or by an explicit indicator-gradient routine, optionally normalize the direction pointwise, take a small ascent step, and project the result back to the feasible set. Appendix~\ref{app:algorithm} records this scheme in pseudocode and makes explicit the two viewpoints used in the paper: objective-space optimization, where the state variable is the approximation set itself, and decision-space optimization, where the layered objective is evaluated only after mapping the decision points into objective space.

\subsection*{How the problem is solved numerically in practice}
In the present computations the optimization variables are the coordinates of the $\mu$ points themselves. For objective-space examples the state variable is directly the approximation set $A\in\Omega^\mu$. For decision-space examples the state variable is a decision configuration $X=(x^{(1)},\dots,x^{(\mu)})$, the objective vectors are $F(x^{(i)})$, and the layered objective is evaluated after mapping the decision points into objective space. The generic implementation uses a fixed finite-difference radius $h$, a step size $\alpha$, optional pointwise normalization, and Euclidean projection onto the feasible set; see Appendix~\ref{app:algorithm} for the exact pseudocode and default constants used in the reference implementation.

The small-step normalized variant is particularly convenient for visualization. In the numerical examples it first moves the points toward the first front and then lets the repulsion term separate them along that front. For nonlinear fronts, such as the quadratic ten-point example below, this produces a visible front-following diffusion and turns out to be quite robust.

\section{A simple biobjective test problem}
We consider the objective-space problem
\[
\max \{(F_1,F_2): (F_1,F_2)\in \Omega\},
\qquad
\Omega:=\{(F_1,F_2)\in [0,1]^2: F_1+F_2\le 1\}.
\]
The Pareto front is the line segment
\[
\mathcal{P}:=\{(F_1,F_2)\in [0,1]^2: F_1+F_2=1\}.
\]
Thus the interior of $\Omega$ consists of dominated objective vectors, while the upper boundary segment is the first nondominated front.

For all experiments we use the same scalar surrogate and the same numerical parameters,
\[
\varepsilon=10^{-3},\qquad \tau=10^{-2},\qquad \sigma=0.06.
\]
The examples differ only in the approximation-set size $\mu$.

\section{Numerical examples}
The experiments start from interior points and apply the chamberwise projected finite-difference ascent scheme to $J_{\varepsilon,\tau}$.  The examples illustrate movement toward the Pareto front, disappearance of deeper layers near the optimizer, and tangential redistribution by the repulsion term.

The first two ten-point runs shown in Figure~\ref{fig:ten-point-paths} use the same objective-space triangle problem but different initializations. In the left panel the approximation set starts from ten mutually nondominated interior points on the line $F_1+F_2=0.7$. In the right panel it starts from a dominated triangular $4+3+2+1$ configuration in the lower-left part of $\Omega$, with the points placed in the subtriangle with vertices $(0,0)$, $(0,0.5)$, and $(0.5,0)$. In both cases the ascent reaches a nearly evenly spread ten-point front approximation; the corresponding first-layer values are approximately $\HV\approx 0.44392$ and $\Mag\approx 2.11098$ for the nondominated line start, and $\HV\approx 0.44437$ and $\Mag\approx 2.11109$ for the dominated triangular start.

\begin{figure}[H]
\centering
\begin{minipage}[t]{0.48\linewidth}
\centering
\vspace{0pt}
\begin{tikzpicture}[scale=7.2, every node/.style={font=\scriptsize}]
  \draw[very thin,gray!30] (0,0) grid[step=0.1] (1,1);
  \fill[blue!4] (0,0) -- (1,0) -- (0,1) -- cycle;
  \draw[thick] (0,0) -- (1,0) -- (0,1) -- cycle;
  \draw[ultra thick,blue!70!black] (1,0) -- (0,1);
  \foreach \x/\lab in {0/0,1/1} {\draw[thin,gray!65] (\x,0) -- +(0,-0.014); \node[below,font=\scriptsize] at (\x,-0.022) {\lab};}
  \foreach \y/\lab in {0/0,1/1} {\draw[thin,gray!65] (0,\y) -- +(-0.014,0); \node[left,font=\scriptsize] at (-0.022,\y) {\lab};}
  \node[above] at (0.5,1.03) {$\mu=10$: line start};
  \draw[gray!70, line width=0.35pt, -{Latex[length=1.05mm,width=0.75mm]}] (0.020,0.680) -- (0.024,0.840);
  \draw[gray!62, line width=0.35pt, -{Latex[length=1.05mm,width=0.75mm]}] (0.024,0.840) -- (0.016,0.984);
  \draw[blue!40, line width=0.35pt, -{Latex[length=1.05mm,width=0.75mm]}] (0.016,0.984) -- (0.000,1.000);
  \draw[cyan!50!black, line width=0.35pt, -{Latex[length=1.05mm,width=0.75mm]}] (0.000,1.000) -- (0.000,1.000);
  \draw[teal!60!black, line width=0.35pt, -{Latex[length=1.05mm,width=0.75mm]}] (0.000,1.000) -- (0.000,1.000);
  \draw[green!50!black, line width=0.35pt, -{Latex[length=1.05mm,width=0.75mm]}] (0.000,1.000) -- (0.000,1.000);
  \draw[orange!70!black, line width=0.35pt, -{Latex[length=1.05mm,width=0.75mm]}] (0.000,1.000) -- (0.000,1.000);
  \draw[red!75!black, line width=0.35pt, -{Latex[length=1.05mm,width=0.75mm]}] (0.000,1.000) -- (0.000,1.000);
  \draw[gray!70, line width=0.35pt, -{Latex[length=1.05mm,width=0.75mm]}] (0.093,0.607) -- (0.158,0.752);
  \draw[gray!62, line width=0.35pt, -{Latex[length=1.05mm,width=0.75mm]}] (0.158,0.752) -- (0.182,0.818);
  \draw[blue!40, line width=0.35pt, -{Latex[length=1.05mm,width=0.75mm]}] (0.182,0.818) -- (0.173,0.827);
  \draw[cyan!50!black, line width=0.35pt, -{Latex[length=1.05mm,width=0.75mm]}] (0.173,0.827) -- (0.172,0.828);
  \draw[teal!60!black, line width=0.35pt, -{Latex[length=1.05mm,width=0.75mm]}] (0.172,0.828) -- (0.171,0.829);
  \draw[green!50!black, line width=0.35pt, -{Latex[length=1.05mm,width=0.75mm]}] (0.171,0.829) -- (0.171,0.829);
  \draw[orange!70!black, line width=0.35pt, -{Latex[length=1.05mm,width=0.75mm]}] (0.171,0.829) -- (0.171,0.829);
  \draw[red!75!black, line width=0.35pt, -{Latex[length=1.05mm,width=0.75mm]}] (0.171,0.829) -- (0.171,0.829);
  \draw[gray!70, line width=0.35pt, -{Latex[length=1.05mm,width=0.75mm]}] (0.167,0.533) -- (0.254,0.667);
  \draw[gray!62, line width=0.35pt, -{Latex[length=1.05mm,width=0.75mm]}] (0.254,0.667) -- (0.283,0.717);
  \draw[blue!40, line width=0.35pt, -{Latex[length=1.05mm,width=0.75mm]}] (0.283,0.717) -- (0.278,0.722);
  \draw[cyan!50!black, line width=0.35pt, -{Latex[length=1.05mm,width=0.75mm]}] (0.278,0.722) -- (0.277,0.723);
  \draw[teal!60!black, line width=0.35pt, -{Latex[length=1.05mm,width=0.75mm]}] (0.277,0.723) -- (0.277,0.723);
  \draw[green!50!black, line width=0.35pt, -{Latex[length=1.05mm,width=0.75mm]}] (0.277,0.723) -- (0.277,0.723);
  \draw[orange!70!black, line width=0.35pt, -{Latex[length=1.05mm,width=0.75mm]}] (0.277,0.723) -- (0.277,0.723);
  \draw[red!75!black, line width=0.35pt, -{Latex[length=1.05mm,width=0.75mm]}] (0.277,0.723) -- (0.277,0.723);
  \draw[gray!70, line width=0.35pt, -{Latex[length=1.05mm,width=0.75mm]}] (0.240,0.460) -- (0.340,0.584);
  \draw[gray!62, line width=0.35pt, -{Latex[length=1.05mm,width=0.75mm]}] (0.340,0.584) -- (0.372,0.628);
  \draw[blue!40, line width=0.35pt, -{Latex[length=1.05mm,width=0.75mm]}] (0.372,0.628) -- (0.370,0.630);
  \draw[cyan!50!black, line width=0.35pt, -{Latex[length=1.05mm,width=0.75mm]}] (0.370,0.630) -- (0.369,0.631);
  \draw[teal!60!black, line width=0.35pt, -{Latex[length=1.05mm,width=0.75mm]}] (0.369,0.631) -- (0.369,0.631);
  \draw[green!50!black, line width=0.35pt, -{Latex[length=1.05mm,width=0.75mm]}] (0.369,0.631) -- (0.369,0.631);
  \draw[orange!70!black, line width=0.35pt, -{Latex[length=1.05mm,width=0.75mm]}] (0.369,0.631) -- (0.369,0.631);
  \draw[red!75!black, line width=0.35pt, -{Latex[length=1.05mm,width=0.75mm]}] (0.369,0.631) -- (0.369,0.631);
  \draw[gray!70, line width=0.35pt, -{Latex[length=1.05mm,width=0.75mm]}] (0.313,0.387) -- (0.423,0.503);
  \draw[gray!62, line width=0.35pt, -{Latex[length=1.05mm,width=0.75mm]}] (0.423,0.503) -- (0.458,0.542);
  \draw[blue!40, line width=0.35pt, -{Latex[length=1.05mm,width=0.75mm]}] (0.458,0.542) -- (0.457,0.543);
  \draw[cyan!50!black, line width=0.35pt, -{Latex[length=1.05mm,width=0.75mm]}] (0.457,0.543) -- (0.457,0.543);
  \draw[teal!60!black, line width=0.35pt, -{Latex[length=1.05mm,width=0.75mm]}] (0.457,0.543) -- (0.457,0.543);
  \draw[green!50!black, line width=0.35pt, -{Latex[length=1.05mm,width=0.75mm]}] (0.457,0.543) -- (0.457,0.543);
  \draw[orange!70!black, line width=0.35pt, -{Latex[length=1.05mm,width=0.75mm]}] (0.457,0.543) -- (0.457,0.543);
  \draw[red!75!black, line width=0.35pt, -{Latex[length=1.05mm,width=0.75mm]}] (0.457,0.543) -- (0.457,0.543);
  \draw[gray!70, line width=0.35pt, -{Latex[length=1.05mm,width=0.75mm]}] (0.387,0.313) -- (0.503,0.423);
  \draw[gray!62, line width=0.35pt, -{Latex[length=1.05mm,width=0.75mm]}] (0.503,0.423) -- (0.542,0.458);
  \draw[blue!40, line width=0.35pt, -{Latex[length=1.05mm,width=0.75mm]}] (0.542,0.458) -- (0.543,0.457);
  \draw[cyan!50!black, line width=0.35pt, -{Latex[length=1.05mm,width=0.75mm]}] (0.543,0.457) -- (0.543,0.457);
  \draw[teal!60!black, line width=0.35pt, -{Latex[length=1.05mm,width=0.75mm]}] (0.543,0.457) -- (0.543,0.457);
  \draw[green!50!black, line width=0.35pt, -{Latex[length=1.05mm,width=0.75mm]}] (0.543,0.457) -- (0.543,0.457);
  \draw[orange!70!black, line width=0.35pt, -{Latex[length=1.05mm,width=0.75mm]}] (0.543,0.457) -- (0.543,0.457);
  \draw[red!75!black, line width=0.35pt, -{Latex[length=1.05mm,width=0.75mm]}] (0.543,0.457) -- (0.543,0.457);
  \draw[gray!70, line width=0.35pt, -{Latex[length=1.05mm,width=0.75mm]}] (0.460,0.240) -- (0.584,0.340);
  \draw[gray!62, line width=0.35pt, -{Latex[length=1.05mm,width=0.75mm]}] (0.584,0.340) -- (0.628,0.372);
  \draw[blue!40, line width=0.35pt, -{Latex[length=1.05mm,width=0.75mm]}] (0.628,0.372) -- (0.630,0.370);
  \draw[cyan!50!black, line width=0.35pt, -{Latex[length=1.05mm,width=0.75mm]}] (0.630,0.370) -- (0.631,0.369);
  \draw[teal!60!black, line width=0.35pt, -{Latex[length=1.05mm,width=0.75mm]}] (0.631,0.369) -- (0.631,0.369);
  \draw[green!50!black, line width=0.35pt, -{Latex[length=1.05mm,width=0.75mm]}] (0.631,0.369) -- (0.631,0.369);
  \draw[orange!70!black, line width=0.35pt, -{Latex[length=1.05mm,width=0.75mm]}] (0.631,0.369) -- (0.631,0.369);
  \draw[red!75!black, line width=0.35pt, -{Latex[length=1.05mm,width=0.75mm]}] (0.631,0.369) -- (0.631,0.369);
  \draw[gray!70, line width=0.35pt, -{Latex[length=1.05mm,width=0.75mm]}] (0.533,0.167) -- (0.667,0.254);
  \draw[gray!62, line width=0.35pt, -{Latex[length=1.05mm,width=0.75mm]}] (0.667,0.254) -- (0.717,0.283);
  \draw[blue!40, line width=0.35pt, -{Latex[length=1.05mm,width=0.75mm]}] (0.717,0.283) -- (0.722,0.278);
  \draw[cyan!50!black, line width=0.35pt, -{Latex[length=1.05mm,width=0.75mm]}] (0.722,0.278) -- (0.723,0.277);
  \draw[teal!60!black, line width=0.35pt, -{Latex[length=1.05mm,width=0.75mm]}] (0.723,0.277) -- (0.723,0.277);
  \draw[green!50!black, line width=0.35pt, -{Latex[length=1.05mm,width=0.75mm]}] (0.723,0.277) -- (0.723,0.277);
  \draw[orange!70!black, line width=0.35pt, -{Latex[length=1.05mm,width=0.75mm]}] (0.723,0.277) -- (0.723,0.277);
  \draw[red!75!black, line width=0.35pt, -{Latex[length=1.05mm,width=0.75mm]}] (0.723,0.277) -- (0.723,0.277);
  \draw[gray!70, line width=0.35pt, -{Latex[length=1.05mm,width=0.75mm]}] (0.607,0.093) -- (0.752,0.158);
  \draw[gray!62, line width=0.35pt, -{Latex[length=1.05mm,width=0.75mm]}] (0.752,0.158) -- (0.818,0.182);
  \draw[blue!40, line width=0.35pt, -{Latex[length=1.05mm,width=0.75mm]}] (0.818,0.182) -- (0.827,0.173);
  \draw[cyan!50!black, line width=0.35pt, -{Latex[length=1.05mm,width=0.75mm]}] (0.827,0.173) -- (0.828,0.172);
  \draw[teal!60!black, line width=0.35pt, -{Latex[length=1.05mm,width=0.75mm]}] (0.828,0.172) -- (0.829,0.171);
  \draw[green!50!black, line width=0.35pt, -{Latex[length=1.05mm,width=0.75mm]}] (0.829,0.171) -- (0.829,0.171);
  \draw[orange!70!black, line width=0.35pt, -{Latex[length=1.05mm,width=0.75mm]}] (0.829,0.171) -- (0.829,0.171);
  \draw[red!75!black, line width=0.35pt, -{Latex[length=1.05mm,width=0.75mm]}] (0.829,0.171) -- (0.829,0.171);
  \draw[gray!70, line width=0.35pt, -{Latex[length=1.05mm,width=0.75mm]}] (0.680,0.020) -- (0.840,0.024);
  \draw[gray!62, line width=0.35pt, -{Latex[length=1.05mm,width=0.75mm]}] (0.840,0.024) -- (0.984,0.016);
  \draw[blue!40, line width=0.35pt, -{Latex[length=1.05mm,width=0.75mm]}] (0.984,0.016) -- (1.000,0.000);
  \draw[cyan!50!black, line width=0.35pt, -{Latex[length=1.05mm,width=0.75mm]}] (1.000,0.000) -- (1.000,0.000);
  \draw[teal!60!black, line width=0.35pt, -{Latex[length=1.05mm,width=0.75mm]}] (1.000,0.000) -- (1.000,0.000);
  \draw[green!50!black, line width=0.35pt, -{Latex[length=1.05mm,width=0.75mm]}] (1.000,0.000) -- (1.000,0.000);
  \draw[orange!70!black, line width=0.35pt, -{Latex[length=1.05mm,width=0.75mm]}] (1.000,0.000) -- (1.000,0.000);
  \draw[red!75!black, line width=0.35pt, -{Latex[length=1.05mm,width=0.75mm]}] (1.000,0.000) -- (1.000,0.000);
  \filldraw[fill=darkblue, draw=white, line width=0.12pt] (0.02446,0.83993) circle (0.0068);
  \filldraw[fill=darkblue, draw=white, line width=0.12pt] (0.15846,0.75204) circle (0.0068);
  \filldraw[fill=darkblue, draw=white, line width=0.12pt] (0.25420,0.66666) circle (0.0068);
  \filldraw[fill=darkblue, draw=white, line width=0.12pt] (0.34043,0.58430) circle (0.0068);
  \filldraw[fill=darkblue, draw=white, line width=0.12pt] (0.42266,0.50341) circle (0.0068);
  \filldraw[fill=darkblue, draw=white, line width=0.12pt] (0.50341,0.42266) circle (0.0068);
  \filldraw[fill=darkblue, draw=white, line width=0.12pt] (0.58430,0.34043) circle (0.0068);
  \filldraw[fill=darkblue, draw=white, line width=0.12pt] (0.66666,0.25420) circle (0.0068);
  \filldraw[fill=darkblue, draw=white, line width=0.12pt] (0.75204,0.15846) circle (0.0068);
  \filldraw[fill=darkblue, draw=white, line width=0.12pt] (0.83993,0.02446) circle (0.0068);
  \filldraw[fill=darkblue, draw=white, line width=0.12pt] (0.01625,0.98375) circle (0.0068);
  \filldraw[fill=darkblue, draw=white, line width=0.12pt] (0.18239,0.81761) circle (0.0068);
  \filldraw[fill=darkblue, draw=white, line width=0.12pt] (0.28273,0.71727) circle (0.0068);
  \filldraw[fill=darkblue, draw=white, line width=0.12pt] (0.37225,0.62775) circle (0.0068);
  \filldraw[fill=darkblue, draw=white, line width=0.12pt] (0.45780,0.54221) circle (0.0068);
  \filldraw[fill=darkblue, draw=white, line width=0.12pt] (0.54221,0.45780) circle (0.0068);
  \filldraw[fill=darkblue, draw=white, line width=0.12pt] (0.62775,0.37225) circle (0.0068);
  \filldraw[fill=darkblue, draw=white, line width=0.12pt] (0.71727,0.28273) circle (0.0068);
  \filldraw[fill=darkblue, draw=white, line width=0.12pt] (0.81761,0.18239) circle (0.0068);
  \filldraw[fill=darkblue, draw=white, line width=0.12pt] (0.98375,0.01625) circle (0.0068);
  \filldraw[fill=darkblue, draw=white, line width=0.12pt] (0.00000,1.00000) circle (0.0068);
  \filldraw[fill=darkblue, draw=white, line width=0.12pt] (0.17329,0.82671) circle (0.0068);
  \filldraw[fill=darkblue, draw=white, line width=0.12pt] (0.27834,0.72166) circle (0.0068);
  \filldraw[fill=darkblue, draw=white, line width=0.12pt] (0.37000,0.63000) circle (0.0068);
  \filldraw[fill=darkblue, draw=white, line width=0.12pt] (0.45709,0.54291) circle (0.0068);
  \filldraw[fill=darkblue, draw=white, line width=0.12pt] (0.54291,0.45709) circle (0.0068);
  \filldraw[fill=darkblue, draw=white, line width=0.12pt] (0.63000,0.37000) circle (0.0068);
  \filldraw[fill=darkblue, draw=white, line width=0.12pt] (0.72166,0.27834) circle (0.0068);
  \filldraw[fill=darkblue, draw=white, line width=0.12pt] (0.82671,0.17329) circle (0.0068);
  \filldraw[fill=darkblue, draw=white, line width=0.12pt] (1.00000,0.00000) circle (0.0068);
  \filldraw[fill=darkblue, draw=white, line width=0.12pt] (0.00000,1.00000) circle (0.0068);
  \filldraw[fill=darkblue, draw=white, line width=0.12pt] (0.17154,0.82846) circle (0.0068);
  \filldraw[fill=darkblue, draw=white, line width=0.12pt] (0.27713,0.72287) circle (0.0068);
  \filldraw[fill=darkblue, draw=white, line width=0.12pt] (0.36928,0.63072) circle (0.0068);
  \filldraw[fill=darkblue, draw=white, line width=0.12pt] (0.45685,0.54315) circle (0.0068);
  \filldraw[fill=darkblue, draw=white, line width=0.12pt] (0.54315,0.45685) circle (0.0068);
  \filldraw[fill=darkblue, draw=white, line width=0.12pt] (0.63072,0.36928) circle (0.0068);
  \filldraw[fill=darkblue, draw=white, line width=0.12pt] (0.72287,0.27713) circle (0.0068);
  \filldraw[fill=darkblue, draw=white, line width=0.12pt] (0.82846,0.17154) circle (0.0068);
  \filldraw[fill=darkblue, draw=white, line width=0.12pt] (1.00000,0.00000) circle (0.0068);
  \filldraw[fill=darkblue, draw=white, line width=0.12pt] (0.00000,1.00000) circle (0.0068);
  \filldraw[fill=darkblue, draw=white, line width=0.12pt] (0.17120,0.82880) circle (0.0068);
  \filldraw[fill=darkblue, draw=white, line width=0.12pt] (0.27688,0.72311) circle (0.0068);
  \filldraw[fill=darkblue, draw=white, line width=0.12pt] (0.36913,0.63087) circle (0.0068);
  \filldraw[fill=darkblue, draw=white, line width=0.12pt] (0.45680,0.54320) circle (0.0068);
  \filldraw[fill=darkblue, draw=white, line width=0.12pt] (0.54320,0.45680) circle (0.0068);
  \filldraw[fill=darkblue, draw=white, line width=0.12pt] (0.63087,0.36913) circle (0.0068);
  \filldraw[fill=darkblue, draw=white, line width=0.12pt] (0.72311,0.27688) circle (0.0068);
  \filldraw[fill=darkblue, draw=white, line width=0.12pt] (0.82880,0.17120) circle (0.0068);
  \filldraw[fill=darkblue, draw=white, line width=0.12pt] (1.00000,0.00000) circle (0.0068);
  \filldraw[fill=darkblue, draw=white, line width=0.12pt] (0.00000,1.00000) circle (0.0068);
  \filldraw[fill=darkblue, draw=white, line width=0.12pt] (0.17113,0.82887) circle (0.0068);
  \filldraw[fill=darkblue, draw=white, line width=0.12pt] (0.27684,0.72316) circle (0.0068);
  \filldraw[fill=darkblue, draw=white, line width=0.12pt] (0.36911,0.63089) circle (0.0068);
  \filldraw[fill=darkblue, draw=white, line width=0.12pt] (0.45679,0.54321) circle (0.0068);
  \filldraw[fill=darkblue, draw=white, line width=0.12pt] (0.54321,0.45679) circle (0.0068);
  \filldraw[fill=darkblue, draw=white, line width=0.12pt] (0.63089,0.36911) circle (0.0068);
  \filldraw[fill=darkblue, draw=white, line width=0.12pt] (0.72316,0.27684) circle (0.0068);
  \filldraw[fill=darkblue, draw=white, line width=0.12pt] (0.82887,0.17113) circle (0.0068);
  \filldraw[fill=darkblue, draw=white, line width=0.12pt] (1.00000,0.00000) circle (0.0068);
  \filldraw[fill=darkblue, draw=white, line width=0.12pt] (0.00000,1.00000) circle (0.0068);
  \filldraw[fill=darkblue, draw=white, line width=0.12pt] (0.17111,0.82889) circle (0.0068);
  \filldraw[fill=darkblue, draw=white, line width=0.12pt] (0.27683,0.72317) circle (0.0068);
  \filldraw[fill=darkblue, draw=white, line width=0.12pt] (0.36910,0.63090) circle (0.0068);
  \filldraw[fill=darkblue, draw=white, line width=0.12pt] (0.45679,0.54321) circle (0.0068);
  \filldraw[fill=darkblue, draw=white, line width=0.12pt] (0.54321,0.45679) circle (0.0068);
  \filldraw[fill=darkblue, draw=white, line width=0.12pt] (0.63090,0.36910) circle (0.0068);
  \filldraw[fill=darkblue, draw=white, line width=0.12pt] (0.72317,0.27683) circle (0.0068);
  \filldraw[fill=darkblue, draw=white, line width=0.12pt] (0.82889,0.17111) circle (0.0068);
  \filldraw[fill=darkblue, draw=white, line width=0.12pt] (1.00000,0.00000) circle (0.0068);
  \filldraw[fill=darkblue, draw=white, line width=0.12pt] (0.020,0.680) circle (0.0075);
  \filldraw[fill=darkblue, draw=white, line width=0.12pt] (0.093,0.607) circle (0.0075);
  \filldraw[fill=darkblue, draw=white, line width=0.12pt] (0.167,0.533) circle (0.0075);
  \filldraw[fill=darkblue, draw=white, line width=0.12pt] (0.240,0.460) circle (0.0075);
  \filldraw[fill=darkblue, draw=white, line width=0.12pt] (0.313,0.387) circle (0.0075);
  \filldraw[fill=darkblue, draw=white, line width=0.12pt] (0.387,0.313) circle (0.0075);
  \filldraw[fill=darkblue, draw=white, line width=0.12pt] (0.460,0.240) circle (0.0075);
  \filldraw[fill=darkblue, draw=white, line width=0.12pt] (0.533,0.167) circle (0.0075);
  \filldraw[fill=darkblue, draw=white, line width=0.12pt] (0.607,0.093) circle (0.0075);
  \filldraw[fill=darkblue, draw=white, line width=0.12pt] (0.680,0.020) circle (0.0075);
  \fill[red!80!black] (0.000,1.000) circle (0.0095);
  \fill[red!80!black] (0.171,0.829) circle (0.0095);
  \fill[red!80!black] (0.277,0.723) circle (0.0095);
  \fill[red!80!black] (0.369,0.631) circle (0.0095);
  \fill[red!80!black] (0.457,0.543) circle (0.0095);
  \fill[red!80!black] (0.543,0.457) circle (0.0095);
  \fill[red!80!black] (0.631,0.369) circle (0.0095);
  \fill[red!80!black] (0.723,0.277) circle (0.0095);
  \fill[red!80!black] (0.829,0.171) circle (0.0095);
  \fill[red!80!black] (1.000,0.000) circle (0.0095);
  \node[below] at (0.5,-0.07) {$F_1$};
  \node[rotate=90] at (-0.09,0.5) {$F_2$};
\end{tikzpicture}
\par\smallskip{\scriptsize (a) Nondominated line start: ten interior points on $F_1+F_2=0.7$.}
\end{minipage}\hfill
\begin{minipage}[t]{0.48\linewidth}
\centering
\vspace{0pt}
\begin{tikzpicture}[scale=7.2, every node/.style={font=\scriptsize}]
  \draw[very thin,gray!30] (0,0) grid[step=0.1] (1,1);
  \fill[blue!4] (0,0) -- (1,0) -- (0,1) -- cycle;
  \draw[thick] (0,0) -- (1,0) -- (0,1) -- cycle;
  \draw[ultra thick,blue!70!black] (1,0) -- (0,1);
  \foreach \x/\lab in {0/0,1/1} {\draw[thin,gray!65] (\x,0) -- +(0,-0.014); \node[below,font=\scriptsize] at (\x,-0.022) {\lab};}
  \foreach \y/\lab in {0/0,1/1} {\draw[thin,gray!65] (0,\y) -- +(-0.014,0); \node[left,font=\scriptsize] at (-0.022,\y) {\lab};}
  \node[above] at (0.5,1.03) {$\mu=10$: triangular start};
  \draw[gray!70, line width=0.35pt, -{Latex[length=1.05mm,width=0.75mm]}] (0.03000,0.27000) -- (0.03846,0.50985);
  \draw[gray!62, line width=0.35pt, -{Latex[length=1.05mm,width=0.75mm]}] (0.03846,0.50985) -- (0.05065,0.74954);
  \draw[blue!40, line width=0.35pt, -{Latex[length=1.05mm,width=0.75mm]}] (0.05065,0.74954) -- (0.03118,0.96882);
  \draw[cyan!50!black, line width=0.35pt, -{Latex[length=1.05mm,width=0.75mm]}] (0.03118,0.96882) -- (0.00000,1.00000);
  \draw[teal!60!black, line width=0.35pt, -{Latex[length=1.05mm,width=0.75mm]}] (0.00000,1.00000) -- (0.00000,1.00000);
  \draw[green!50!black, line width=0.35pt, -{Latex[length=1.05mm,width=0.75mm]}] (0.00000,1.00000) -- (0.00000,1.00000);
  \draw[orange!70!black, line width=0.35pt, -{Latex[length=1.05mm,width=0.75mm]}] (0.00000,1.00000) -- (0.00000,1.00000);
  \draw[red!75!black, line width=0.35pt, -{Latex[length=1.05mm,width=0.75mm]}] (0.00000,1.00000) -- (0.00000,1.00000);
  \draw[gray!70, line width=0.35pt, -{Latex[length=1.05mm,width=0.75mm]}] (0.08000,0.22000) -- (0.21545,0.41768);
  \draw[gray!62, line width=0.35pt, -{Latex[length=1.05mm,width=0.75mm]}] (0.21545,0.41768) -- (0.34953,0.61666);
  \draw[blue!40, line width=0.35pt, -{Latex[length=1.05mm,width=0.75mm]}] (0.34953,0.61666) -- (0.33292,0.66708);
  \draw[cyan!50!black, line width=0.35pt, -{Latex[length=1.05mm,width=0.75mm]}] (0.33292,0.66708) -- (0.31528,0.68472);
  \draw[teal!60!black, line width=0.35pt, -{Latex[length=1.05mm,width=0.75mm]}] (0.31528,0.68472) -- (0.31289,0.68711);
  \draw[green!50!black, line width=0.35pt, -{Latex[length=1.05mm,width=0.75mm]}] (0.31289,0.68711) -- (0.31488,0.68512);
  \draw[orange!70!black, line width=0.35pt, -{Latex[length=1.05mm,width=0.75mm]}] (0.31488,0.68512) -- (0.31741,0.68259);
  \draw[red!75!black, line width=0.35pt, -{Latex[length=1.05mm,width=0.75mm]}] (0.31741,0.68259) -- (0.31981,0.68019);
  \draw[gray!70, line width=0.35pt, -{Latex[length=1.05mm,width=0.75mm]}] (0.13000,0.17000) -- (0.32775,0.30536);
  \draw[gray!62, line width=0.35pt, -{Latex[length=1.05mm,width=0.75mm]}] (0.32775,0.30536) -- (0.53274,0.42978);
  \draw[blue!40, line width=0.35pt, -{Latex[length=1.05mm,width=0.75mm]}] (0.53274,0.42978) -- (0.62932,0.37068);
  \draw[cyan!50!black, line width=0.35pt, -{Latex[length=1.05mm,width=0.75mm]}] (0.62932,0.37068) -- (0.62176,0.37824);
  \draw[teal!60!black, line width=0.35pt, -{Latex[length=1.05mm,width=0.75mm]}] (0.62176,0.37824) -- (0.63799,0.36201);
  \draw[green!50!black, line width=0.35pt, -{Latex[length=1.05mm,width=0.75mm]}] (0.63799,0.36201) -- (0.64450,0.35550);
  \draw[orange!70!black, line width=0.35pt, -{Latex[length=1.05mm,width=0.75mm]}] (0.64450,0.35550) -- (0.64883,0.35117);
  \draw[red!75!black, line width=0.35pt, -{Latex[length=1.05mm,width=0.75mm]}] (0.64883,0.35117) -- (0.65212,0.34788);
  \draw[gray!70, line width=0.35pt, -{Latex[length=1.05mm,width=0.75mm]}] (0.18000,0.12000) -- (0.41986,0.12821);
  \draw[gray!62, line width=0.35pt, -{Latex[length=1.05mm,width=0.75mm]}] (0.41986,0.12821) -- (0.65954,0.14052);
  \draw[blue!40, line width=0.35pt, -{Latex[length=1.05mm,width=0.75mm]}] (0.65954,0.14052) -- (0.87635,0.12365);
  \draw[cyan!50!black, line width=0.35pt, -{Latex[length=1.05mm,width=0.75mm]}] (0.87635,0.12365) -- (1.00000,0.00000);
  \draw[teal!60!black, line width=0.35pt, -{Latex[length=1.05mm,width=0.75mm]}] (1.00000,0.00000) -- (1.00000,0.00000);
  \draw[green!50!black, line width=0.35pt, -{Latex[length=1.05mm,width=0.75mm]}] (1.00000,0.00000) -- (1.00000,0.00000);
  \draw[orange!70!black, line width=0.35pt, -{Latex[length=1.05mm,width=0.75mm]}] (1.00000,0.00000) -- (1.00000,0.00000);
  \draw[red!75!black, line width=0.35pt, -{Latex[length=1.05mm,width=0.75mm]}] (1.00000,0.00000) -- (1.00000,0.00000);
  \draw[gray!70, line width=0.35pt, -{Latex[length=1.05mm,width=0.75mm]}] (0.05000,0.19000) -- (0.03981,0.40861);
  \draw[gray!62, line width=0.35pt, -{Latex[length=1.05mm,width=0.75mm]}] (0.03981,0.40861) -- (0.10350,0.63694);
  \draw[blue!40, line width=0.35pt, -{Latex[length=1.05mm,width=0.75mm]}] (0.10350,0.63694) -- (0.19936,0.80064);
  \draw[cyan!50!black, line width=0.35pt, -{Latex[length=1.05mm,width=0.75mm]}] (0.19936,0.80064) -- (0.20914,0.79086);
  \draw[teal!60!black, line width=0.35pt, -{Latex[length=1.05mm,width=0.75mm]}] (0.20914,0.79086) -- (0.20883,0.79117);
  \draw[green!50!black, line width=0.35pt, -{Latex[length=1.05mm,width=0.75mm]}] (0.20883,0.79117) -- (0.20935,0.79066);
  \draw[orange!70!black, line width=0.35pt, -{Latex[length=1.05mm,width=0.75mm]}] (0.20935,0.79066) -- (0.21078,0.78922);
  \draw[red!75!black, line width=0.35pt, -{Latex[length=1.05mm,width=0.75mm]}] (0.21078,0.78922) -- (0.21238,0.78762);
  \draw[gray!70, line width=0.35pt, -{Latex[length=1.05mm,width=0.75mm]}] (0.10000,0.14000) -- (0.22905,0.27261);
  \draw[gray!62, line width=0.35pt, -{Latex[length=1.05mm,width=0.75mm]}] (0.22905,0.27261) -- (0.38229,0.44944);
  \draw[blue!40, line width=0.35pt, -{Latex[length=1.05mm,width=0.75mm]}] (0.38229,0.44944) -- (0.47008,0.52992);
  \draw[cyan!50!black, line width=0.35pt, -{Latex[length=1.05mm,width=0.75mm]}] (0.47008,0.52992) -- (0.51373,0.48627);
  \draw[teal!60!black, line width=0.35pt, -{Latex[length=1.05mm,width=0.75mm]}] (0.51373,0.48627) -- (0.52521,0.47479);
  \draw[green!50!black, line width=0.35pt, -{Latex[length=1.05mm,width=0.75mm]}] (0.52521,0.47479) -- (0.53158,0.46842);
  \draw[orange!70!black, line width=0.35pt, -{Latex[length=1.05mm,width=0.75mm]}] (0.53158,0.46842) -- (0.53597,0.46403);
  \draw[red!75!black, line width=0.35pt, -{Latex[length=1.05mm,width=0.75mm]}] (0.53597,0.46403) -- (0.53940,0.46060);
  \draw[gray!70, line width=0.35pt, -{Latex[length=1.05mm,width=0.75mm]}] (0.15000,0.09000) -- (0.36773,0.06907);
  \draw[gray!62, line width=0.35pt, -{Latex[length=1.05mm,width=0.75mm]}] (0.36773,0.06907) -- (0.60714,0.08244);
  \draw[blue!40, line width=0.35pt, -{Latex[length=1.05mm,width=0.75mm]}] (0.60714,0.08244) -- (0.83857,0.07541);
  \draw[cyan!50!black, line width=0.35pt, -{Latex[length=1.05mm,width=0.75mm]}] (0.83857,0.07541) -- (0.89555,0.10445);
  \draw[teal!60!black, line width=0.35pt, -{Latex[length=1.05mm,width=0.75mm]}] (0.89555,0.10445) -- (0.87849,0.12151);
  \draw[green!50!black, line width=0.35pt, -{Latex[length=1.05mm,width=0.75mm]}] (0.87849,0.12151) -- (0.87974,0.12026);
  \draw[orange!70!black, line width=0.35pt, -{Latex[length=1.05mm,width=0.75mm]}] (0.87974,0.12026) -- (0.88153,0.11847);
  \draw[red!75!black, line width=0.35pt, -{Latex[length=1.05mm,width=0.75mm]}] (0.88153,0.11847) -- (0.88296,0.11704);
  \draw[gray!70, line width=0.35pt, -{Latex[length=1.05mm,width=0.75mm]}] (0.07000,0.11000) -- (0.03270,0.26254);
  \draw[gray!62, line width=0.35pt, -{Latex[length=1.05mm,width=0.75mm]}] (0.03270,0.26254) -- (0.04360,0.50211);
  \draw[blue!40, line width=0.35pt, -{Latex[length=1.05mm,width=0.75mm]}] (0.04360,0.50211) -- (0.07233,0.74037);
  \draw[cyan!50!black, line width=0.35pt, -{Latex[length=1.05mm,width=0.75mm]}] (0.07233,0.74037) -- (0.10575,0.89425);
  \draw[teal!60!black, line width=0.35pt, -{Latex[length=1.05mm,width=0.75mm]}] (0.10575,0.89425) -- (0.10456,0.89544);
  \draw[green!50!black, line width=0.35pt, -{Latex[length=1.05mm,width=0.75mm]}] (0.10456,0.89544) -- (0.10455,0.89546);
  \draw[orange!70!black, line width=0.35pt, -{Latex[length=1.05mm,width=0.75mm]}] (0.10455,0.89546) -- (0.10516,0.89484);
  \draw[red!75!black, line width=0.35pt, -{Latex[length=1.05mm,width=0.75mm]}] (0.10516,0.89484) -- (0.10594,0.89406);
  \draw[gray!70, line width=0.35pt, -{Latex[length=1.05mm,width=0.75mm]}] (0.12000,0.06000) -- (0.28843,0.00000);
  \draw[gray!62, line width=0.35pt, -{Latex[length=1.05mm,width=0.75mm]}] (0.28843,0.00000) -- (0.51338,0.03064);
  \draw[blue!40, line width=0.35pt, -{Latex[length=1.05mm,width=0.75mm]}] (0.51338,0.03064) -- (0.60697,0.24533);
  \draw[cyan!50!black, line width=0.35pt, -{Latex[length=1.05mm,width=0.75mm]}] (0.60697,0.24533) -- (0.74708,0.25292);
  \draw[teal!60!black, line width=0.35pt, -{Latex[length=1.05mm,width=0.75mm]}] (0.74708,0.25292) -- (0.75669,0.24331);
  \draw[green!50!black, line width=0.35pt, -{Latex[length=1.05mm,width=0.75mm]}] (0.75669,0.24331) -- (0.76082,0.23918);
  \draw[orange!70!black, line width=0.35pt, -{Latex[length=1.05mm,width=0.75mm]}] (0.76082,0.23918) -- (0.76418,0.23582);
  \draw[red!75!black, line width=0.35pt, -{Latex[length=1.05mm,width=0.75mm]}] (0.76418,0.23582) -- (0.76677,0.23323);
  \draw[gray!70, line width=0.35pt, -{Latex[length=1.05mm,width=0.75mm]}] (0.09000,0.03000) -- (0.11467,0.08694);
  \draw[gray!62, line width=0.35pt, -{Latex[length=1.05mm,width=0.75mm]}] (0.11467,0.08694) -- (0.20067,0.29212);
  \draw[blue!40, line width=0.35pt, -{Latex[length=1.05mm,width=0.75mm]}] (0.20067,0.29212) -- (0.37025,0.46034);
  \draw[cyan!50!black, line width=0.35pt, -{Latex[length=1.05mm,width=0.75mm]}] (0.37025,0.46034) -- (0.41566,0.58434);
  \draw[teal!60!black, line width=0.35pt, -{Latex[length=1.05mm,width=0.75mm]}] (0.41566,0.58434) -- (0.41767,0.58233);
  \draw[green!50!black, line width=0.35pt, -{Latex[length=1.05mm,width=0.75mm]}] (0.41767,0.58233) -- (0.42196,0.57804);
  \draw[orange!70!black, line width=0.35pt, -{Latex[length=1.05mm,width=0.75mm]}] (0.42196,0.57804) -- (0.42561,0.57439);
  \draw[red!75!black, line width=0.35pt, -{Latex[length=1.05mm,width=0.75mm]}] (0.42561,0.57439) -- (0.42869,0.57131);
  \filldraw[fill=darkblue, draw=white, line width=0.12pt] (0.03846,0.50985) circle (0.0068);
  \filldraw[fill=darkblue, draw=white, line width=0.12pt] (0.21545,0.41768) circle (0.0068);
  \filldraw[fill=darkblue, draw=white, line width=0.12pt] (0.32775,0.30536) circle (0.0068);
  \filldraw[fill=darkblue, draw=white, line width=0.12pt] (0.41986,0.12821) circle (0.0068);
  \filldraw[fill=darkgreen, draw=white, line width=0.12pt] (0.03981,0.40861) circle (0.0068);
  \filldraw[fill=darkgreen, draw=white, line width=0.12pt] (0.22905,0.27261) circle (0.0068);
  \filldraw[fill=darkgreen, draw=white, line width=0.12pt] (0.36773,0.06907) circle (0.0068);
  \filldraw[fill=darkorange, draw=white, line width=0.12pt] (0.03270,0.26254) circle (0.0068);
  \filldraw[fill=darkorange, draw=white, line width=0.12pt] (0.28843,0.00000) circle (0.0068);
  \filldraw[fill=darkorange, draw=white, line width=0.12pt] (0.11467,0.08694) circle (0.0068);
  \filldraw[fill=darkblue, draw=white, line width=0.12pt] (0.05065,0.74954) circle (0.0068);
  \filldraw[fill=darkblue, draw=white, line width=0.12pt] (0.34953,0.61666) circle (0.0068);
  \filldraw[fill=darkblue, draw=white, line width=0.12pt] (0.53274,0.42978) circle (0.0068);
  \filldraw[fill=darkblue, draw=white, line width=0.12pt] (0.65954,0.14052) circle (0.0068);
  \filldraw[fill=darkblue, draw=white, line width=0.12pt] (0.10350,0.63694) circle (0.0068);
  \filldraw[fill=darkblue, draw=white, line width=0.12pt] (0.38229,0.44944) circle (0.0068);
  \filldraw[fill=darkgreen, draw=white, line width=0.12pt] (0.60714,0.08244) circle (0.0068);
  \filldraw[fill=darkgreen, draw=white, line width=0.12pt] (0.04360,0.50211) circle (0.0068);
  \filldraw[fill=darkorange, draw=white, line width=0.12pt] (0.51338,0.03064) circle (0.0068);
  \filldraw[fill=darkgreen, draw=white, line width=0.12pt] (0.20067,0.29213) circle (0.0068);
  \filldraw[fill=darkblue, draw=white, line width=0.12pt] (0.03118,0.96882) circle (0.0068);
  \filldraw[fill=darkblue, draw=white, line width=0.12pt] (0.33292,0.66708) circle (0.0068);
  \filldraw[fill=darkblue, draw=white, line width=0.12pt] (0.62932,0.37068) circle (0.0068);
  \filldraw[fill=darkblue, draw=white, line width=0.12pt] (0.87635,0.12365) circle (0.0068);
  \filldraw[fill=darkblue, draw=white, line width=0.12pt] (0.19936,0.80064) circle (0.0068);
  \filldraw[fill=darkblue, draw=white, line width=0.12pt] (0.47008,0.52992) circle (0.0068);
  \filldraw[fill=darkgreen, draw=white, line width=0.12pt] (0.83857,0.07541) circle (0.0068);
  \filldraw[fill=darkgreen, draw=white, line width=0.12pt] (0.07233,0.74037) circle (0.0068);
  \filldraw[fill=darkgreen, draw=white, line width=0.12pt] (0.60698,0.24533) circle (0.0068);
  \filldraw[fill=darkgreen, draw=white, line width=0.12pt] (0.37025,0.46034) circle (0.0068);
  \filldraw[fill=darkblue, draw=white, line width=0.12pt] (0.00000,1.00000) circle (0.0068);
  \filldraw[fill=darkblue, draw=white, line width=0.12pt] (0.31528,0.68472) circle (0.0068);
  \filldraw[fill=darkblue, draw=white, line width=0.12pt] (0.62176,0.37824) circle (0.0068);
  \filldraw[fill=darkblue, draw=white, line width=0.12pt] (1.00000,0.00000) circle (0.0068);
  \filldraw[fill=darkblue, draw=white, line width=0.12pt] (0.20914,0.79086) circle (0.0068);
  \filldraw[fill=darkblue, draw=white, line width=0.12pt] (0.51374,0.48627) circle (0.0068);
  \filldraw[fill=darkblue, draw=white, line width=0.12pt] (0.89555,0.10445) circle (0.0068);
  \filldraw[fill=darkblue, draw=white, line width=0.12pt] (0.10575,0.89425) circle (0.0068);
  \filldraw[fill=darkblue, draw=white, line width=0.12pt] (0.74708,0.25292) circle (0.0068);
  \filldraw[fill=darkblue, draw=white, line width=0.12pt] (0.41566,0.58434) circle (0.0068);
  \filldraw[fill=darkblue, draw=white, line width=0.12pt] (0.00000,1.00000) circle (0.0068);
  \filldraw[fill=darkblue, draw=white, line width=0.12pt] (0.31289,0.68711) circle (0.0068);
  \filldraw[fill=darkblue, draw=white, line width=0.12pt] (0.63799,0.36201) circle (0.0068);
  \filldraw[fill=darkblue, draw=white, line width=0.12pt] (1.00000,0.00000) circle (0.0068);
  \filldraw[fill=darkblue, draw=white, line width=0.12pt] (0.20883,0.79117) circle (0.0068);
  \filldraw[fill=darkblue, draw=white, line width=0.12pt] (0.52521,0.47479) circle (0.0068);
  \filldraw[fill=darkblue, draw=white, line width=0.12pt] (0.87849,0.12151) circle (0.0068);
  \filldraw[fill=darkblue, draw=white, line width=0.12pt] (0.10456,0.89544) circle (0.0068);
  \filldraw[fill=darkblue, draw=white, line width=0.12pt] (0.75669,0.24331) circle (0.0068);
  \filldraw[fill=darkblue, draw=white, line width=0.12pt] (0.41767,0.58233) circle (0.0068);
  \filldraw[fill=darkblue, draw=white, line width=0.12pt] (0.00000,1.00000) circle (0.0068);
  \filldraw[fill=darkblue, draw=white, line width=0.12pt] (0.31488,0.68512) circle (0.0068);
  \filldraw[fill=darkblue, draw=white, line width=0.12pt] (0.64450,0.35550) circle (0.0068);
  \filldraw[fill=darkblue, draw=white, line width=0.12pt] (1.00000,0.00000) circle (0.0068);
  \filldraw[fill=darkblue, draw=white, line width=0.12pt] (0.20935,0.79065) circle (0.0068);
  \filldraw[fill=darkblue, draw=white, line width=0.12pt] (0.53158,0.46842) circle (0.0068);
  \filldraw[fill=darkblue, draw=white, line width=0.12pt] (0.87974,0.12026) circle (0.0068);
  \filldraw[fill=darkblue, draw=white, line width=0.12pt] (0.10454,0.89546) circle (0.0068);
  \filldraw[fill=darkblue, draw=white, line width=0.12pt] (0.76082,0.23918) circle (0.0068);
  \filldraw[fill=darkblue, draw=white, line width=0.12pt] (0.42196,0.57804) circle (0.0068);
  \filldraw[fill=darkblue, draw=white, line width=0.12pt] (0.00000,1.00000) circle (0.0068);
  \filldraw[fill=darkblue, draw=white, line width=0.12pt] (0.31741,0.68259) circle (0.0068);
  \filldraw[fill=darkblue, draw=white, line width=0.12pt] (0.64883,0.35117) circle (0.0068);
  \filldraw[fill=darkblue, draw=white, line width=0.12pt] (1.00000,0.00000) circle (0.0068);
  \filldraw[fill=darkblue, draw=white, line width=0.12pt] (0.21078,0.78922) circle (0.0068);
  \filldraw[fill=darkblue, draw=white, line width=0.12pt] (0.53597,0.46403) circle (0.0068);
  \filldraw[fill=darkblue, draw=white, line width=0.12pt] (0.88153,0.11847) circle (0.0068);
  \filldraw[fill=darkblue, draw=white, line width=0.12pt] (0.10516,0.89484) circle (0.0068);
  \filldraw[fill=darkblue, draw=white, line width=0.12pt] (0.76418,0.23583) circle (0.0068);
  \filldraw[fill=darkblue, draw=white, line width=0.12pt] (0.42561,0.57439) circle (0.0068);
  \filldraw[fill=darkblue, draw=white, line width=0.12pt] (0.03000,0.27000) circle (0.0075);
  \filldraw[fill=darkblue, draw=white, line width=0.12pt] (0.08000,0.22000) circle (0.0075);
  \filldraw[fill=darkblue, draw=white, line width=0.12pt] (0.13000,0.17000) circle (0.0075);
  \filldraw[fill=darkblue, draw=white, line width=0.12pt] (0.18000,0.12000) circle (0.0075);
  \filldraw[fill=darkgreen, draw=white, line width=0.12pt] (0.05000,0.19000) circle (0.0075);
  \filldraw[fill=darkgreen, draw=white, line width=0.12pt] (0.10000,0.14000) circle (0.0075);
  \filldraw[fill=darkgreen, draw=white, line width=0.12pt] (0.15000,0.09000) circle (0.0075);
  \filldraw[fill=darkorange, draw=white, line width=0.12pt] (0.07000,0.11000) circle (0.0075);
  \filldraw[fill=darkorange, draw=white, line width=0.12pt] (0.12000,0.06000) circle (0.0075);
  \filldraw[fill=darkorange, draw=white, line width=0.12pt] (0.09000,0.03000) circle (0.0075);
  \fill[red!80!black] (0.00000,1.00000) circle (0.0095);
  \fill[red!80!black] (0.31981,0.68019) circle (0.0095);
  \fill[red!80!black] (0.65212,0.34788) circle (0.0095);
  \fill[red!80!black] (1.00000,0.00000) circle (0.0095);
  \fill[red!80!black] (0.21238,0.78762) circle (0.0095);
  \fill[red!80!black] (0.53940,0.46060) circle (0.0095);
  \fill[red!80!black] (0.88296,0.11704) circle (0.0095);
  \fill[red!80!black] (0.10594,0.89406) circle (0.0095);
  \fill[red!80!black] (0.76677,0.23323) circle (0.0095);
  \fill[red!80!black] (0.42869,0.57131) circle (0.0095);
  \node[below] at (0.5,-0.07) {$F_1$};
  \node[rotate=90] at (-0.09,0.5) {$F_2$};
\end{tikzpicture}
\par\smallskip{\scriptsize (b) Dominated triangular start: $4+3+2+1$ points in the subtriangle with vertices $(0,0)$, $(0,0.5)$, and $(0.5,0)$.}
\end{minipage}
\caption{Ten-point triangle example with two different initializations. In both panels, the blue segment is the Pareto front; rank-colored markers and red points mark the sampled nonfinal iterates and last plotted iterates, respectively. Path colors progress from early to late segments, and arrow tips indicate increasing iteration. The marker colors encode the current dominance rank using the palette of Figure~\ref{fig:intro-layered-example}: blue for layer~1, green for layer~2, and orange for layer~3 or deeper.}
\label{fig:ten-point-paths}
\end{figure}
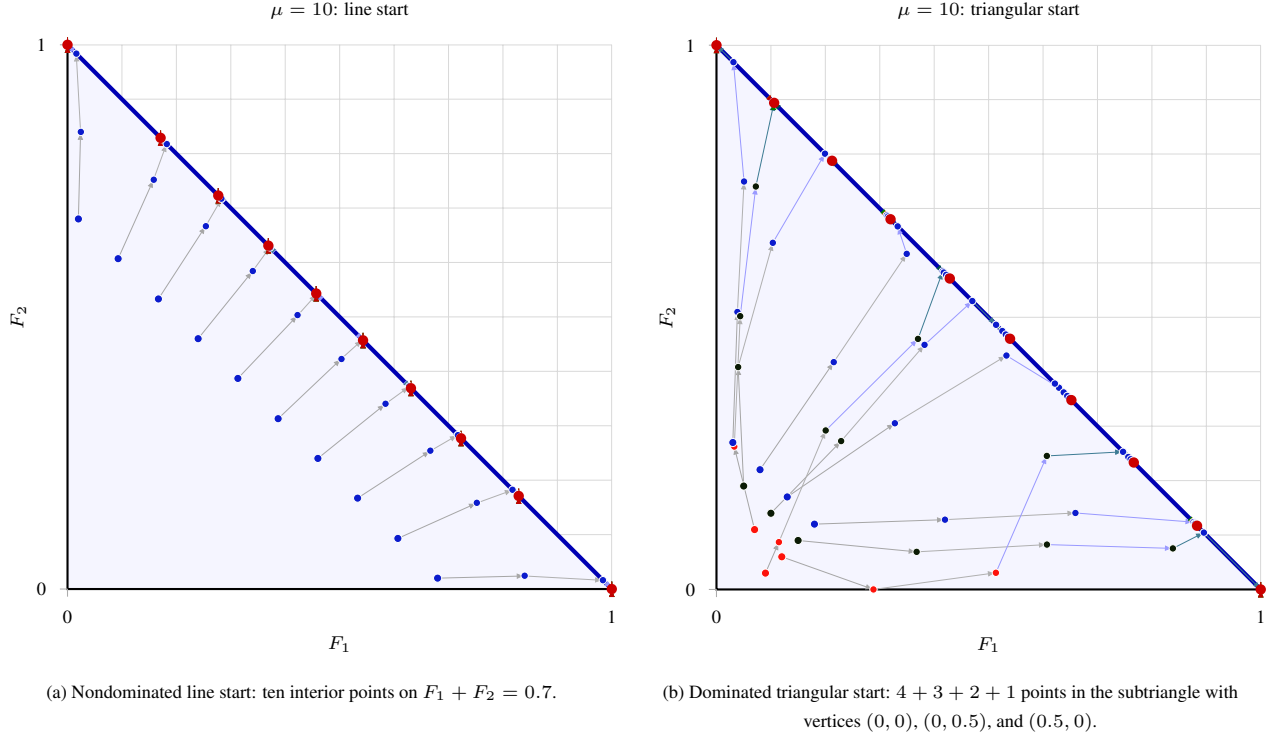

\subsection{A summed quadratic biobjective problem with ten points}

We next consider a smooth biobjective optimization problem in the decision space
\[
\Omega=[-2,2]^2.
\]
For $x=(x_1,x_2)\in\Omega$, define the component functions
\[
f_1(x)=\frac12\bigl(1-(x_1-1)^2\bigr)+\frac12\bigl(1-(x_2-1)^2\bigr),
\]
\[
f_2(x)=\frac12\bigl(1-x_1^2\bigr)+\frac12\bigl(1-x_2^2\bigr),
\]
and write the corresponding objective vector value as $F(x)=(F_1,F_2)=(f_1(x),f_2(x))$.
Thus each coordinate contributes a term favoring $x_i=1$ in the first objective and $x_i=0$ in the second. The problem is smooth in decision space, but its efficient set and Pareto front can still be described explicitly.

\begin{proposition}
For the above problem, the efficient set is
\[
\mathcal E=\{(t,t):\,t\in[0,1]\},
\]
and the Pareto front is the curve
\[
\mathcal F
=
\bigl\{(2t-t^2,\ 1-t^2):\, t\in[0,1]\bigr\}.
\]
\end{proposition}

\begin{proof}
If $x_i<0$, then increasing $x_i$ toward $0$ improves both
\[
\frac12(1-(x_i-1)^2)
\qquad\text{and}\qquad
\frac12(1-x_i^2),
\]
so no efficient point can have $x_i<0$. Likewise, if $x_i>1$, then decreasing $x_i$ toward $1$ improves both objectives, so no efficient point can have $x_i>1$. Hence every efficient point lies in $[0,1]^2$.

Now let $x=(x_1,x_2)\in[0,1]^2$, and set
\[
t:=\frac{x_1+x_2}{2}.
\]
Then
\[
f_2(x)=1-\frac{x_1^2+x_2^2}{2},
\qquad
f_1(x)=x_1+x_2-\frac{x_1^2+x_2^2}{2}.
\]
Since
\[
x_1^2+x_2^2 \ge 2t^2,
\]
with equality if and only if $x_1=x_2=t$, it follows that
\[
f_2(t,t)\ge f_2(x),
\qquad
f_1(t,t)\ge f_1(x),
\]
and the inequalities are strict unless $x_1=x_2$. Thus every off-diagonal point in $[0,1]^2$ is dominated by its diagonal average. Hence the efficient set is exactly the stated diagonal segment, and the stated front formula follows by substitution.
\end{proof}

This example is useful because it is smooth in decision space, has a one-dimensional Pareto set, and still makes the nonsmooth layered optimization visible through changes in the active dominated-set geometry. The two objective-space runs in Figure~\ref{fig:ten-point-paths} already illustrate that the method can start either from a fully nondominated interior set or from a dominated multilayer set. To make the nonlinear geometry in decision space more visible, we now choose a perturbed interior starting set.

\begin{example}[Perturbed ten-point gradient diffusion]
We initialize ten points in decision space by
\[
(0.10,0.74),\ (0.18,0.49),\ (0.12,0.61),\ (0.33,0.58),\ (0.46,0.28),
\]
\[
(0.41,0.45),\ (0.63,0.12),\ (0.57,0.26),\ (0.71,0.33),\ (0.82,0.08).
\]
All these points are feasible, mutually distinct, and lie away from the efficient diagonal. We then apply the implemented projected finite-difference ascent scheme in decision space: the layered objective is evaluated in objective space after mapping the decision points through $F$, symmetric finite differences are taken with respect to the decision coordinates, the resulting directions are normalized separately for each point, and small projected steps are taken in the box $[0,1]^2$. This tighter projection is sufficient here because Proposition~1 shows that all efficient points lie in $[0,1]^2$, so clipping the run to this box keeps the iterates in the relevant region while avoiding excursions that would only clutter the figure. In the implementation used for this example we employ the setting
\[
\quad \alpha=0.004,\quad \varepsilon=10^{-3},\qquad \tau=2\cdot 10^{-4},\quad \sigma=0.03,\quad h=10^{-5},\quad T=540,
\]
which is an example-specific override of the generic defaults listed in Appendix~\ref{app:algorithm}. Thus the numerical scheme should be read as a Clarke-motivated finite-difference diffusion rather than as a full projected subgradient method with convergence guarantees. Whenever the geometry is smooth, it reduces to an ordinary small-step ascent in decision space. At points where the active dominated-set geometry changes, the Clarke viewpoint remains the appropriate first-order interpretation.
\end{example}

\begin{figure}[t]
\centering
\begin{minipage}{0.48\textwidth}
\centering
\begin{tikzpicture}[scale=4.9, every node/.style={font=\scriptsize}]
  \draw[very thin,gray!25] (0.000,0.000) grid[step=0.1] (1.000,1.000);
  \draw[thick] (0.000,0.000) rectangle (1.000,1.000);
  \draw[ultra thick,blue!70!black] (0.000,0.000) -- (1.000,1.000);
  \foreach \x/\lab in {0/0,1/1} {\draw[thin,gray!65] (\x,0) -- +(0,-0.014); \node[below,font=\scriptsize] at (\x,-0.022) {\lab};}
  \foreach \y/\lab in {0/0,1/1} {\draw[thin,gray!65] (0,\y) -- +(-0.014,0); \node[left,font=\scriptsize] at (-0.022,\y) {\lab};}
  \node[above] at (0.5,1.03) {decision space};
  \draw[gray!70, line width=0.56pt, -{Latex[length=1.05mm,width=0.75mm]}] (0.100,0.740) -- (0.312,0.844);
  \draw[gray!62, line width=0.56pt, -{Latex[length=1.05mm,width=0.75mm]}] (0.312,0.844) -- (0.561,0.943);
  \draw[blue!40, line width=0.56pt, -{Latex[length=1.05mm,width=0.75mm]}] (0.561,0.943) -- (0.686,0.995);
  \draw[cyan!50!black, line width=0.56pt, -{Latex[length=1.05mm,width=0.75mm]}] (0.686,0.995) -- (0.744,0.958);
  \draw[teal!60!black, line width=0.56pt, -{Latex[length=1.05mm,width=0.75mm]}] (0.744,0.958) -- (0.860,0.860);
  \draw[green!50!black, line width=0.56pt] (0.860,0.860) -- (0.860,0.860);
  \draw[orange!70!black, line width=0.56pt, -{Latex[length=1.05mm,width=0.75mm]}] (0.860,0.860) -- (0.863,0.863);
  \draw[red!75!black, line width=0.56pt] (0.863,0.863) -- (0.863,0.863);
  \draw[gray!70, line width=0.56pt, -{Latex[length=1.05mm,width=0.75mm]}] (0.180,0.490) -- (0.093,0.237);
  \draw[gray!62, line width=0.56pt, -{Latex[length=1.05mm,width=0.75mm]}] (0.093,0.237) -- (0.011,0.011);
  \draw[blue!40, line width=0.56pt, -{Latex[length=1.05mm,width=0.75mm]}] (0.011,0.011) -- (0.014,0.014);
  \draw[cyan!50!black, line width=0.56pt] (0.014,0.014) -- (0.014,0.014);
  \draw[teal!60!black, line width=0.56pt, -{Latex[length=1.05mm,width=0.75mm]}] (0.014,0.014) -- (0.011,0.011);
  \draw[green!50!black, line width=0.56pt] (0.011,0.011) -- (0.011,0.011);
  \draw[orange!70!black, line width=0.56pt, -{Latex[length=1.05mm,width=0.75mm]}] (0.011,0.011) -- (0.014,0.014);
  \draw[red!75!black, line width=0.56pt] (0.014,0.014) -- (0.014,0.014);
  \draw[gray!70, line width=0.56pt, -{Latex[length=1.05mm,width=0.75mm]}] (0.120,0.610) -- (0.027,0.417);
  \draw[gray!62, line width=0.56pt, -{Latex[length=1.05mm,width=0.75mm]}] (0.027,0.417) -- (0.144,0.179);
  \draw[blue!40, line width=0.56pt, -{Latex[length=1.05mm,width=0.75mm]}] (0.144,0.179) -- (0.155,0.155);
  \draw[cyan!50!black, line width=0.56pt, -{Latex[length=1.05mm,width=0.75mm]}] (0.155,0.155) -- (0.161,0.161);
  \draw[teal!60!black, line width=0.56pt, -{Latex[length=1.05mm,width=0.75mm]}] (0.161,0.161) -- (0.163,0.163);
  \draw[green!50!black, line width=0.56pt] (0.163,0.163) -- (0.163,0.163);
  \draw[orange!70!black, line width=0.56pt, -{Latex[length=1.05mm,width=0.75mm]}] (0.163,0.163) -- (0.161,0.161);
  \draw[red!75!black, line width=0.56pt] (0.161,0.161) -- (0.161,0.161);
  \draw[gray!70, line width=0.56pt, -{Latex[length=1.05mm,width=0.75mm]}] (0.330,0.580) -- (0.586,0.586);
  \draw[gray!62, line width=0.56pt, -{Latex[length=1.05mm,width=0.75mm]}] (0.586,0.586) -- (0.733,0.733);
  \draw[blue!40, line width=0.56pt, -{Latex[length=1.05mm,width=0.75mm]}] (0.733,0.733) -- (0.821,0.821);
  \draw[cyan!50!black, line width=0.56pt] (0.821,0.821) -- (0.821,0.821);
  \draw[teal!60!black, line width=0.56pt, -{Latex[length=1.05mm,width=0.75mm]}] (0.821,0.821) -- (0.761,0.761);
  \draw[green!50!black, line width=0.56pt] (0.761,0.761) -- (0.761,0.761);
  \draw[orange!70!black, line width=0.56pt, -{Latex[length=1.05mm,width=0.75mm]}] (0.761,0.761) -- (0.758,0.758);
  \draw[red!75!black, line width=0.56pt] (0.758,0.758) -- (0.758,0.758);
  \draw[gray!70, line width=0.56pt, -{Latex[length=1.05mm,width=0.75mm]}] (0.460,0.280) -- (0.288,0.288);
  \draw[gray!62, line width=0.56pt, -{Latex[length=1.05mm,width=0.75mm]}] (0.288,0.288) -- (0.283,0.283);
  \draw[blue!40, line width=0.56pt, -{Latex[length=1.05mm,width=0.75mm]}] (0.283,0.283) -- (0.274,0.274);
  \draw[cyan!50!black, line width=0.56pt, -{Latex[length=1.05mm,width=0.75mm]}] (0.274,0.274) -- (0.280,0.280);
  \draw[teal!60!black, line width=0.56pt, -{Latex[length=1.05mm,width=0.75mm]}] (0.280,0.280) -- (0.277,0.277);
  \draw[green!50!black, line width=0.56pt] (0.277,0.277) -- (0.277,0.277);
  \draw[orange!70!black, line width=0.56pt, -{Latex[length=1.05mm,width=0.75mm]}] (0.277,0.277) -- (0.280,0.280);
  \draw[red!75!black, line width=0.56pt] (0.280,0.280) -- (0.280,0.280);
  \draw[gray!70, line width=0.56pt, -{Latex[length=1.05mm,width=0.75mm]}] (0.410,0.450) -- (0.396,0.396);
  \draw[gray!62, line width=0.56pt, -{Latex[length=1.05mm,width=0.75mm]}] (0.396,0.396) -- (0.503,0.503);
  \draw[blue!40, line width=0.56pt, -{Latex[length=1.05mm,width=0.75mm]}] (0.503,0.503) -- (0.591,0.591);
  \draw[cyan!50!black, line width=0.56pt] (0.591,0.591) -- (0.591,0.591);
  \draw[teal!60!black, line width=0.56pt, -{Latex[length=1.05mm,width=0.75mm]}] (0.591,0.591) -- (0.577,0.577);
  \draw[green!50!black, line width=0.56pt] (0.577,0.577) -- (0.577,0.577);
  \draw[orange!70!black, line width=0.56pt, -{Latex[length=1.05mm,width=0.75mm]}] (0.577,0.577) -- (0.574,0.574);
  \draw[red!75!black, line width=0.56pt] (0.574,0.574) -- (0.574,0.574);
  \draw[gray!70, line width=0.56pt, -{Latex[length=1.05mm,width=0.75mm]}] (0.630,0.120) -- (0.539,0.179);
  \draw[gray!62, line width=0.56pt, -{Latex[length=1.05mm,width=0.75mm]}] (0.539,0.179) -- (0.406,0.379);
  \draw[blue!40, line width=0.56pt, -{Latex[length=1.05mm,width=0.75mm]}] (0.406,0.379) -- (0.484,0.484);
  \draw[cyan!50!black, line width=0.56pt, -{Latex[length=1.05mm,width=0.75mm]}] (0.484,0.484) -- (0.489,0.489);
  \draw[teal!60!black, line width=0.56pt, -{Latex[length=1.05mm,width=0.75mm]}] (0.489,0.489) -- (0.481,0.481);
  \draw[green!50!black, line width=0.56pt] (0.481,0.481) -- (0.481,0.481);
  \draw[orange!70!black, line width=0.56pt, -{Latex[length=1.05mm,width=0.75mm]}] (0.481,0.481) -- (0.484,0.484);
  \draw[red!75!black, line width=0.56pt] (0.484,0.484) -- (0.484,0.484);
  \draw[gray!70, line width=0.56pt, -{Latex[length=1.05mm,width=0.75mm]}] (0.570,0.260) -- (0.595,0.383);
  \draw[gray!62, line width=0.56pt, -{Latex[length=1.05mm,width=0.75mm]}] (0.595,0.383) -- (0.615,0.615);
  \draw[blue!40, line width=0.56pt, -{Latex[length=1.05mm,width=0.75mm]}] (0.615,0.615) -- (0.703,0.703);
  \draw[cyan!50!black, line width=0.56pt] (0.703,0.703) -- (0.703,0.703);
  \draw[teal!60!black, line width=0.56pt, -{Latex[length=1.05mm,width=0.75mm]}] (0.703,0.703) -- (0.666,0.666);
  \draw[green!50!black, line width=0.56pt] (0.666,0.666) -- (0.666,0.666);
  \draw[orange!70!black, line width=0.56pt, -{Latex[length=1.05mm,width=0.75mm]}] (0.666,0.666) -- (0.669,0.669);
  \draw[red!75!black, line width=0.56pt] (0.669,0.669) -- (0.669,0.669);
  \draw[gray!70, line width=0.56pt, -{Latex[length=1.05mm,width=0.75mm]}] (0.710,0.330) -- (0.813,0.578);
  \draw[gray!62, line width=0.56pt, -{Latex[length=1.05mm,width=0.75mm]}] (0.813,0.578) -- (0.916,0.829);
  \draw[blue!40, line width=0.56pt, -{Latex[length=1.05mm,width=0.75mm]}] (0.916,0.829) -- (0.986,0.986);
  \draw[cyan!50!black, line width=0.56pt] (0.986,0.986) -- (0.986,0.986);
  \draw[teal!60!black, line width=0.56pt, -{Latex[length=1.05mm,width=0.75mm]}] (0.986,0.986) -- (0.989,0.989);
  \draw[green!50!black, line width=0.56pt, -{Latex[length=1.05mm,width=0.75mm]}] (0.989,0.989) -- (0.992,0.992);
  \draw[orange!70!black, line width=0.56pt, -{Latex[length=1.05mm,width=0.75mm]}] (0.992,0.992) -- (0.990,0.990);
  \draw[red!75!black, line width=0.56pt] (0.990,0.990) -- (0.990,0.990);
  \draw[gray!70, line width=0.56pt, -{Latex[length=1.05mm,width=0.75mm]}] (0.820,0.080) -- (0.830,0.134);
  \draw[gray!62, line width=0.56pt, -{Latex[length=1.05mm,width=0.75mm]}] (0.830,0.134) -- (0.566,0.136);
  \draw[blue!40, line width=0.56pt, -{Latex[length=1.05mm,width=0.75mm]}] (0.566,0.136) -- (0.418,0.341);
  \draw[cyan!50!black, line width=0.56pt, -{Latex[length=1.05mm,width=0.75mm]}] (0.418,0.341) -- (0.383,0.383);
  \draw[teal!60!black, line width=0.56pt, -{Latex[length=1.05mm,width=0.75mm]}] (0.383,0.383) -- (0.385,0.385);
  \draw[green!50!black, line width=0.56pt] (0.385,0.385) -- (0.385,0.385);
  \draw[orange!70!black, line width=0.56pt, -{Latex[length=1.05mm,width=0.75mm]}] (0.385,0.385) -- (0.383,0.383);
  \draw[red!75!black, line width=0.56pt] (0.383,0.383) -- (0.383,0.383);
  \filldraw[fill=darkgreen, draw=white, line width=0.12pt] (0.31200,0.84400) circle (0.0105);
  \filldraw[fill=darkblue, draw=white, line width=0.12pt] (0.09300,0.23700) circle (0.0105);
  \filldraw[fill=darkgreen, draw=white, line width=0.12pt] (0.02700,0.41700) circle (0.0105);
  \filldraw[fill=darkblue, draw=white, line width=0.12pt] (0.58600,0.58600) circle (0.0105);
  \filldraw[fill=darkblue, draw=white, line width=0.12pt] (0.28800,0.28800) circle (0.0105);
  \filldraw[fill=darkblue, draw=white, line width=0.12pt] (0.39600,0.39600) circle (0.0105);
  \filldraw[fill=darkgreen, draw=white, line width=0.12pt] (0.53900,0.17900) circle (0.0105);
  \filldraw[fill=darkblue, draw=white, line width=0.12pt] (0.59500,0.38300) circle (0.0105);
  \filldraw[fill=darkblue, draw=white, line width=0.12pt] (0.81300,0.57800) circle (0.0105);
  \filldraw[fill=darkgreen, draw=white, line width=0.12pt] (0.83000,0.13400) circle (0.0105);
  \filldraw[fill=darkgreen, draw=white, line width=0.12pt] (0.56100,0.94300) circle (0.0105);
  \filldraw[fill=darkblue, draw=white, line width=0.12pt] (0.01100,0.01100) circle (0.0105);
  \filldraw[fill=darkblue, draw=white, line width=0.12pt] (0.14400,0.17900) circle (0.0105);
  \filldraw[fill=darkblue, draw=white, line width=0.12pt] (0.73300,0.73300) circle (0.0105);
  \filldraw[fill=darkblue, draw=white, line width=0.12pt] (0.28300,0.28300) circle (0.0105);
  \filldraw[fill=darkblue, draw=white, line width=0.12pt] (0.50300,0.50300) circle (0.0105);
  \filldraw[fill=darkblue, draw=white, line width=0.12pt] (0.40600,0.37900) circle (0.0105);
  \filldraw[fill=darkblue, draw=white, line width=0.12pt] (0.61500,0.61500) circle (0.0105);
  \filldraw[fill=darkblue, draw=white, line width=0.12pt] (0.91600,0.82900) circle (0.0105);
  \filldraw[fill=darkgreen, draw=white, line width=0.12pt] (0.56600,0.13600) circle (0.0105);
  \filldraw[fill=darkgreen, draw=white, line width=0.12pt] (0.68600,0.99500) circle (0.0105);
  \filldraw[fill=darkblue, draw=white, line width=0.12pt] (0.01400,0.01400) circle (0.0105);
  \filldraw[fill=darkblue, draw=white, line width=0.12pt] (0.15500,0.15500) circle (0.0105);
  \filldraw[fill=darkblue, draw=white, line width=0.12pt] (0.82100,0.82100) circle (0.0105);
  \filldraw[fill=darkblue, draw=white, line width=0.12pt] (0.27400,0.27400) circle (0.0105);
  \filldraw[fill=darkblue, draw=white, line width=0.12pt] (0.59100,0.59100) circle (0.0105);
  \filldraw[fill=darkblue, draw=white, line width=0.12pt] (0.48400,0.48400) circle (0.0105);
  \filldraw[fill=darkblue, draw=white, line width=0.12pt] (0.70300,0.70300) circle (0.0105);
  \filldraw[fill=darkblue, draw=white, line width=0.12pt] (0.98600,0.98600) circle (0.0105);
  \filldraw[fill=darkblue, draw=white, line width=0.12pt] (0.41800,0.34100) circle (0.0105);
  \filldraw[fill=darkgreen, draw=white, line width=0.12pt] (0.74400,0.95800) circle (0.0105);
  \filldraw[fill=darkblue, draw=white, line width=0.12pt] (0.01400,0.01400) circle (0.0105);
  \filldraw[fill=darkblue, draw=white, line width=0.12pt] (0.16100,0.16100) circle (0.0105);
  \filldraw[fill=darkblue, draw=white, line width=0.12pt] (0.82100,0.82100) circle (0.0105);
  \filldraw[fill=darkblue, draw=white, line width=0.12pt] (0.28000,0.28000) circle (0.0105);
  \filldraw[fill=darkblue, draw=white, line width=0.12pt] (0.59100,0.59100) circle (0.0105);
  \filldraw[fill=darkblue, draw=white, line width=0.12pt] (0.48900,0.48900) circle (0.0105);
  \filldraw[fill=darkblue, draw=white, line width=0.12pt] (0.70300,0.70300) circle (0.0105);
  \filldraw[fill=darkblue, draw=white, line width=0.12pt] (0.98600,0.98600) circle (0.0105);
  \filldraw[fill=darkblue, draw=white, line width=0.12pt] (0.38300,0.38300) circle (0.0105);
  \filldraw[fill=darkblue, draw=white, line width=0.12pt] (0.86000,0.86000) circle (0.0105);
  \filldraw[fill=darkblue, draw=white, line width=0.12pt] (0.01100,0.01100) circle (0.0105);
  \filldraw[fill=darkblue, draw=white, line width=0.12pt] (0.16300,0.16300) circle (0.0105);
  \filldraw[fill=darkblue, draw=white, line width=0.12pt] (0.76100,0.76100) circle (0.0105);
  \filldraw[fill=darkblue, draw=white, line width=0.12pt] (0.27700,0.27700) circle (0.0105);
  \filldraw[fill=darkblue, draw=white, line width=0.12pt] (0.57700,0.57700) circle (0.0105);
  \filldraw[fill=darkblue, draw=white, line width=0.12pt] (0.48100,0.48100) circle (0.0105);
  \filldraw[fill=darkblue, draw=white, line width=0.12pt] (0.66600,0.66600) circle (0.0105);
  \filldraw[fill=darkblue, draw=white, line width=0.12pt] (0.98900,0.98900) circle (0.0105);
  \filldraw[fill=darkblue, draw=white, line width=0.12pt] (0.38500,0.38500) circle (0.0105);
  \filldraw[fill=darkblue, draw=white, line width=0.12pt] (0.86000,0.86000) circle (0.0105);
  \filldraw[fill=darkblue, draw=white, line width=0.12pt] (0.01100,0.01100) circle (0.0105);
  \filldraw[fill=darkblue, draw=white, line width=0.12pt] (0.16300,0.16300) circle (0.0105);
  \filldraw[fill=darkblue, draw=white, line width=0.12pt] (0.76100,0.76100) circle (0.0105);
  \filldraw[fill=darkblue, draw=white, line width=0.12pt] (0.27700,0.27700) circle (0.0105);
  \filldraw[fill=darkblue, draw=white, line width=0.12pt] (0.57700,0.57700) circle (0.0105);
  \filldraw[fill=darkblue, draw=white, line width=0.12pt] (0.48100,0.48100) circle (0.0105);
  \filldraw[fill=darkblue, draw=white, line width=0.12pt] (0.66600,0.66600) circle (0.0105);
  \filldraw[fill=darkblue, draw=white, line width=0.12pt] (0.99200,0.99200) circle (0.0105);
  \filldraw[fill=darkblue, draw=white, line width=0.12pt] (0.38500,0.38500) circle (0.0105);
  \filldraw[fill=darkblue, draw=white, line width=0.12pt] (0.86300,0.86300) circle (0.0105);
  \filldraw[fill=darkblue, draw=white, line width=0.12pt] (0.01400,0.01400) circle (0.0105);
  \filldraw[fill=darkblue, draw=white, line width=0.12pt] (0.16100,0.16100) circle (0.0105);
  \filldraw[fill=darkblue, draw=white, line width=0.12pt] (0.75800,0.75800) circle (0.0105);
  \filldraw[fill=darkblue, draw=white, line width=0.12pt] (0.28000,0.28000) circle (0.0105);
  \filldraw[fill=darkblue, draw=white, line width=0.12pt] (0.57400,0.57400) circle (0.0105);
  \filldraw[fill=darkblue, draw=white, line width=0.12pt] (0.48400,0.48400) circle (0.0105);
  \filldraw[fill=darkblue, draw=white, line width=0.12pt] (0.66900,0.66900) circle (0.0105);
  \filldraw[fill=darkblue, draw=white, line width=0.12pt] (0.99000,0.99000) circle (0.0105);
  \filldraw[fill=darkblue, draw=white, line width=0.12pt] (0.38300,0.38300) circle (0.0105);
  \filldraw[fill=darkorange, draw=white, line width=0.12pt] (0.100,0.740) circle (0.014);
  \filldraw[fill=darkblue, draw=white, line width=0.12pt] (0.180,0.490) circle (0.014);
  \filldraw[fill=darkgreen, draw=white, line width=0.12pt] (0.120,0.610) circle (0.014);
  \filldraw[fill=darkblue, draw=white, line width=0.12pt] (0.330,0.580) circle (0.014);
  \filldraw[fill=darkblue, draw=white, line width=0.12pt] (0.460,0.280) circle (0.014);
  \filldraw[fill=darkblue, draw=white, line width=0.12pt] (0.410,0.450) circle (0.014);
  \filldraw[fill=darkorange, draw=white, line width=0.12pt] (0.630,0.120) circle (0.014);
  \filldraw[fill=darkgreen, draw=white, line width=0.12pt] (0.570,0.260) circle (0.014);
  \filldraw[fill=darkblue, draw=white, line width=0.12pt] (0.710,0.330) circle (0.014);
  \filldraw[fill=darkorange, draw=white, line width=0.12pt] (0.820,0.080) circle (0.014);
  \fill[red!80!black] (0.863,0.863) circle (0.018);
  \fill[red!80!black] (0.014,0.014) circle (0.018);
  \fill[red!80!black] (0.161,0.161) circle (0.018);
  \fill[red!80!black] (0.758,0.758) circle (0.018);
  \fill[red!80!black] (0.280,0.280) circle (0.018);
  \fill[red!80!black] (0.574,0.574) circle (0.018);
  \fill[red!80!black] (0.484,0.484) circle (0.018);
  \fill[red!80!black] (0.669,0.669) circle (0.018);
  \fill[red!80!black] (0.990,0.990) circle (0.018);
  \fill[red!80!black] (0.383,0.383) circle (0.018);
  \node[below] at (0.5,-0.09) {$x_1$};
  \node[rotate=90] at (-0.12,0.5) {$x_2$};
\end{tikzpicture}
\end{minipage}\hfill
\begin{minipage}{0.48\textwidth}
\centering
\begin{tikzpicture}[scale=4.9, every node/.style={font=\scriptsize}]
  \draw[very thin,gray!25] (0.000,0.000) grid[step=0.1] (1.000,1.000);
  \draw[thick] (0.000,0.000) rectangle (1.000,1.000);
  \draw[ultra thick,blue!70!black] plot[smooth] coordinates {(0.000,1.000) (0.190,0.990) (0.360,0.960) (0.510,0.910) (0.640,0.840) (0.750,0.750) (0.840,0.640) (0.910,0.510) (0.960,0.360) (0.990,0.190) (1.000,0.000)};
  \foreach \x/\lab in {0/0,1/1} {\draw[thin,gray!65] (\x,0) -- +(0,-0.014); \node[below,font=\scriptsize] at (\x,-0.022) {\lab};}
  \foreach \y/\lab in {0/0,1/1} {\draw[thin,gray!65] (0,\y) -- +(-0.014,0); \node[left,font=\scriptsize] at (-0.022,\y) {\lab};}
  \node[above] at (0.5,1.03) {objective space};
  \draw[gray!70, line width=0.56pt, -{Latex[length=1.05mm,width=0.75mm]}] (0.561,0.721) -- (0.751,0.595);
  \draw[gray!62, line width=0.56pt, -{Latex[length=1.05mm,width=0.75mm]}] (0.751,0.595) -- (0.902,0.398);
  \draw[blue!40, line width=0.56pt, -{Latex[length=1.05mm,width=0.75mm]}] (0.902,0.398) -- (0.951,0.269);
  \draw[cyan!50!black, line width=0.56pt, -{Latex[length=1.05mm,width=0.75mm]}] (0.951,0.269) -- (0.966,0.264);
  \draw[teal!60!black, line width=0.56pt, -{Latex[length=1.05mm,width=0.75mm]}] (0.966,0.264) -- (0.980,0.260);
  \draw[green!50!black, line width=0.56pt] (0.980,0.260) -- (0.980,0.260);
  \draw[orange!70!black, line width=0.56pt, -{Latex[length=1.05mm,width=0.75mm]}] (0.980,0.260) -- (0.981,0.255);
  \draw[red!75!black, line width=0.56pt] (0.981,0.255) -- (0.981,0.255);
  \draw[gray!70, line width=0.56pt, -{Latex[length=1.05mm,width=0.75mm]}] (0.534,0.864) -- (0.297,0.968);
  \draw[gray!62, line width=0.56pt, -{Latex[length=1.05mm,width=0.75mm]}] (0.297,0.968) -- (0.022,1.000);
  \draw[blue!40, line width=0.56pt, -{Latex[length=1.05mm,width=0.75mm]}] (0.022,1.000) -- (0.028,1.000);
  \draw[cyan!50!black, line width=0.56pt] (0.028,1.000) -- (0.028,1.000);
  \draw[teal!60!black, line width=0.56pt, -{Latex[length=1.05mm,width=0.75mm]}] (0.028,1.000) -- (0.022,1.000);
  \draw[green!50!black, line width=0.56pt] (0.022,1.000) -- (0.022,1.000);
  \draw[orange!70!black, line width=0.56pt, -{Latex[length=1.05mm,width=0.75mm]}] (0.022,1.000) -- (0.028,1.000);
  \draw[red!75!black, line width=0.56pt] (0.028,1.000) -- (0.028,1.000);
  \draw[gray!70, line width=0.56pt, -{Latex[length=1.05mm,width=0.75mm]}] (0.537,0.807) -- (0.357,0.913);
  \draw[gray!62, line width=0.56pt, -{Latex[length=1.05mm,width=0.75mm]}] (0.357,0.913) -- (0.296,0.974);
  \draw[blue!40, line width=0.56pt, -{Latex[length=1.05mm,width=0.75mm]}] (0.296,0.974) -- (0.286,0.976);
  \draw[cyan!50!black, line width=0.56pt, -{Latex[length=1.05mm,width=0.75mm]}] (0.286,0.976) -- (0.295,0.974);
  \draw[teal!60!black, line width=0.56pt, -{Latex[length=1.05mm,width=0.75mm]}] (0.295,0.974) -- (0.300,0.973);
  \draw[green!50!black, line width=0.56pt] (0.300,0.973) -- (0.300,0.973);
  \draw[orange!70!black, line width=0.56pt, -{Latex[length=1.05mm,width=0.75mm]}] (0.300,0.973) -- (0.295,0.974);
  \draw[red!75!black, line width=0.56pt] (0.295,0.974) -- (0.295,0.974);
  \draw[gray!70, line width=0.56pt, -{Latex[length=1.05mm,width=0.75mm]}] (0.687,0.777) -- (0.828,0.657);
  \draw[gray!62, line width=0.56pt, -{Latex[length=1.05mm,width=0.75mm]}] (0.828,0.657) -- (0.929,0.463);
  \draw[blue!40, line width=0.56pt, -{Latex[length=1.05mm,width=0.75mm]}] (0.929,0.463) -- (0.968,0.327);
  \draw[cyan!50!black, line width=0.56pt] (0.968,0.327) -- (0.968,0.327);
  \draw[teal!60!black, line width=0.56pt, -{Latex[length=1.05mm,width=0.75mm]}] (0.968,0.327) -- (0.943,0.421);
  \draw[green!50!black, line width=0.56pt] (0.943,0.421) -- (0.943,0.421);
  \draw[orange!70!black, line width=0.56pt, -{Latex[length=1.05mm,width=0.75mm]}] (0.943,0.421) -- (0.942,0.425);
  \draw[red!75!black, line width=0.56pt] (0.942,0.425) -- (0.942,0.425);
  \draw[gray!70, line width=0.56pt, -{Latex[length=1.05mm,width=0.75mm]}] (0.595,0.855) -- (0.494,0.917);
  \draw[gray!62, line width=0.56pt, -{Latex[length=1.05mm,width=0.75mm]}] (0.494,0.917) -- (0.486,0.920);
  \draw[blue!40, line width=0.56pt, -{Latex[length=1.05mm,width=0.75mm]}] (0.486,0.920) -- (0.473,0.925);
  \draw[cyan!50!black, line width=0.56pt, -{Latex[length=1.05mm,width=0.75mm]}] (0.473,0.925) -- (0.481,0.922);
  \draw[teal!60!black, line width=0.56pt, -{Latex[length=1.05mm,width=0.75mm]}] (0.481,0.922) -- (0.477,0.923);
  \draw[green!50!black, line width=0.56pt] (0.477,0.923) -- (0.477,0.923);
  \draw[orange!70!black, line width=0.56pt, -{Latex[length=1.05mm,width=0.75mm]}] (0.477,0.923) -- (0.481,0.922);
  \draw[red!75!black, line width=0.56pt] (0.481,0.922) -- (0.481,0.922);
  \draw[gray!70, line width=0.56pt, -{Latex[length=1.05mm,width=0.75mm]}] (0.675,0.815) -- (0.635,0.843);
  \draw[gray!62, line width=0.56pt, -{Latex[length=1.05mm,width=0.75mm]}] (0.635,0.843) -- (0.753,0.747);
  \draw[blue!40, line width=0.56pt, -{Latex[length=1.05mm,width=0.75mm]}] (0.753,0.747) -- (0.833,0.651);
  \draw[cyan!50!black, line width=0.56pt] (0.833,0.651) -- (0.833,0.651);
  \draw[teal!60!black, line width=0.56pt, -{Latex[length=1.05mm,width=0.75mm]}] (0.833,0.651) -- (0.821,0.667);
  \draw[green!50!black, line width=0.56pt] (0.821,0.667) -- (0.821,0.667);
  \draw[orange!70!black, line width=0.56pt, -{Latex[length=1.05mm,width=0.75mm]}] (0.821,0.667) -- (0.819,0.670);
  \draw[red!75!black, line width=0.56pt] (0.819,0.670) -- (0.819,0.670);
  \draw[gray!70, line width=0.56pt, -{Latex[length=1.05mm,width=0.75mm]}] (0.544,0.794) -- (0.557,0.839);
  \draw[gray!62, line width=0.56pt, -{Latex[length=1.05mm,width=0.75mm]}] (0.557,0.839) -- (0.631,0.846);
  \draw[blue!40, line width=0.56pt, -{Latex[length=1.05mm,width=0.75mm]}] (0.631,0.846) -- (0.733,0.766);
  \draw[cyan!50!black, line width=0.56pt, -{Latex[length=1.05mm,width=0.75mm]}] (0.733,0.766) -- (0.739,0.761);
  \draw[teal!60!black, line width=0.56pt, -{Latex[length=1.05mm,width=0.75mm]}] (0.739,0.761) -- (0.730,0.769);
  \draw[green!50!black, line width=0.56pt] (0.730,0.769) -- (0.730,0.769);
  \draw[orange!70!black, line width=0.56pt, -{Latex[length=1.05mm,width=0.75mm]}] (0.730,0.769) -- (0.733,0.766);
  \draw[red!75!black, line width=0.56pt] (0.733,0.766) -- (0.733,0.766);
  \draw[gray!70, line width=0.56pt, -{Latex[length=1.05mm,width=0.75mm]}] (0.634,0.804) -- (0.728,0.749);
  \draw[gray!62, line width=0.56pt, -{Latex[length=1.05mm,width=0.75mm]}] (0.728,0.749) -- (0.852,0.622);
  \draw[blue!40, line width=0.56pt, -{Latex[length=1.05mm,width=0.75mm]}] (0.852,0.622) -- (0.912,0.506);
  \draw[cyan!50!black, line width=0.56pt] (0.912,0.506) -- (0.912,0.506);
  \draw[teal!60!black, line width=0.56pt, -{Latex[length=1.05mm,width=0.75mm]}] (0.912,0.506) -- (0.889,0.556);
  \draw[green!50!black, line width=0.56pt] (0.889,0.556) -- (0.889,0.556);
  \draw[orange!70!black, line width=0.56pt, -{Latex[length=1.05mm,width=0.75mm]}] (0.889,0.556) -- (0.890,0.553);
  \draw[red!75!black, line width=0.56pt] (0.890,0.553) -- (0.890,0.553);
  \draw[gray!70, line width=0.56pt, -{Latex[length=1.05mm,width=0.75mm]}] (0.734,0.694) -- (0.893,0.503);
  \draw[gray!62, line width=0.56pt, -{Latex[length=1.05mm,width=0.75mm]}] (0.893,0.503) -- (0.982,0.237);
  \draw[blue!40, line width=0.56pt, -{Latex[length=1.05mm,width=0.75mm]}] (0.982,0.237) -- (1.000,0.028);
  \draw[cyan!50!black, line width=0.56pt] (1.000,0.028) -- (1.000,0.028);
  \draw[teal!60!black, line width=0.56pt, -{Latex[length=1.05mm,width=0.75mm]}] (1.000,0.028) -- (1.000,0.022);
  \draw[green!50!black, line width=0.56pt, -{Latex[length=1.05mm,width=0.75mm]}] (1.000,0.022) -- (1.000,0.015);
  \draw[orange!70!black, line width=0.56pt, -{Latex[length=1.05mm,width=0.75mm]}] (1.000,0.015) -- (1.000,0.021);
  \draw[red!75!black, line width=0.56pt] (1.000,0.021) -- (1.000,0.021);
  \draw[gray!70, line width=0.56pt, -{Latex[length=1.05mm,width=0.75mm]}] (0.561,0.661) -- (0.610,0.647);
  \draw[gray!62, line width=0.56pt, -{Latex[length=1.05mm,width=0.75mm]}] (0.610,0.647) -- (0.533,0.831);
  \draw[blue!40, line width=0.56pt, -{Latex[length=1.05mm,width=0.75mm]}] (0.533,0.831) -- (0.613,0.855);
  \draw[cyan!50!black, line width=0.56pt, -{Latex[length=1.05mm,width=0.75mm]}] (0.613,0.855) -- (0.619,0.854);
  \draw[teal!60!black, line width=0.56pt, -{Latex[length=1.05mm,width=0.75mm]}] (0.619,0.854) -- (0.622,0.851);
  \draw[green!50!black, line width=0.56pt] (0.622,0.851) -- (0.622,0.851);
  \draw[orange!70!black, line width=0.56pt, -{Latex[length=1.05mm,width=0.75mm]}] (0.622,0.851) -- (0.619,0.854);
  \draw[red!75!black, line width=0.56pt] (0.619,0.854) -- (0.619,0.854);
  \filldraw[fill=darkgreen, draw=white, line width=0.12pt] (0.75100,0.59500) circle (0.0105);
  \filldraw[fill=darkblue, draw=white, line width=0.12pt] (0.29700,0.96800) circle (0.0105);
  \filldraw[fill=darkgreen, draw=white, line width=0.12pt] (0.35700,0.91300) circle (0.0105);
  \filldraw[fill=darkblue, draw=white, line width=0.12pt] (0.82800,0.65700) circle (0.0105);
  \filldraw[fill=darkblue, draw=white, line width=0.12pt] (0.49400,0.91700) circle (0.0105);
  \filldraw[fill=darkblue, draw=white, line width=0.12pt] (0.63500,0.84300) circle (0.0105);
  \filldraw[fill=darkgreen, draw=white, line width=0.12pt] (0.55700,0.83900) circle (0.0105);
  \filldraw[fill=darkblue, draw=white, line width=0.12pt] (0.72800,0.74900) circle (0.0105);
  \filldraw[fill=darkblue, draw=white, line width=0.12pt] (0.89300,0.50300) circle (0.0105);
  \filldraw[fill=darkgreen, draw=white, line width=0.12pt] (0.61000,0.64700) circle (0.0105);
  \filldraw[fill=darkgreen, draw=white, line width=0.12pt] (0.90200,0.39800) circle (0.0105);
  \filldraw[fill=darkblue, draw=white, line width=0.12pt] (0.02200,1.00000) circle (0.0105);
  \filldraw[fill=darkblue, draw=white, line width=0.12pt] (0.29600,0.97400) circle (0.0105);
  \filldraw[fill=darkblue, draw=white, line width=0.12pt] (0.92900,0.46300) circle (0.0105);
  \filldraw[fill=darkblue, draw=white, line width=0.12pt] (0.48600,0.92000) circle (0.0105);
  \filldraw[fill=darkblue, draw=white, line width=0.12pt] (0.75300,0.74700) circle (0.0105);
  \filldraw[fill=darkblue, draw=white, line width=0.12pt] (0.63100,0.84600) circle (0.0105);
  \filldraw[fill=darkblue, draw=white, line width=0.12pt] (0.85200,0.62200) circle (0.0105);
  \filldraw[fill=darkblue, draw=white, line width=0.12pt] (0.98200,0.23700) circle (0.0105);
  \filldraw[fill=darkgreen, draw=white, line width=0.12pt] (0.53300,0.83100) circle (0.0105);
  \filldraw[fill=darkgreen, draw=white, line width=0.12pt] (0.95100,0.26900) circle (0.0105);
  \filldraw[fill=darkblue, draw=white, line width=0.12pt] (0.02800,1.00000) circle (0.0105);
  \filldraw[fill=darkblue, draw=white, line width=0.12pt] (0.28600,0.97600) circle (0.0105);
  \filldraw[fill=darkblue, draw=white, line width=0.12pt] (0.96800,0.32700) circle (0.0105);
  \filldraw[fill=darkblue, draw=white, line width=0.12pt] (0.47300,0.92500) circle (0.0105);
  \filldraw[fill=darkblue, draw=white, line width=0.12pt] (0.83300,0.65100) circle (0.0105);
  \filldraw[fill=darkblue, draw=white, line width=0.12pt] (0.73300,0.76600) circle (0.0105);
  \filldraw[fill=darkblue, draw=white, line width=0.12pt] (0.91200,0.50600) circle (0.0105);
  \filldraw[fill=darkblue, draw=white, line width=0.12pt] (1.00000,0.02800) circle (0.0105);
  \filldraw[fill=darkblue, draw=white, line width=0.12pt] (0.61300,0.85500) circle (0.0105);
  \filldraw[fill=darkgreen, draw=white, line width=0.12pt] (0.96600,0.26400) circle (0.0105);
  \filldraw[fill=darkblue, draw=white, line width=0.12pt] (0.02800,1.00000) circle (0.0105);
  \filldraw[fill=darkblue, draw=white, line width=0.12pt] (0.29500,0.97400) circle (0.0105);
  \filldraw[fill=darkblue, draw=white, line width=0.12pt] (0.96800,0.32700) circle (0.0105);
  \filldraw[fill=darkblue, draw=white, line width=0.12pt] (0.48100,0.92200) circle (0.0105);
  \filldraw[fill=darkblue, draw=white, line width=0.12pt] (0.83300,0.65100) circle (0.0105);
  \filldraw[fill=darkblue, draw=white, line width=0.12pt] (0.73900,0.76100) circle (0.0105);
  \filldraw[fill=darkblue, draw=white, line width=0.12pt] (0.91200,0.50600) circle (0.0105);
  \filldraw[fill=darkblue, draw=white, line width=0.12pt] (1.00000,0.02800) circle (0.0105);
  \filldraw[fill=darkblue, draw=white, line width=0.12pt] (0.61900,0.85400) circle (0.0105);
  \filldraw[fill=darkblue, draw=white, line width=0.12pt] (0.98000,0.26000) circle (0.0105);
  \filldraw[fill=darkblue, draw=white, line width=0.12pt] (0.02200,1.00000) circle (0.0105);
  \filldraw[fill=darkblue, draw=white, line width=0.12pt] (0.30000,0.97300) circle (0.0105);
  \filldraw[fill=darkblue, draw=white, line width=0.12pt] (0.94300,0.42100) circle (0.0105);
  \filldraw[fill=darkblue, draw=white, line width=0.12pt] (0.47700,0.92300) circle (0.0105);
  \filldraw[fill=darkblue, draw=white, line width=0.12pt] (0.82100,0.66700) circle (0.0105);
  \filldraw[fill=darkblue, draw=white, line width=0.12pt] (0.73000,0.76900) circle (0.0105);
  \filldraw[fill=darkblue, draw=white, line width=0.12pt] (0.88900,0.55600) circle (0.0105);
  \filldraw[fill=darkblue, draw=white, line width=0.12pt] (1.00000,0.02200) circle (0.0105);
  \filldraw[fill=darkblue, draw=white, line width=0.12pt] (0.62200,0.85100) circle (0.0105);
  \filldraw[fill=darkblue, draw=white, line width=0.12pt] (0.98000,0.26000) circle (0.0105);
  \filldraw[fill=darkblue, draw=white, line width=0.12pt] (0.02200,1.00000) circle (0.0105);
  \filldraw[fill=darkblue, draw=white, line width=0.12pt] (0.30000,0.97300) circle (0.0105);
  \filldraw[fill=darkblue, draw=white, line width=0.12pt] (0.94300,0.42100) circle (0.0105);
  \filldraw[fill=darkblue, draw=white, line width=0.12pt] (0.47700,0.92300) circle (0.0105);
  \filldraw[fill=darkblue, draw=white, line width=0.12pt] (0.82100,0.66700) circle (0.0105);
  \filldraw[fill=darkblue, draw=white, line width=0.12pt] (0.73000,0.76900) circle (0.0105);
  \filldraw[fill=darkblue, draw=white, line width=0.12pt] (0.88900,0.55600) circle (0.0105);
  \filldraw[fill=darkblue, draw=white, line width=0.12pt] (1.00000,0.01500) circle (0.0105);
  \filldraw[fill=darkblue, draw=white, line width=0.12pt] (0.62200,0.85100) circle (0.0105);
  \filldraw[fill=darkblue, draw=white, line width=0.12pt] (0.98100,0.25500) circle (0.0105);
  \filldraw[fill=darkblue, draw=white, line width=0.12pt] (0.02800,1.00000) circle (0.0105);
  \filldraw[fill=darkblue, draw=white, line width=0.12pt] (0.29500,0.97400) circle (0.0105);
  \filldraw[fill=darkblue, draw=white, line width=0.12pt] (0.94200,0.42500) circle (0.0105);
  \filldraw[fill=darkblue, draw=white, line width=0.12pt] (0.48100,0.92200) circle (0.0105);
  \filldraw[fill=darkblue, draw=white, line width=0.12pt] (0.81900,0.67000) circle (0.0105);
  \filldraw[fill=darkblue, draw=white, line width=0.12pt] (0.73300,0.76600) circle (0.0105);
  \filldraw[fill=darkblue, draw=white, line width=0.12pt] (0.89000,0.55300) circle (0.0105);
  \filldraw[fill=darkblue, draw=white, line width=0.12pt] (1.00000,0.02100) circle (0.0105);
  \filldraw[fill=darkblue, draw=white, line width=0.12pt] (0.61900,0.85400) circle (0.0105);
  \filldraw[fill=darkorange, draw=white, line width=0.12pt] (0.561,0.721) circle (0.014);
  \filldraw[fill=darkblue, draw=white, line width=0.12pt] (0.534,0.864) circle (0.014);
  \filldraw[fill=darkgreen, draw=white, line width=0.12pt] (0.537,0.807) circle (0.014);
  \filldraw[fill=darkblue, draw=white, line width=0.12pt] (0.687,0.777) circle (0.014);
  \filldraw[fill=darkblue, draw=white, line width=0.12pt] (0.595,0.855) circle (0.014);
  \filldraw[fill=darkblue, draw=white, line width=0.12pt] (0.675,0.815) circle (0.014);
  \filldraw[fill=darkorange, draw=white, line width=0.12pt] (0.544,0.794) circle (0.014);
  \filldraw[fill=darkgreen, draw=white, line width=0.12pt] (0.634,0.804) circle (0.014);
  \filldraw[fill=darkblue, draw=white, line width=0.12pt] (0.734,0.694) circle (0.014);
  \filldraw[fill=darkorange, draw=white, line width=0.12pt] (0.561,0.661) circle (0.014);
  \fill[red!80!black] (0.981,0.255) circle (0.018);
  \fill[red!80!black] (0.028,1.000) circle (0.018);
  \fill[red!80!black] (0.295,0.974) circle (0.018);
  \fill[red!80!black] (0.942,0.425) circle (0.018);
  \fill[red!80!black] (0.481,0.922) circle (0.018);
  \fill[red!80!black] (0.819,0.670) circle (0.018);
  \fill[red!80!black] (0.733,0.766) circle (0.018);
  \fill[red!80!black] (0.890,0.553) circle (0.018);
  \fill[red!80!black] (1.000,0.021) circle (0.018);
  \fill[red!80!black] (0.619,0.854) circle (0.018);
  \node[below] at (0.5,-0.09) {$F_1$};
  \node[rotate=90] at (-0.12,0.5) {$F_2$};
\end{tikzpicture}
\end{minipage}
\caption{Projected finite-difference diffusion for the perturbed $10$-point summed quadratic example using the active repulsion term ($\tau=2\cdot 10^{-4}$, $\sigma=0.03$), a smaller step size $\alpha=0.004$, and $540$ projected iterations in $[0,1]^2$. Left: decision-space paths approaching the efficient diagonal. Right: the corresponding objective-space paths approaching the curved Pareto front. Rank-colored markers denote sampled nonfinal iterates, including the initial set, and red points denote the last plotted iterates. The trajectory colors encode time along the sampled path, progressing from gray through blue/teal and orange to red; arrow tips indicate increasing iteration. The marker colors encode the current objective-space dominance rank using the same layer palette as Figure~\ref{fig:intro-layered-example}: blue for layer~1, green for layer~2, and orange for layer~3 or deeper.}
\label{fig:quadratic-ten-point-both}
\end{figure}
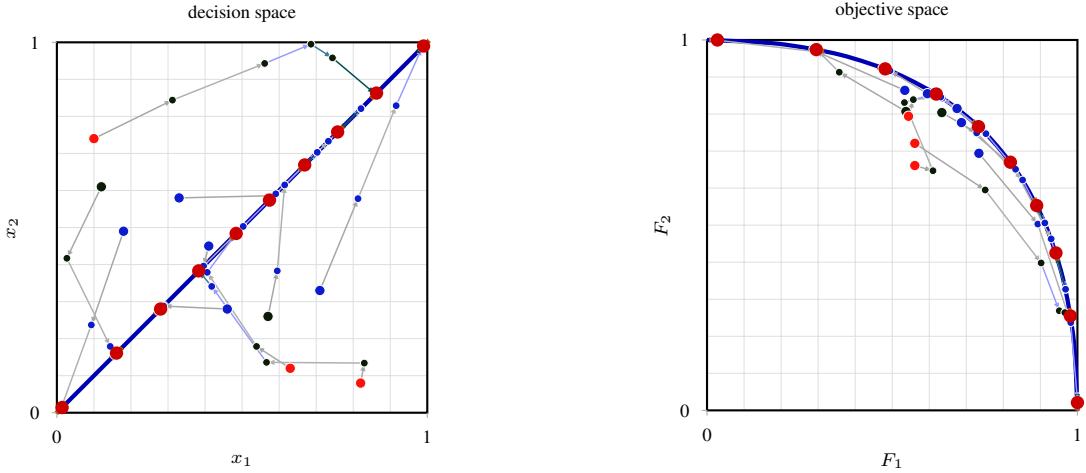

The perturbation makes the nonlinear geometry visible in both spaces. In decision space, the paths bend toward the diagonal efficient set $\{(t,t):t\in[0,1]\}$. In objective space, the images of these paths bend toward the curved front $\mathcal F$. The final approximation set is numerically close to
\[
(0.013,1.000),\ (0.281,0.977),\ (0.443,0.936),\ (0.589,0.872),\ (0.681,0.811),
\]
\[
(0.770,0.729),\ (0.839,0.641),\ (0.920,0.485),\ (0.969,0.320),\ (1.000,0.017).
\]
Thus the longer projected diffusion drives the approximation set close to the Pareto front while the repulsion term keeps the points separated along the curve. In particular, this ten-point example shows that the normalized projected finite-difference scheme is already effective on a genuinely nonlinear front: the trajectories bend in both decision and objective space, yet the final point set remains well distributed rather than collapsing into a few clusters.

\section{A three-objective supersphere benchmark}
\label{sec:three-objective-supersphere}

We now record three-objective experiments for both the layered hypervolume and layered magnitude indicators.  The purpose of this section is not to replace the biobjective analysis above, but to show that the same layered construction extends naturally from the anchored staircase in the plane to the anchored orthogonal union in three dimensions.  For a finite first-front set $B\subseteq [0,\infty)^3$, the three-dimensional magnitude indicator used here is
\[
\Mag_3(\Dom(B))
=1+\frac{L_1(B)+L_2(B)+L_3(B)}{2}
 +\frac{A_{12}(B)+A_{13}(B)+A_{23}(B)}{4}
 +\frac{\HV_3(B)}{8},
\]
where $L_i(B)=\max_{b\in B} b_i$, $A_{ij}(B)$ is the two-dimensional dominated area of the projection onto the $(i,j)$ coordinate plane, and $\HV_3(B)$ is the ordinary anchored three-dimensional hypervolume.  Thus hypervolume appears as the volume term of the magnitude expansion, while magnitude also keeps the edge and face contributions.  The layered three-dimensional surrogate is obtained exactly as before by evaluating each nondomination layer separately and multiplying layer $\ell$ by $\varepsilon^{\ell-1}$.

The benchmark is the simplex-constrained supersphere problem described in Appendix~\ref{app:supersphere-benchmark}.  We first show the three Pareto-front surfaces in Figure~\ref{fig:supersphere-gamma-surfaces}, because the curvature parameter determines how strongly the interior of the front bulges away from the coordinate planes.  We then compare magnitude and hypervolume on these same three cases in Table~\ref{tab:three-d-hv-vs-mag}, followed by point-set visualizations for increasing Das--Dennis budgets.  In each run for Table~\ref{tab:three-d-hv-vs-mag} we used Das--Dennis initialization with $H=3$, a small Gaussian perturbation $\sigma_{\rm DD}=0.01$, and therefore $\mu=(H+1)(H+2)/2=10$ points.  Magnitude and hypervolume runs start from the same seed for each value of $\gamma\in\{0.25,0.5,1.0\}$.  The reference archive for the IGD value was generated by sampling the decision simplex.  The implementation switches between exact inclusion--exclusion for small fronts and sweep-based forward derivatives once the first front exceeds the prescribed threshold.  In Table~\ref{tab:three-d-hv-vs-mag}, the final sets are cross-evaluated by both $\Mag_3$ and $\HV_3$.

\newcommand{\superspheresurface}[2]{%
\begin{tikzpicture}
\begin{axis}[
  width=0.31\linewidth,
  view={135}{27},
  xlabel={$F_1$}, ylabel={$F_2$}, zlabel={$F_3$},
  xmin=0,xmax=1,ymin=0,ymax=1,zmin=0,zmax=1,
  xtick={0,1},ytick={0,1},ztick={0,1},
  title={#2},
  title style={font=\small},
  tick label style={font=\scriptsize},
  label style={font=\scriptsize},
]
\addplot3[
  surf,
  shader=flat,
  samples=15,
  samples y=15,
  domain=0:1,
  y domain=0:1,
  draw=blue!35,
  fill opacity=0.45,
] ({1-((((x-1)^2+((1-x)*y)^2+((1-x)*(1-y))^2)/2)^(#1))},
   {1-(((x^2+(((1-x)*y)-1)^2+((1-x)*(1-y))^2)/2)^(#1))},
   {1-(((x^2+((1-x)*y)^2+(((1-x)*(1-y))-1)^2)/2)^(#1))});
\end{axis}
\end{tikzpicture}%
}

\begin{figure}[htbp]
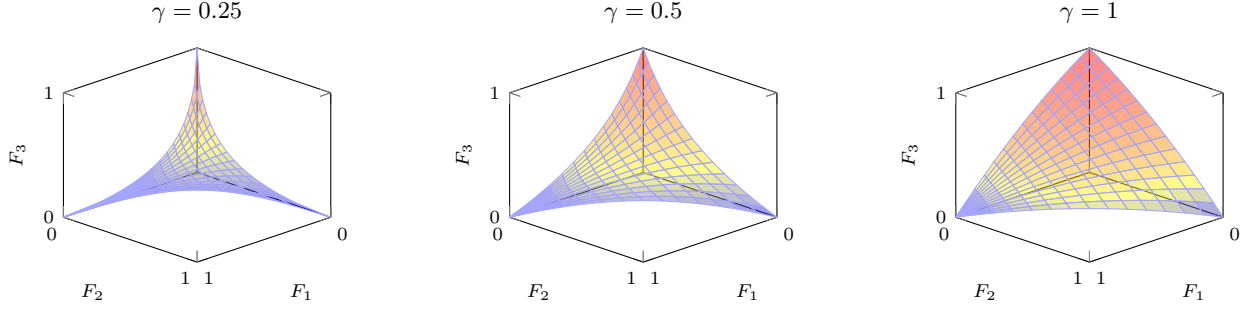

\centering
\superspheresurface{0.25}{$\gamma=0.25$}\hfill
\superspheresurface{0.5}{$\gamma=0.5$}\hfill
\superspheresurface{1.0}{$\gamma=1$}
\caption{Analytic supersphere Pareto-front surfaces for the three bulge parameters used in the experiments.  The vertices remain fixed at the three unit vectors in objective space; changing $\gamma$ changes the position of the interior of the surface.}
\label{fig:supersphere-gamma-surfaces}
\end{figure}

\begin{table}[htbp]
\centering
\small
\setlength{\tabcolsep}{4pt}
\begin{tabular}{clrrrrrrr}
\hline
$\gamma$ & Optimized & iters & acc. & $\Mag_3^0$ & $\Mag_3^f$ & $\HV_3^0$ & $\HV_3^f$ & IGD$^f$ \\
\hline
0.25 & Magnitude & 31 & 1 & 2.5235 & 2.59869 & 0.0243 & 0.02505 & 0.10389 \\
0.25 & Hypervolume & 80 & 78 & 2.5235 & 1.88994 & 0.0243 & 0.03311 & 0.0749 \\
0.5 & Magnitude & 31 & 1 & 2.76845 & 2.77634 & 0.12395 & 0.12768 & 0.13779 \\
0.5 & Hypervolume & 80 & 78 & 2.76845 & 2.45964 & 0.12395 & 0.16074 & 0.10531 \\
1.0 & Magnitude & 80 & 75 & 3.07394 & 3.08901 & 0.43042 & 0.46661 & 0.14698 \\
1.0 & Hypervolume & 80 & 76 & 3.07394 & 3.0022 & 0.43042 & 0.51512 & 0.14814 \\
\hline
\end{tabular}
\caption{Three-objective comparison on the supersphere benchmark for the three bulge parameters shown in Figure~\ref{fig:supersphere-gamma-surfaces}.  Each pair of rows uses a common perturbed Das--Dennis initialization with $H=3$, $\mu=10$, seed $8$, $80$ maximum iterations, and exact-front threshold $6$.  The columns $\Mag_3^0,\HV_3^0$ report the common initial values for the corresponding $\gamma$, while $\Mag_3^f,\HV_3^f$ and IGD$^f$ report final cross-evaluations of the approximation set produced by the optimized indicator.}
\label{tab:three-d-hv-vs-mag}
\end{table}

Across the three curvatures, the hypervolume run consistently obtains the larger final volume term, while the magnitude-indicator run gives the larger final value of $\Mag_3$.  For $\gamma=0.25$ and $\gamma=0.5$, the magnitude-indicator ascent quickly stabilizes after preserving the extreme vertices; for $\gamma=1$ it also accepts many interior moves.  This behavior is consistent with the formula above: in three dimensions, $\HV_3$ measures only the volume contribution, while $\Mag_3$ also rewards the coordinate-axis extents and the projected dominated areas.

Figure~\ref{fig:supersphere-3d-comparison} combines the point-set comparisons for $\gamma=0.25$, $0.5$, and $1.0$ in a single float.  Rows correspond to the three front curvatures, and columns to the Das--Dennis budgets $H=2,3,4$.  Red circles denote the final approximation set obtained by optimizing $\Mag_3$, while black triangles denote the final approximation set obtained by optimizing $\HV_3$.  As the front becomes less strongly bulged, both optimized sets move further away from the coordinate axes and deeper into the interior of the surface; the characteristic difference remains visible, with magnitude retaining more extreme extent and hypervolum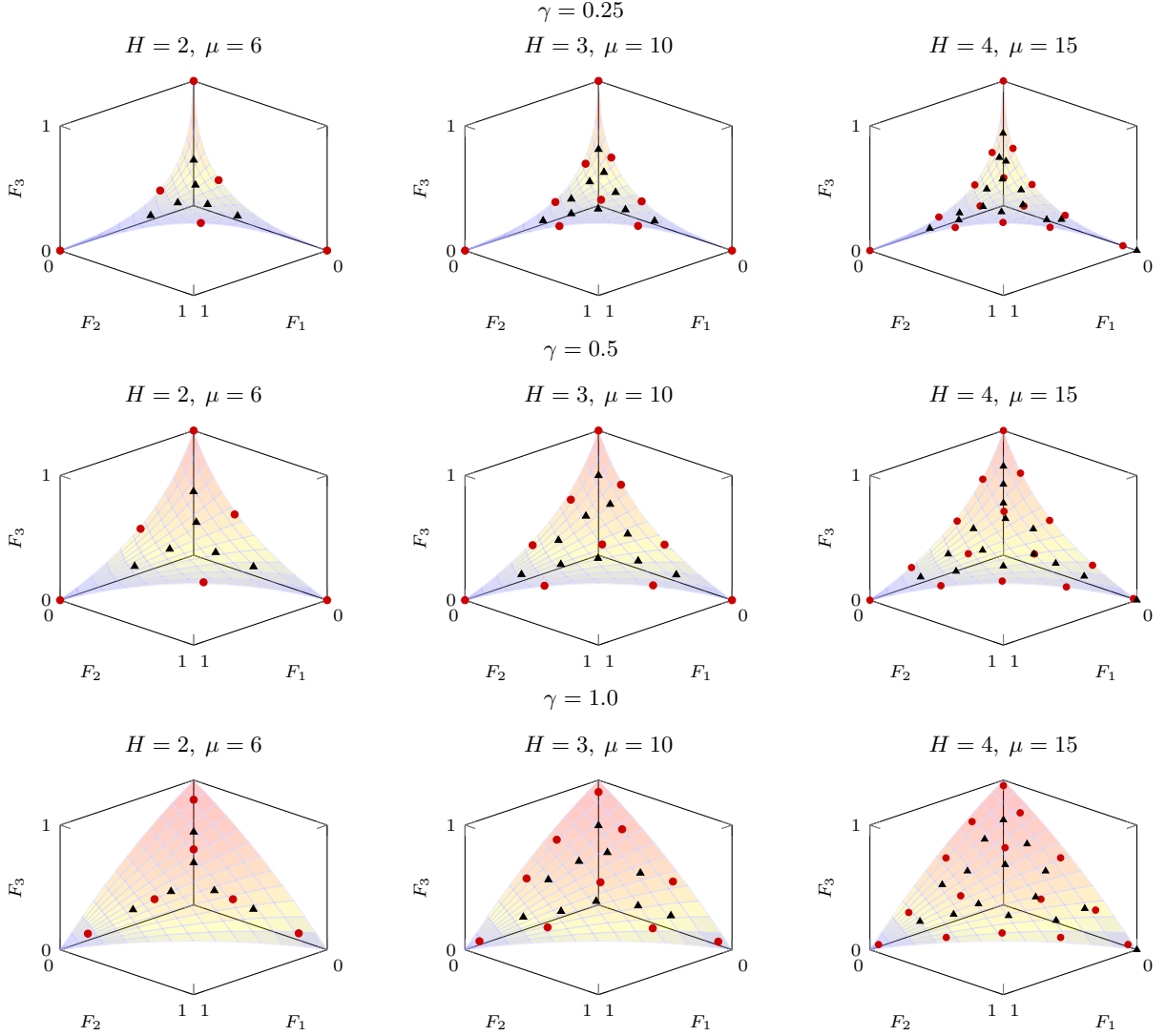
\begin{figure}[p]
\centering
{\small\bfseries $\gamma=0.25$}\par\vspace{0.15em}
\begin{minipage}{0.32\linewidth}
\centering
\begin{tikzpicture}
\begin{axis}[
  width=\linewidth,
  view={135}{27},
  xlabel={$F_1$}, ylabel={$F_2$}, zlabel={$F_3$},
  xmin=0,xmax=1,ymin=0,ymax=1,zmin=0,zmax=1,
  xtick={0,1}, ytick={0,1}, ztick={0,1},
  title={$H=2,\ \mu=6$},
  title style={font=\small},
  tick label style={font=\scriptsize},
  label style={font=\scriptsize},
]
\addplot3[surf,shader=flat,opacity=0.22,samples=13,samples y=13,domain=0:1,y domain=0:1,draw=blue!35,fill opacity=0.22]
({1-((((x-1)^2+((1-x)*y)^2+((1-x)*(1-y))^2)/2)^0.25)},
 {1-(((x^2+(((1-x)*y)-1)^2+((1-x)*(1-y))^2)/2)^0.25)},
 {1-(((x^2+((1-x)*y)^2+(((1-x)*(1-y))-1)^2)/2)^0.25)});
\addplot3[only marks,mark=*,mark size=1.4pt,red!80!black] coordinates {
(0.0000,0.0000,1.0000) (0.0747,0.2614,0.3257) (0.0000,1.0000,0.0000)
(0.3232,0.0725,0.2637) (0.2665,0.3202,0.0725) (1.0000,0.0000,0.0000)};
\addplot3[only marks,mark=triangle*,mark size=1.6pt,black] coordinates {
(0.1282,0.1282,0.4581) (0.1962,0.2093,0.3116) (0.1285,0.4576,0.1285)
(0.3144,0.1957,0.2069) (0.2078,0.3123,0.1971) (0.4526,0.1308,0.1308)};
\end{axis}
\end{tikzpicture}
\end{minipage}\hfill
\begin{minipage}{0.32\linewidth}
\centering
\begin{tikzpicture}
\begin{axis}[
  width=\linewidth,
  view={135}{27},
  xlabel={$F_1$}, ylabel={$F_2$}, zlabel={$F_3$},
  xmin=0,xmax=1,ymin=0,ymax=1,zmin=0,zmax=1,
  xtick={0,1}, ytick={0,1}, ztick={0,1},
  title={$H=3,\ \mu=10$},
  title style={font=\small},
  tick label style={font=\scriptsize},
  label style={font=\scriptsize},
]
\addplot3[surf,shader=flat,opacity=0.22,samples=13,samples y=13,domain=0:1,y domain=0:1,draw=blue!35,fill opacity=0.22]
({1-((((x-1)^2+((1-x)*y)^2+((1-x)*(1-y))^2)/2)^0.25)},
 {1-(((x^2+(((1-x)*y)-1)^2+((1-x)*(1-y))^2)/2)^0.25)},
 {1-(((x^2+((1-x)*y)^2+(((1-x)*(1-y))-1)^2)/2)^0.25)});
\addplot3[only marks,mark=*,mark size=1.4pt,red!80!black] coordinates {
(0.0000,0.0000,1.0000) (0.0602,0.1559,0.4637) (0.0724,0.3955,0.2031)
(0.0000,1.0000,0.0000) (0.1770,0.0805,0.4289) (0.2378,0.2571,0.2253)
(0.1605,0.4565,0.0600) (0.3990,0.0758,0.2000) (0.4546,0.1618,0.0582)
(1.0000,0.0000,0.0000)};
\addplot3[only marks,mark=triangle*,mark size=1.6pt,black] coordinates {
(0.0995,0.0995,0.5229) (0.1426,0.1833,0.3845) (0.1580,0.2862,0.2660)
(0.1006,0.5202,0.1006) (0.2238,0.1579,0.3295) (0.2734,0.2706,0.1692)
(0.1734,0.3752,0.1650) (0.3417,0.1376,0.2258) (0.3867,0.1826,0.1409)
(0.5170,0.1020,0.1020)};
\end{axis}
\end{tikzpicture}
\end{minipage}\hfill
\begin{minipage}{0.32\linewidth}
\centering
\begin{tikzpicture}
\begin{axis}[
  width=\linewidth,
  view={135}{27},
  xlabel={$F_1$}, ylabel={$F_2$}, zlabel={$F_3$},
  xmin=0,xmax=1,ymin=0,ymax=1,zmin=0,zmax=1,
  xtick={0,1}, ytick={0,1}, ztick={0,1},
  title={$H=4,\ \mu=15$},
  title style={font=\small},
  tick label style={font=\scriptsize},
  label style={font=\scriptsize},
]
\addplot3[surf,shader=flat,opacity=0.22,samples=13,samples y=13,domain=0:1,y domain=0:1,draw=blue!35,fill opacity=0.22]
({1-((((x-1)^2+((1-x)*y)^2+((1-x)*(1-y))^2)/2)^0.25)},
 {1-(((x^2+(((1-x)*y)-1)^2+((1-x)*(1-y))^2)/2)^0.25)},
 {1-(((x^2+((1-x)*y)^2+(((1-x)*(1-y))-1)^2)/2)^0.25)});
\addplot3[only marks,mark=*,mark size=1.25pt,red!80!black] coordinates {
(0.0000,0.0000,1.0000) (0.0502,0.1221,0.5211) (0.0726,0.2870,0.2988)
(0.0515,0.5143,0.1258) (0.0044,0.8997,0.0044) (0.1372,0.0528,0.4945)
(0.1791,0.1818,0.3538) (0.1872,0.3410,0.1871) (0.1400,0.4896,0.0532)
(0.2887,0.0748,0.2969) (0.3509,0.1767,0.1872) (0.2942,0.2914,0.0761)
(0.5307,0.0475,0.1170) (0.4966,0.1359,0.0535) (1.0000,0.0000,0.0000)};
\addplot3[only marks,mark=triangle*,mark size=1.45pt,black] coordinates {
(0.0626,0.0604,0.6246) (0.1214,0.1408,0.4508) (0.1470,0.2797,0.2797)
(0.0880,0.5227,0.1091) (0.0000,1.0000,0.0000) (0.1353,0.1046,0.4725)
(0.1879,0.1811,0.3460) (0.1877,0.3337,0.1944) (0.1403,0.4651,0.1056)
(0.2733,0.1488,0.2850) (0.3419,0.1894,0.1841) (0.2857,0.2714,0.1508)
(0.4472,0.1191,0.1456) (0.4708,0.1371,0.1038) (0.6164,0.0649,0.0636)};
\end{axis}
\end{tikzpicture}
\end{minipage}
\vspace{0.45em}
{\small\bfseries $\gamma=0.5$}\par\vspace{0.15em}
\begin{minipage}{0.32\linewidth}
\centering
\begin{tikzpicture}
\begin{axis}[
  width=\linewidth,
  view={135}{27},
  xlabel={$F_1$}, ylabel={$F_2$}, zlabel={$F_3$},
  xmin=0,xmax=1,ymin=0,ymax=1,zmin=0,zmax=1,
  xtick={0,1}, ytick={0,1}, ztick={0,1},
  title={$H=2,\ \mu=6$},
  title style={font=\small},
  tick label style={font=\scriptsize},
  label style={font=\scriptsize},
]
\addplot3[surf,shader=flat,opacity=0.22,samples=13,samples y=13,domain=0:1,y domain=0:1,draw=blue!35,fill opacity=0.22]
({1-((((x-1)^2+((1-x)*y)^2+((1-x)*(1-y))^2)/2)^0.5)},
 {1-(((x^2+(((1-x)*y)-1)^2+((1-x)*(1-y))^2)/2)^0.5)},
 {1-(((x^2+((1-x)*y)^2+(((1-x)*(1-y))-1)^2)/2)^0.5)});
\addplot3[only marks,mark=*,mark size=1.4pt,red!80!black] coordinates {
(0.0000,0.0000,1.0000) (0.1478,0.4552,0.5443) (0.0000,1.0000,0.0000)
(0.5410,0.1434,0.4587) (0.4628,0.5370,0.1433) (1.0000,0.0000,0.0000)};
\addplot3[only marks,mark=triangle*,mark size=1.6pt,black] coordinates {
(0.2518,0.2506,0.6914) (0.3544,0.3741,0.5264) (0.2479,0.6959,0.2479)
(0.5324,0.3529,0.3682) (0.3697,0.5347,0.3484) (0.6923,0.2505,0.2505)};
\end{axis}
\end{tikzpicture}
\end{minipage}\hfill
\begin{minipage}{0.32\linewidth}
\centering
\begin{tikzpicture}
\begin{axis}[
  width=\linewidth,
  view={135}{27},
  xlabel={$F_1$}, ylabel={$F_2$}, zlabel={$F_3$},
  xmin=0,xmax=1,ymin=0,ymax=1,zmin=0,zmax=1,
  xtick={0,1}, ytick={0,1}, ztick={0,1},
  title={$H=3,\ \mu=10$},
  title style={font=\small},
  tick label style={font=\scriptsize},
  label style={font=\scriptsize},
]
\addplot3[surf,shader=flat,opacity=0.22,samples=13,samples y=13,domain=0:1,y domain=0:1,draw=blue!35,fill opacity=0.22]
({1-((((x-1)^2+((1-x)*y)^2+((1-x)*(1-y))^2)/2)^0.5)},
 {1-(((x^2+(((1-x)*y)-1)^2+((1-x)*(1-y))^2)/2)^0.5)},
 {1-(((x^2+((1-x)*y)^2+(((1-x)*(1-y))-1)^2)/2)^0.5)});
\addplot3[only marks,mark=*,mark size=1.4pt,red!80!black] coordinates {
(0.0000,0.0000,1.0000) (0.1193,0.2878,0.7119) (0.1405,0.6351,0.3643)
(0.0000,1.0000,0.0000) (0.3642,0.1570,0.6334) (0.4203,0.4475,0.3991)
(0.2954,0.7044,0.1188) (0.6371,0.1438,0.3619) (0.7023,0.2977,0.1155) (1.0000,0.0000,0.0000)};
\addplot3[only marks,mark=triangle*,mark size=1.6pt,black] coordinates {
(0.1874,0.1874,0.7746) (0.2526,0.3400,0.6222) (0.2839,0.5022,0.4534)
(0.1894,0.7720,0.1894) (0.3869,0.2927,0.5588) (0.4707,0.4670,0.3131)
(0.3190,0.6159,0.2912) (0.5560,0.2562,0.4121) (0.6172,0.3346,0.2695) (0.7689,0.1919,0.1919)};
\end{axis}
\end{tikzpicture}
\end{minipage}\hfill
\begin{minipage}{0.32\linewidth}
\centering
\begin{tikzpicture}
\begin{axis}[
  width=\linewidth,
  view={135}{27},
  xlabel={$F_1$}, ylabel={$F_2$}, zlabel={$F_3$},
  xmin=0,xmax=1,ymin=0,ymax=1,zmin=0,zmax=1,
  xtick={0,1}, ytick={0,1}, ztick={0,1},
  title={$H=4,\ \mu=15$},
  title style={font=\small},
  tick label style={font=\scriptsize},
  label style={font=\scriptsize},
]
\addplot3[surf,shader=flat,opacity=0.22,samples=13,samples y=13,domain=0:1,y domain=0:1,draw=blue!35,fill opacity=0.22]
({1-((((x-1)^2+((1-x)*y)^2+((1-x)*(1-y))^2)/2)^0.5)},
 {1-(((x^2+(((1-x)*y)-1)^2+((1-x)*(1-y))^2)/2)^0.5)},
 {1-(((x^2+((1-x)*y)^2+(((1-x)*(1-y))-1)^2)/2)^0.5)});
\addplot3[only marks,mark=*,mark size=1.25pt,red!80!black] coordinates {
(0.0000,0.0000,1.0000) (0.0974,0.2255,0.7744) (0.1451,0.4914,0.5084)
(0.0999,0.7679,0.2320) (0.0120,0.9861,0.0120) (0.2591,0.1050,0.7408)
(0.3235,0.3281,0.5862) (0.3371,0.5696,0.3366) (0.2640,0.7359,0.1059)
(0.4942,0.1478,0.5054) (0.5824,0.3196,0.3368) (0.5039,0.4954,0.1529)
(0.7827,0.0945,0.2172) (0.7339,0.2657,0.1158) (1.0000,0.0000,0.0000)};
\addplot3[only marks,mark=triangle*,mark size=1.45pt,black] coordinates {
(0.1496,0.1497,0.8218) (0.2225,0.2225,0.7294) (0.2629,0.4869,0.4804)
(0.1795,0.7848,0.1791) (0.0000,1.0000,0.0000) (0.2948,0.2948,0.6312)
(0.3438,0.3592,0.5468) (0.3396,0.5678,0.3365) (0.2716,0.6636,0.2716)
(0.4860,0.2617,0.4819) (0.5247,0.3691,0.3617) (0.4841,0.4837,0.2619)
(0.6453,0.2318,0.3254) (0.6637,0.3094,0.2217) (0.7926,0.1731,0.1731)};
\end{axis}
\end{tikzpicture}
\end{minipage}
\vspace{0.45em}
{\small\bfseries $\gamma=1.0$}\par\vspace{0.15em}
\begin{minipage}{0.32\linewidth}
\centering
\begin{tikzpicture}
\begin{axis}[
  width=\linewidth,
  view={135}{27},
  xlabel={$F_1$}, ylabel={$F_2$}, zlabel={$F_3$},
  xmin=0,xmax=1,ymin=0,ymax=1,zmin=0,zmax=1,
  xtick={0,1}, ytick={0,1}, ztick={0,1},
  title={$H=2,\ \mu=6$},
  title style={font=\small},
  tick label style={font=\scriptsize},
  label style={font=\scriptsize},
]
\addplot3[surf,shader=flat,opacity=0.22,samples=13,samples y=13,domain=0:1,y domain=0:1,draw=blue!35,fill opacity=0.22]
({1-((((x-1)^2+((1-x)*y)^2+((1-x)*(1-y))^2)/2)^1.0)},
 {1-(((x^2+(((1-x)*y)-1)^2+((1-x)*(1-y))^2)/2)^1.0)},
 {1-(((x^2+((1-x)*y)^2+(((1-x)*(1-y))-1)^2)/2)^1.0)});
\addplot3[only marks,mark=*,mark size=1.4pt,red!80!black] coordinates {
(0.1975,0.1975,0.9849) (0.5399,0.5399,0.8338) (0.1978,0.9849,0.1978)
(0.8334,0.5405,0.5402) (0.5396,0.8341,0.5396) (0.9856,0.1936,0.1937)};
\addplot3[only marks,mark=triangle*,mark size=1.6pt,black] coordinates {
(0.4425,0.4425,0.9030) (0.6004,0.6004,0.7704) (0.4510,0.8981,0.4510)
(0.7701,0.6006,0.6006) (0.6070,0.7619,0.6072) (0.9008,0.4464,0.4464)};
\end{axis}
\end{tikzpicture}
\end{minipage}\hfill
\begin{minipage}{0.32\linewidth}
\centering
\begin{tikzpicture}
\begin{axis}[
  width=\linewidth,
  view={135}{27},
  xlabel={$F_1$}, ylabel={$F_2$}, zlabel={$F_3$},
  xmin=0,xmax=1,ymin=0,ymax=1,zmin=0,zmax=1,
  xtick={0,1}, ytick={0,1}, ztick={0,1},
  title={$H=3,\ \mu=10$},
  title style={font=\small},
  tick label style={font=\scriptsize},
  label style={font=\scriptsize},
]
\addplot3[surf,shader=flat,opacity=0.22,samples=13,samples y=13,domain=0:1,y domain=0:1,draw=blue!35,fill opacity=0.22]
({1-((((x-1)^2+((1-x)*y)^2+((1-x)*(1-y))^2)/2)^1.0)},
 {1-(((x^2+(((1-x)*y)-1)^2+((1-x)*(1-y))^2)/2)^1.0)},
 {1-(((x^2+((1-x)*y)^2+(((1-x)*(1-y))-1)^2)/2)^1.0)});
\addplot3[only marks,mark=*,mark size=1.4pt,red!80!black] coordinates {
(0.1248,0.1248,0.9943) (0.3274,0.5040,0.9054) (0.3019,0.8593,0.6053)
(0.0996,0.9965,0.0996) (0.6158,0.3030,0.8529) (0.6614,0.6759,0.6625)
(0.5013,0.9074,0.3194) (0.8459,0.3070,0.6268) (0.8993,0.5183,0.3301) (0.9961,0.1044,0.1044)};
\addplot3[only marks,mark=triangle*,mark size=1.6pt,black] coordinates {
(0.4009,0.4009,0.9243) (0.5206,0.5873,0.8187) (0.4548,0.7686,0.6957)
(0.3880,0.9302,0.3876) (0.6615,0.5152,0.7731) (0.7226,0.7051,0.5429)
(0.5503,0.8468,0.4960) (0.8098,0.4341,0.6495) (0.8498,0.5693,0.4591) (0.9364,0.3731,0.3731)};
\end{axis}
\end{tikzpicture}
\end{minipage}\hfill
\begin{minipage}{0.32\linewidth}
\centering
\begin{tikzpicture}
\begin{axis}[
  width=\linewidth,
  view={135}{27},
  xlabel={$F_1$}, ylabel={$F_2$}, zlabel={$F_3$},
  xmin=0,xmax=1,ymin=0,ymax=1,zmin=0,zmax=1,
  xtick={0,1}, ytick={0,1}, ztick={0,1},
  title={$H=4,\ \mu=15$},
  title style={font=\small},
  tick label style={font=\scriptsize},
  label style={font=\scriptsize},
]
\addplot3[surf,shader=flat,opacity=0.22,samples=13,samples y=13,domain=0:1,y domain=0:1,draw=blue!35,fill opacity=0.22]
({1-((((x-1)^2+((1-x)*y)^2+((1-x)*(1-y))^2)/2)^1.0)},
 {1-(((x^2+(((1-x)*y)-1)^2+((1-x)*(1-y))^2)/2)^1.0)},
 {1-(((x^2+((1-x)*y)^2+(((1-x)*(1-y))-1)^2)/2)^1.0)});
\addplot3[only marks,mark=*,mark size=1.25pt,red!80!black] coordinates {
(0.0588,0.0591,0.9988) (0.2399,0.3656,0.9560) (0.3122,0.7409,0.7563)
(0.2576,0.9482,0.3920) (0.0648,0.9985,0.0648) (0.4735,0.2389,0.9240)
(0.5251,0.5375,0.8412) (0.5481,0.8294,0.5415) (0.4890,0.9178,0.2453)
(0.7417,0.3117,0.7555) (0.8352,0.5157,0.5589) (0.7534,0.7437,0.3138)
(0.9544,0.2456,0.3704) (0.9176,0.4898,0.2453) (0.9986,0.0637,0.0637)};
\addplot3[only marks,mark=triangle*,mark size=1.45pt,black] coordinates {
(0.3609,0.3609,0.9413) (0.4139,0.5913,0.8503) (0.4462,0.7624,0.7060)
(0.3235,0.9330,0.4227) (0.0000,1.0000,0.0000) (0.5498,0.4096,0.8720)
(0.6014,0.6148,0.7605) (0.5687,0.8107,0.5598) (0.5030,0.8977,0.3795)
(0.7466,0.4774,0.7118) (0.8027,0.6163,0.5183) (0.7160,0.7546,0.4427)
(0.8469,0.3886,0.6068) (0.8842,0.5112,0.4268) (0.9530,0.3284,0.3284)};
\end{axis}
\end{tikzpicture}
\end{minipage}
\caption{Final approximation sets for the three supersphere fronts and the three Das--Dennis point budgets. Rows correspond to $\gamma=0.25$, $0.5$, and $1.0$; columns correspond to $H=2,3,4$ (i.e., $\mu=6,10,15$). Red circles show the final set obtained by optimizing $\Mag_3$, and black triangles show the final set obtained by optimizing $\HV_3$.}
\label{fig:supersphere-3d-comparison}
\end{figure}

\begin{table}[htbp]
\centering
\small
\setlength{\tabcolsep}{5pt}
\begin{tabular}{rrclrrr}
\hline
$H$ & $\mu$ & Optimized & iters & $\Mag_3^f$ & $\HV_3^f$ & IGD$^f$ \\
\hline
2 & 6 & Magnitude & 31 & 2.56904 & 0.01487 & 0.17032 \\
2 & 6 & Hypervolume & 80 & 1.77975 & 0.029 & 0.09878 \\
3 & 10 & Magnitude & 31 & 2.59869 & 0.02505 & 0.10389 \\
3 & 10 & Hypervolume & 80 & 1.88994 & 0.03311 & 0.0749 \\
4 & 15 & Magnitude & 44 & 2.56597 & 0.03067 & 0.07675 \\
4 & 15 & Hypervolume & 36 & 2.23831 & 0.03477 & 0.06501 \\
\hline
\end{tabular}
\caption{Point-set sensitivity for the $\gamma=0.25$ supersphere case shown in Figure~\ref{fig:supersphere-3d-comparison}.  The Das--Dennis parameter $H$ gives $\mu=(H+1)(H+2)/2$ initial points.  For each $H$, magnitude and hypervolume use the same perturbed initial point set and are cross-evaluated by both indicators.}
\label{tab:three-d-pointset-sensitivity}
\end{table}

\section{Convergence behaviour}
\label{sec:convergence-behaviour}

This section records seven representative convergence traces rather than an exhaustive performance study.  The aim is to show how layered ascent behaves when the initial approximation set is away from the final indicator configuration.  We use a curved two-objective quadratic front, the three-objective supersphere instance with $\gamma=1$ and a simplex-constrained Das--Dennis point set, a three-objective layered-start experiment in which the same $\gamma=1$ objective is optimized in the mildly enlarged decision box $[-0.4,1.4]^3$, and four longer 500-episode variants of the layered-start run with stochastic stagnation recovery.  The longer variants use two point budgets, $H=4$ with $\mu=15$ and $H=5$ with $\mu=(H+1)(H+2)/2=21$, and are run once with the layered magnitude indicator as the optimized indicator and once with the layered hypervolume indicator.  The layered-start experiments are included specifically to make the layer dynamics visible without using an extreme starting cloud: the initial points are sampled from $[-0.25,1.25]^3$, so several points begin outside the natural nondominated range, but the set does not collapse to one layer immediately.  In all convergence plots and layer tables, points are reported every twentieth iteration or episode, with the final endpoint included explicitly, so that the figure remains self-contained and the manuscript does not depend on external graphics files.

In the two-objective case we use $\mu=10$ points and the quadratic map
\[
 f_1(x)=1-\frac{(x_1-1)^2+(x_2-1)^2}{2},\qquad
 f_2(x)=1-\frac{x_1^2+x_2^2}{2}.
\]
Its efficient set is the diagonal segment $x_1=x_2=t$, $0\le t\le 1$, and the Pareto front is the curve $(2t-t^2,1-t^2)$.  The initial points are deliberately placed off this diagonal, which produces several nondomination layers.  The simplex-constrained three-objective run uses the supersphere benchmark with $\gamma=1$ and a perturbed Das--Dennis grid with $H=4$, hence $\mu=15$ points.  In this case the initial set is already on the Pareto-front surface and all points remain mutually nondominated; the convergence therefore consists mainly of redistribution on the front.  The layered-start three-objective run also uses $\mu=15$ points, but samples the initial decisions in $[-0.25,1.25]^3$ and projects onto $[-0.4,1.4]^3$.  It uses the same layer weight $\varepsilon=10^{-3}$ as the other convergence traces.  The sampled layer profile evolves from $8+5+2$ through $11+4$ and reaches a single nondominated layer only after several ascent steps.

The 500-episode variants replace the previous small one-point stochastic restart by a larger backtracking perturbation strategy.  When the optimized layered indicator grows by less than $5\cdot 10^{-3}$ over a ten-episode window, several points are perturbed simultaneously with step length about $0.16$.  The perturbed state is allowed to have a temporarily lower indicator value; after this backtracking move, the deterministic gradient ascent continues until the next stagnation.  To avoid reporting a run that ends immediately after a random move, perturbations are disabled during the final ten episodes.  With the fixed seed used here, each of the four recovery runs makes $38$ accepted perturbation/backtracking moves.  The hypervolume runs use exactly the same starting clouds and recovery rule as the magnitude-indicator runs, but replace the front contribution by the ordinary anchored hypervolume of each layer.

As a gradient-free reference, we also use a projected stochastic hillclimber.  Starting from the same approximation set and feasible projection, it repeatedly selects one point, perturbs it in a random unit direction, projects the whole set back to the feasible domain, and accepts the trial set only if the chosen layered indicator does not decrease.  Failed trials shrink the step size, while stagnation can trigger the same recovery logic as in the gradient runs.  Thus the method is an accept--reject baseline for the same nonsmooth scalar objective rather than a different quality indicator; Algorithm~\ref{alg:stochastic-reference} gives the complete pseudocode and the source-code reference is listed in the reproducibility paragraph.

The CPU times in Table~\ref{tab:convergence-growth-summary} are process CPU seconds measured by the reproduction script on the execution machine; they are intended as order-of-magnitude reproducibility information rather than hardware-independent benchmark times. We do not repeat the simplex-constrained $H=4$, $\mu=15$ supersphere view here, since the corresponding final point-set comparison for $\gamma=1$ is already shown in Figure~\ref{fig:supersphere-3d-comparison}. The convergence section therefore keeps only the layered-start three-objective view in Figure~\ref{fig:convergence-layered-start-3d}, which is the nonredundant case: it starts with several dominated layers and visibly collapses toward a single nondominated layer.

\def\ConvTwoDCoords{(0,1.97332) (20,2.19575) (40,2.19838) (50,2.19842)}
\def\ConvThreeDCoords{(0,3.11902) (20,3.12557) (40,3.12568) (60,3.12617) (70,3.12721)}
\def\ConvThreeDFarCoords{(0,2.70275) (20,3.11784) (40,3.12891) (45,3.12922)}
\def\ConvThreeDRecoveryCoords{(0,2.70275) (20,3.04563) (40,3.11894) (60,3.12791) (80,3.12811) (100,3.12476) (120,3.12891) (140,3.12968) (160,3.12785) (180,3.11622) (200,3.1296) (220,3.12935) (240,3.12462) (260,3.13029) (280,3.12864) (300,3.12394) (320,3.13017) (340,3.12986) (360,3.12462) (380,3.13051) (400,3.12905) (420,3.11921) (440,3.13007) (460,3.12925) (480,3.12569) (500,3.13047)}
\def\ConvThreeDHFiveRecoveryCoords{(0,2.90496) (20,3.10371) (40,3.1387) (60,3.14385) (80,3.13844) (100,3.14895) (120,3.14875) (140,3.14565) (160,3.14916) (180,3.14847) (200,3.14497) (220,3.1495) (240,3.14809) (260,3.14129) (280,3.14861) (300,3.1484) (320,3.1428) (340,3.14859) (360,3.1487) (380,3.14641) (400,3.14935) (420,3.14858) (440,3.14662) (460,3.14911) (480,3.14776) (500,3.14948)}
\def\ConvThreeDRecoveryHVCoords{(0,0.32132) (20,0.48307) (40,0.52525) (60,0.53459) (80,0.53498) (100,0.53202) (120,0.53646) (140,0.53524) (160,0.53309) (180,0.52111) (200,0.53586) (220,0.53452) (240,0.5333) (260,0.52854) (280,0.53602) (300,0.53687) (320,0.52813) (340,0.53963) (360,0.53749) (380,0.53408) (400,0.53894) (420,0.53933) (440,0.5383) (460,0.52916) (480,0.53985) (500,0.54011)}
\def\ConvThreeDHFiveRecoveryHVCoords{(0,0.37968) (20,0.51546) (40,0.54353) (60,0.55058) (80,0.55344) (100,0.5501) (120,0.55649) (140,0.55579) (160,0.55405) (180,0.55025) (200,0.55556) (220,0.55642) (240,0.55101) (260,0.55684) (280,0.55497) (300,0.55312) (320,0.55602) (340,0.55538) (360,0.55553) (380,0.55206) (400,0.55613) (420,0.55642) (440,0.55489) (460,0.54747) (480,0.55593) (500,0.55632)}
\def\ConvTwoDInitial{(-0.08040,0.95960) (0.18195,0.89195) (0.14235,0.99235) (0.45980,0.87980) (0.23100,0.97100) (0.55715,0.88715) (0.43635,0.88635) (0.65375,0.80375) (0.44880,0.76880) (0.62720,0.66720)}
\def\ConvTwoDFinal{(0.26304,0.97997) (0.68426,0.80807) (0.02047,0.99989) (0.86221,0.60461) (0.56932,0.88184) (0.92717,0.46690) (0.78014,0.71792) (0.99988,0.02181) (0.43427,0.93857) (0.97574,0.28726)}
\def\ConvTwoDPathA{(-0.08040,0.95960) (0.24211,0.98325) (0.25684,0.98098) (0.26304,0.97997)}
\def\ConvTwoDPathB{(0.18195,0.89195) (0.56965,0.88159) (0.67407,0.81587) (0.68426,0.80807)}
\def\ConvTwoDPathC{(0.14235,0.99235) (0.04014,0.99919) (0.01515,0.99994) (0.02047,0.99989)}
\def\ConvTwoDPathD{(0.45980,0.87980) (0.76819,0.73112) (0.85557,0.61564) (0.86221,0.60461)}
\def\ConvTwoDPathE{(0.23100,0.97100) (0.45548,0.93131) (0.56037,0.88646) (0.56932,0.88184)}
\def\ConvTwoDPathF{(0.55715,0.88715) (0.86641,0.59741) (0.92522,0.47215) (0.92717,0.46690)}
\def\ConvTwoDPathG{(0.43635,0.88635) (0.65441,0.82990) (0.76967,0.72953) (0.78014,0.71792)}
\def\ConvTwoDPathH{(0.65375,0.80375) (0.99806,0.08610) (0.99982,0.02692) (0.99988,0.02181)}
\def\ConvTwoDPathI{(0.44880,0.76880) (0.42360,0.94202) (0.42090,0.94287) (0.43427,0.93857)}
\def\ConvTwoDPathJ{(0.62720,0.66720) (0.94346,0.41898) (0.97492,0.29167) (0.97574,0.28726)}
\def\ConvThreeDInitial{(0.00012,0.00024,1.00000) (0.19104,0.42143,0.94267) (0.25000,0.74968,0.75032) (0.19575,0.93950,0.43129) (0.00000,1.00000,0.00000) (0.43126,0.18541,0.93956) (0.55513,0.56048,0.81712) (0.57161,0.80263,0.57263) (0.43841,0.93720,0.18804) (0.74329,0.24996,0.75662) (0.81388,0.55008,0.57168) (0.75974,0.73998,0.26068) (0.94715,0.17704,0.40693) (0.93130,0.45552,0.19341) (0.99994,0.00750,0.01505)}
\def\ConvThreeDFinal{(0.06610,0.06610,0.99848) (0.27160,0.37647,0.95082) (0.31160,0.78937,0.70421) (0.25894,0.94964,0.38596) (0.06352,0.99859,0.06352) (0.51521,0.25202,0.90695) (0.46420,0.59849,0.83287) (0.60780,0.82483,0.47268) (0.49025,0.91734,0.24637) (0.71729,0.31895,0.77749) (0.83311,0.46281,0.59882) (0.78150,0.71279,0.31797) (0.95074,0.26985,0.37752) (0.90814,0.51246,0.25134) (0.99842,0.06731,0.06731)}
\def\ConvThreeDFarInitial{(0.92688,0.37612,0.03290) (-0.84716,0.34796,0.49719) (0.42595,0.61024,0.33143) (0.48400,0.30517,-0.91450) (0.62285,-0.61288,0.43122) (-0.19041,0.84088,0.35830) (0.58438,0.76884,0.17729) (-0.11108,0.70843,0.67328) (0.11755,-0.23001,0.69029) (0.40933,-0.03361,-0.08624) (0.91969,0.47040,0.08967) (0.81800,0.52380,0.20113) (-0.38610,0.21937,0.28621) (0.53409,0.85470,0.48021) (0.84580,0.46122,0.54178)}
\def\ConvThreeDFarFinal{(0.99840,0.06775,0.06775) (0.26865,0.55069,0.89014) (0.46325,0.90077,0.42776) (0.66931,0.81797,0.29372) (0.83090,0.29272,0.65116) (0.29069,0.83696,0.64237) (0.30592,0.95960,0.30951) (0.37725,0.29945,0.94759) (0.08250,0.08250,0.99760) (0.77214,0.68775,0.46407) (0.94886,0.30599,0.36779) (0.89186,0.54743,0.26537) (0.42860,0.70963,0.76435) (0.04882,0.99918,0.04882) (0.65466,0.34606,0.82346)}
\def\ConvThreeDRecoveryInitial{(0.92688,0.37612,0.03290) (-0.84716,0.34796,0.49719) (0.42595,0.61024,0.33143) (0.48400,0.30517,-0.91450) (0.62285,-0.61288,0.43122) (-0.19041,0.84088,0.35830) (0.58438,0.76884,0.17729) (-0.11108,0.70843,0.67328) (0.11755,-0.23001,0.69029) (0.40933,-0.03361,-0.08624) (0.91969,0.47040,0.08967) (0.81800,0.52380,0.20113) (-0.38610,0.21937,0.28621) (0.53409,0.85470,0.48021) (0.84580,0.46122,0.54178)}
\def\ConvThreeDRecoveryFinal{(0.99874,0.06030,0.06030) (0.52922,0.36219,0.89014) (0.41699,0.76298,0.71469) (0.74032,0.74132,0.40671) (0.85967,0.29165,0.60578) (0.31084,0.95865,0.31085) (0.30874,0.88662,0.55210) (0.05927,0.05927,0.99878) (0.30441,0.30438,0.96060) (0.57825,0.87568,0.27940) (0.94981,0.32563,0.34899) (0.87841,0.57285,0.28111) (0.28786,0.61280,0.85577) (0.05409,0.99896,0.05553) (0.71258,0.41331,0.76579)}
\def\ConvThreeDHFiveRecoveryInitial{(0.92688,0.37612,0.03290) (-0.84716,0.34796,0.49719) (0.42595,0.61024,0.33143) (0.48400,0.30517,-0.91450) (0.62285,-0.61288,0.43122) (-0.19041,0.84088,0.35830) (0.58438,0.76884,0.17729) (-0.11108,0.70843,0.67328) (0.11755,-0.23001,0.69029) (0.40933,-0.03361,-0.08624) (0.91969,0.47040,0.08967) (0.81800,0.52380,0.20113) (-0.38610,0.21937,0.28621) (0.53409,0.85470,0.48021) (0.84580,0.46122,0.54178) (0.75777,-0.23690,0.35714) (-0.56368,0.55926,0.49094) (0.02314,0.97881,-0.24806) (0.63104,0.35229,0.80238) (0.97907,0.13055,-0.13739) (0.78554,0.47648,0.31484)}
\def\ConvThreeDHFiveRecoveryFinal{(0.96666,0.27542,0.28991) (0.26849,0.57435,0.87828) (0.23821,0.97731,0.23818) (0.82888,0.64606,0.34941) (0.81150,0.32695,0.67440) (0.58630,0.87182,0.26997) (0.40502,0.93768,0.32353) (0.24682,0.24543,0.97561) (0.38034,0.38547,0.93232) (0.71984,0.77359,0.33443) (0.91562,0.49489,0.24357) (0.66560,0.67354,0.66073) (0.35535,0.71158,0.77807) (0.24792,0.90716,0.51506) (0.56221,0.26351,0.88471) (0.90665,0.25305,0.51581) (0.32371,0.82173,0.66079) (0.03240,0.99964,0.03250) (0.03060,0.03018,0.99969) (0.99939,0.04327,0.04081) (0.69773,0.35673,0.78953)}
\def\ConvThreeDRecoveryHVFinal{(0.94443,0.35275,0.35272) (0.42621,0.81762,0.64091) (0.49313,0.71507,0.73715) (0.84426,0.40971,0.60386) (0.66377,0.69284,0.64139) (0.36010,0.94164,0.36003) (0.50953,0.87029,0.47741) (0.33999,0.33999,0.94906) (0.48403,0.45855,0.88542) (0.78526,0.66840,0.47030) (0.87993,0.52255,0.42593) (0.67548,0.79391,0.42285) (0.64223,0.42759,0.81639) (0.41633,0.61842,0.83422) (0.73970,0.50232,0.70813)}
\def\ConvThreeDHFiveRecoveryHVFinal{(0.88824,0.39580,0.51893) (0.68390,0.41133,0.79020) (0.43842,0.73935,0.73342) (0.69721,0.65686,0.64359) (0.83032,0.39225,0.63303) (0.65910,0.80189,0.43842) (0.55187,0.87652,0.38261) (0.42366,0.47075,0.89943) (0.57106,0.38784,0.86628) (0.59613,0.77168,0.60244) (0.90220,0.49771,0.36559) (0.58525,0.60103,0.77900) (0.37615,0.62907,0.83579) (0.39316,0.90626,0.47129) (0.75080,0.48814,0.70244) (0.82275,0.60391,0.48553) (0.76124,0.72170,0.39772) (0.32173,0.95522,0.32174) (0.32763,0.32758,0.95330) (0.95375,0.32663,0.32584) (0.40980,0.83130,0.62563)}

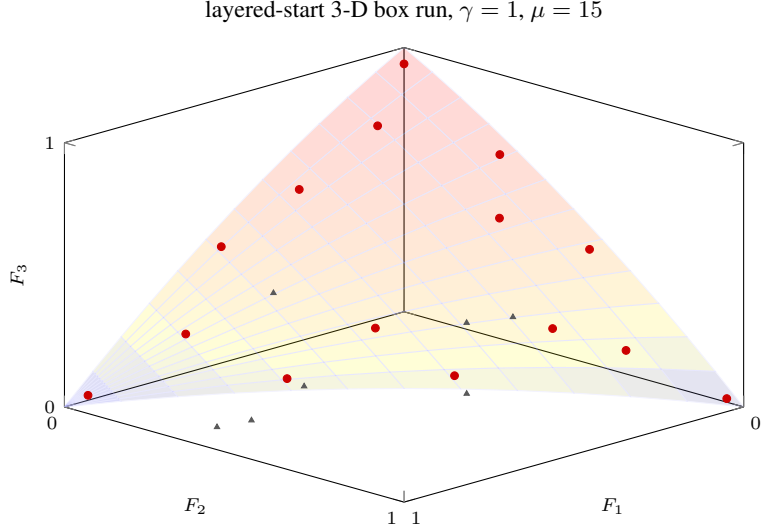
\begin{figure}[htbp]
\centering
\begin{tikzpicture}
\begin{axis}[
  width=0.64\linewidth,
  height=0.46\linewidth,
  view={135}{27},
  xlabel={$F_1$}, ylabel={$F_2$}, zlabel={$F_3$},
  xmin=0,xmax=1,ymin=0,ymax=1,zmin=0,zmax=1,
  xtick={0,1}, ytick={0,1}, ztick={0,1},
  tick label style={font=\scriptsize},
  label style={font=\scriptsize},
  title={layered-start 3-D box run, $\gamma=1$, $\mu=15$},
  title style={font=\small},
]
\addplot3[surf,shader=flat,opacity=0.16,samples=13,samples y=13,domain=0:1,y domain=0:1,draw=blue!30,fill opacity=0.16]
({1-((((x-1)^2+((1-x)*y)^2+((1-x)*(1-y))^2)/2)^1.0)},
 {1-(((x^2+(((1-x)*y)-1)^2+((1-x)*(1-y))^2)/2)^1.0)},
 {1-(((x^2+((1-x)*y)^2+(((1-x)*(1-y))-1)^2)/2)^1.0)});
\addplot3[only marks,mark=triangle*,mark size=1.15pt,gray!70!black] coordinates {\ConvThreeDFarInitial};
\addplot3[only marks,mark=*,mark size=1.45pt,red!80!black] coordinates {\ConvThreeDFarFinal};
\end{axis}
\end{tikzpicture}
\caption{Layered-start three-objective run in objective space.  The gray triangles are the initial objective vectors produced by random decision points sampled mildly outside the natural front range, from $[-0.25,1.25]^3$ and then projected to $[-0.4,1.4]^3$.  The red circles are the final points after $45$ projected ascent iterations.  The layer profile is initially $8+5+2$, then $11+4$, and becomes a single nondominated layer only after several ascent steps.}
\label{fig:convergence-layered-start-3d}
\end{figure}

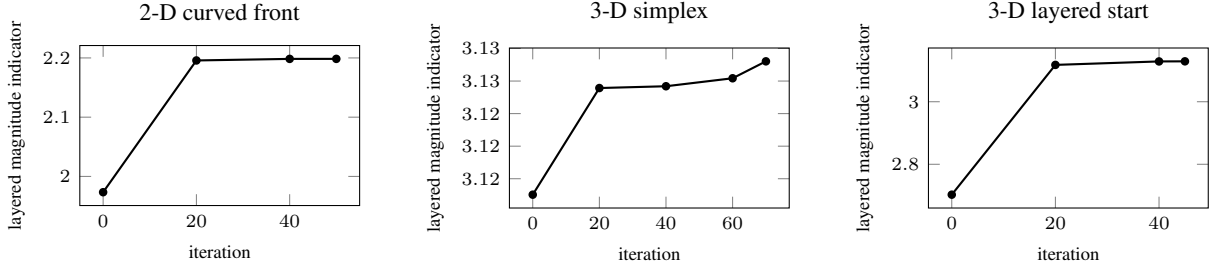
\begin{figure}[htbp]
\centering
\begin{minipage}{0.32\linewidth}
\centering
\begin{tikzpicture}
\begin{axis}[
  width=\linewidth,
  height=0.70\linewidth,
  xlabel={iteration}, ylabel={layered magnitude indicator},
  tick label style={font=\scriptsize},
  label style={font=\scriptsize},
  title={2-D curved front},
  title style={font=\small},
]
\addplot[mark=*,mark size=1.2pt,thick] coordinates {\ConvTwoDCoords};
\end{axis}
\end{tikzpicture}
\end{minipage}\hfill
\begin{minipage}{0.32\linewidth}
\centering
\begin{tikzpicture}
\begin{axis}[
  width=\linewidth,
  height=0.70\linewidth,
  xlabel={iteration}, ylabel={layered magnitude indicator},
  tick label style={font=\scriptsize},
  label style={font=\scriptsize},
  title={3-D simplex},
  title style={font=\small},
]
\addplot[mark=*,mark size=1.2pt,thick] coordinates {\ConvThreeDCoords};
\end{axis}
\end{tikzpicture}
\end{minipage}\hfill
\begin{minipage}{0.32\linewidth}
\centering
\begin{tikzpicture}
\begin{axis}[
  width=\linewidth,
  height=0.70\linewidth,
  xlabel={iteration}, ylabel={layered magnitude indicator},
  tick label style={font=\scriptsize},
  label style={font=\scriptsize},
  title={3-D layered start},
  title style={font=\small},
]
\addplot[mark=*,mark size=1.2pt,thick] coordinates {\ConvThreeDFarCoords};
\end{axis}
\end{tikzpicture}
\end{minipage}
\caption{Layered magnitude growth for the three baseline representative runs.  Only every twentieth recorded iteration is plotted, with the final endpoint added when it is not a multiple of twenty.  The layered-start three-objective run shows the largest growth because several points are initially dominated, but the initial cloud is only mildly outside the natural Pareto-front range.}
\label{fig:convergence-growth}
\end{figure}

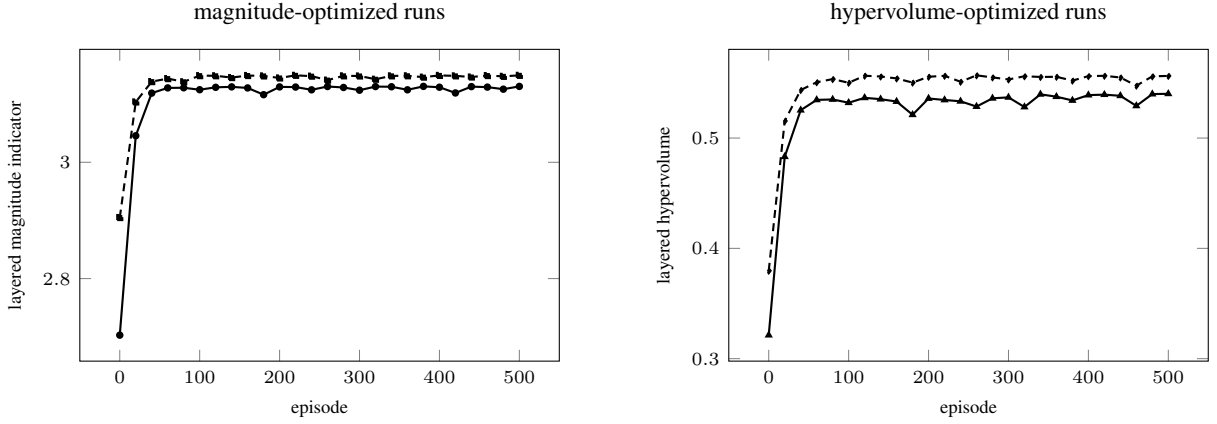
\begin{figure}[htbp]
\centering
\begin{minipage}{0.48\linewidth}
\centering
\begin{tikzpicture}
\begin{axis}[
  width=\linewidth,
  height=0.72\linewidth,
  xlabel={episode}, ylabel={layered magnitude indicator},
  tick label style={font=\scriptsize},
  label style={font=\scriptsize},
  title={magnitude-optimized runs},
  title style={font=\small},
]
\addplot[mark=*,mark size=1.0pt,thick] coordinates {\ConvThreeDRecoveryCoords};
\addplot[mark=square*,mark size=0.95pt,thick,densely dashed] coordinates {\ConvThreeDHFiveRecoveryCoords};
\end{axis}
\end{tikzpicture}
\end{minipage}\hfill
\begin{minipage}{0.48\linewidth}
\centering
\begin{tikzpicture}
\begin{axis}[
  width=\linewidth,
  height=0.72\linewidth,
  xlabel={episode}, ylabel={layered hypervolume},
  tick label style={font=\scriptsize},
  label style={font=\scriptsize},
  title={hypervolume-optimized runs},
  title style={font=\small},
]
\addplot[mark=triangle*,mark size=1.0pt,thick] coordinates {\ConvThreeDRecoveryHVCoords};
\addplot[mark=diamond*,mark size=0.95pt,thick,densely dashed] coordinates {\ConvThreeDHFiveRecoveryHVCoords};
\end{axis}
\end{tikzpicture}
\end{minipage}

\caption{Longer 500-episode 3-D layered-start runs with stagnation recovery.  The left panel optimizes the layered magnitude indicator and the right panel optimizes the layered hypervolume indicator.  In both panels, the solid curve is the $H=4$ budget with $\mu=15$ and the dashed curve is the $H=5$ budget with $\mu=21$.  The plotted curves use every twentieth recorded episode and include the endpoint at episode 500.  After a ten-episode window with indicator growth below $5\cdot 10^{-3}$, the algorithm applies a larger multi-point perturbation, allows a temporary backtracking drop in the optimized indicator, and then resumes gradient ascent.  Perturbations are disabled in the final ten episodes.}
\label{fig:convergence-stochastic-recovery}
\end{figure}

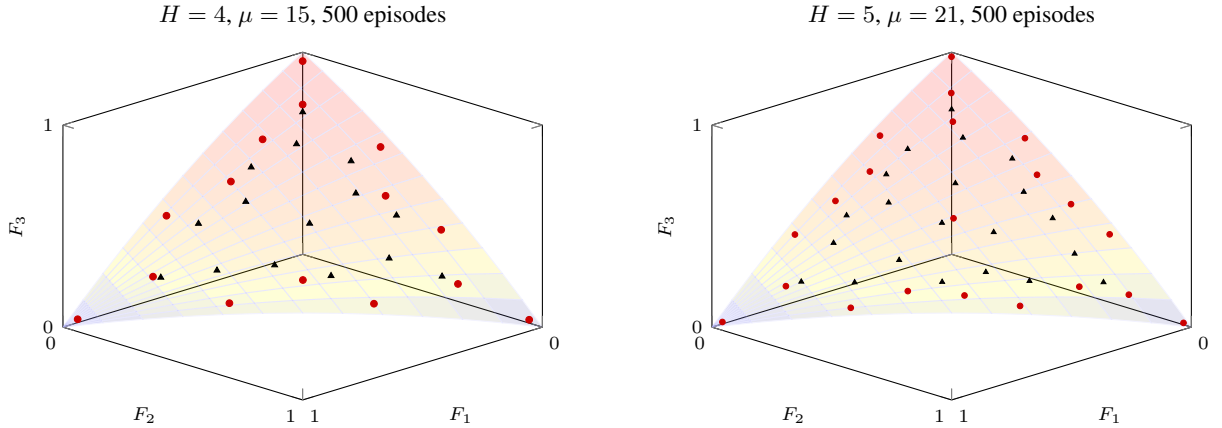
\begin{figure}[htbp]
\centering
\begin{minipage}{0.48\linewidth}
\centering
\begin{tikzpicture}
\begin{axis}[
  width=\linewidth,
  height=0.78\linewidth,
  view={135}{27},
  xlabel={$F_1$}, ylabel={$F_2$}, zlabel={$F_3$},
  xmin=0,xmax=1,ymin=0,ymax=1,zmin=0,zmax=1,
  xtick={0,1}, ytick={0,1}, ztick={0,1},
  tick label style={font=\scriptsize},
  label style={font=\scriptsize},
  title={$H=4$, $\mu=15$, 500 episodes},
  title style={font=\small},
]
\addplot3[surf,shader=flat,opacity=0.16,samples=13,samples y=13,domain=0:1,y domain=0:1,draw=blue!30,fill opacity=0.16]
({1-((((x-1)^2+((1-x)*y)^2+((1-x)*(1-y))^2)/2)^1.0)},
 {1-(((x^2+(((1-x)*y)-1)^2+((1-x)*(1-y))^2)/2)^1.0)},
 {1-(((x^2+((1-x)*y)^2+(((1-x)*(1-y))-1)^2)/2)^1.0)});
\addplot3[only marks,mark=*,mark size=1.25pt,red!80!black] coordinates {\ConvThreeDRecoveryFinal};
\addplot3[only marks,mark=triangle*,mark size=1.25pt,black] coordinates {\ConvThreeDRecoveryHVFinal};
\end{axis}
\end{tikzpicture}
\end{minipage}\hfill
\begin{minipage}{0.48\linewidth}
\centering
\begin{tikzpicture}
\begin{axis}[
  width=\linewidth,
  height=0.78\linewidth,
  view={135}{27},
  xlabel={$F_1$}, ylabel={$F_2$}, zlabel={$F_3$},
  xmin=0,xmax=1,ymin=0,ymax=1,zmin=0,zmax=1,
  xtick={0,1}, ytick={0,1}, ztick={0,1},
  tick label style={font=\scriptsize},
  label style={font=\scriptsize},
  title={$H=5$, $\mu=21$, 500 episodes},
  title style={font=\small},
]
\addplot3[surf,shader=flat,opacity=0.16,samples=13,samples y=13,domain=0:1,y domain=0:1,draw=blue!30,fill opacity=0.16]
({1-((((x-1)^2+((1-x)*y)^2+((1-x)*(1-y))^2)/2)^1.0)},
 {1-(((x^2+(((1-x)*y)-1)^2+((1-x)*(1-y))^2)/2)^1.0)},
 {1-(((x^2+((1-x)*y)^2+(((1-x)*(1-y))-1)^2)/2)^1.0)});
\addplot3[only marks,mark=*,mark size=1.05pt,red!80!black] coordinates {\ConvThreeDHFiveRecoveryFinal};
\addplot3[only marks,mark=triangle*,mark size=1.05pt,black] coordinates {\ConvThreeDHFiveRecoveryHVFinal};
\end{axis}
\end{tikzpicture}
\end{minipage}
\caption{Final objective-space approximation sets for the 500-episode recovery runs.  The transparent blue surface is the analytic supersphere Pareto front for $\gamma=1$.  Red circles are the final sets obtained by optimizing the layered magnitude indicator, and black triangles are the final sets obtained by optimizing the layered hypervolume indicator.  The hypervolume-optimized sets move more strongly toward balanced interior volume, while the magnitude-indicator-optimized sets retain more emphasis on edge and face contributions.}
\label{fig:convergence-recovery-final-sets}
\end{figure}

\begin{table}[htbp]
\centering
\small
\setlength{\tabcolsep}{3pt}
\begin{tabular}{llrrrrrr}
\hline
case & crit. & $\mu$ & iters & CPU [s] & value$_0$ & value$_f$ & growth \\
\hline
curved 2-D front & Mag & 10 & 50 & 2.70 & 1.97332 & 2.19842 & $0.2251$ $(+11.41\%)$ \\
3-D supersphere, $\gamma=1$, $H=4$ & Mag & 15 & 70 & 0.50 & 3.11902 & 3.12721 & $0.00819$ $(+0.26\%)$ \\
3-D layered box, $\gamma=1$ & Mag & 15 & 45 & 0.13 & 2.70275 & 3.12922 & $0.42648$ $(+15.78\%)$ \\
3-D recovery box, $\gamma=1$ & Mag & 15 & 500 & 1.75 & 2.70275 & 3.13047 & $0.42772$ $(+15.83\%)$ \\
3-D recovery box, $\gamma=1$, $H=5$ & Mag & 21 & 500 & 3.12 & 2.90496 & 3.14948 & $0.24452$ $(+8.42\%)$ \\
3-D recovery box, $\gamma=1$ & HV & 15 & 500 & 1.52 & 0.32132 & 0.54011 & $0.21879$ $(+68.09\%)$ \\
3-D recovery box, $\gamma=1$, $H=5$ & HV & 21 & 500 & 2.66 & 0.37968 & 0.55632 & $0.17664$ $(+46.52\%)$ \\
\hline
\end{tabular}
\caption{Overall optimized-indicator growth in the convergence runs.  Rows marked Mag report the layered magnitude indicator; rows marked HV report layered hypervolume.  The simplex-constrained 3-D run starts already on the front, whereas the layered-start box runs begin with several dominated layers and therefore have a much larger indicator increase.  The four recovery rows use the same multi-point backtracking perturbation rule for 500 episodes.  CPU time is reported in process seconds measured by the reproduction script on the execution machine.}
\label{tab:convergence-growth-summary}
\end{table}

\begin{table}[htbp]
\centering
\small
\setlength{\tabcolsep}{4pt}
\begin{tabular}{lrrrl}
\hline
case & $\mu$ & iter. & $\Mag^\ell$ & layer sizes \\
\hline
2-D curved front & 10 & 0 & 1.97332 & $4+5+1$ \\
2-D curved front & 10 & 20 & 2.19575 & $10$ \\
2-D curved front & 10 & 40 & 2.19838 & $10$ \\
2-D curved front & 10 & 50 & 2.19842 & $10$ \\
3-D supersphere & 15 & 0 & 3.11902 & $15$ \\
3-D supersphere & 15 & 20 & 3.12557 & $15$ \\
3-D supersphere & 15 & 40 & 3.12568 & $15$ \\
3-D supersphere & 15 & 60 & 3.12617 & $15$ \\
3-D supersphere & 15 & 70 & 3.12721 & $15$ \\
3-D layered box & 15 & 0 & 2.70275 & $8+5+2$ \\
3-D layered box & 15 & 20 & 3.11784 & $15$ \\
3-D layered box & 15 & 40 & 3.12891 & $15$ \\
3-D layered box & 15 & 45 & 3.12922 & $15$ \\
\hline
\end{tabular}
\caption{Layer-size snapshots for the baseline convergence runs.  Rows are reported every twentieth iteration, with the endpoint included explicitly.  The notation $8+5+2$ means eight points in the first nondominated layer, five points in the second layer, and two points in the third.}
\label{tab:convergence-layer-snapshots}
\end{table}

\begin{table}[htbp]
\centering
\small
\setlength{\tabcolsep}{4pt}
\begin{tabular}{rrrrl}
\hline
iter. & $\Mag^\ell$ $(H=4)$ & layers $(H=4)$ & $\Mag^\ell$ $(H=5)$ & layers $(H=5)$ \\
\hline
0 & 2.70275 & $8+5+2$ & 2.90496 & $11+8+2$ \\
20 & 3.04563 & $11+4$ & 3.10371 & $17+4$ \\
40 & 3.11894 & $15$ & 3.1387 & $20+1$ \\
60 & 3.12791 & $15$ & 3.14385 & $21$ \\
80 & 3.12811 & $15$ & 3.13844 & $21$ \\
100 & 3.12476 & $15$ & 3.14895 & $21$ \\
120 & 3.12891 & $15$ & 3.14875 & $21$ \\
140 & 3.12968 & $15$ & 3.14565 & $21$ \\
160 & 3.12785 & $15$ & 3.14916 & $21$ \\
180 & 3.11622 & $15$ & 3.14847 & $21$ \\
200 & 3.1296 & $15$ & 3.14497 & $21$ \\
220 & 3.12935 & $15$ & 3.1495 & $21$ \\
240 & 3.12462 & $15$ & 3.14809 & $21$ \\
260 & 3.13029 & $15$ & 3.14128 & $21$ \\
280 & 3.12864 & $15$ & 3.14861 & $21$ \\
300 & 3.12394 & $15$ & 3.1484 & $21$ \\
320 & 3.13017 & $15$ & 3.1428 & $21$ \\
340 & 3.12986 & $15$ & 3.14859 & $21$ \\
360 & 3.12462 & $15$ & 3.1487 & $21$ \\
380 & 3.13051 & $15$ & 3.14641 & $21$ \\
400 & 3.12905 & $15$ & 3.14935 & $21$ \\
420 & 3.11921 & $15$ & 3.14858 & $21$ \\
440 & 3.13007 & $15$ & 3.14662 & $21$ \\
460 & 3.12925 & $15$ & 3.14911 & $21$ \\
480 & 3.12569 & $15$ & 3.14776 & $21$ \\
500 & 3.13047 & $15$ & 3.14948 & $21$ \\
\hline
\end{tabular}
\caption{Layer-size snapshots for the 500-episode recovery runs, sampled every twentieth episode.  The $H=4$ run uses $\mu=15$ points and the $H=5$ run uses $\mu=21$ points.  Both runs start with multiple layers and then reach a single nondominated layer.}
\label{tab:convergence-recovery-layer-snapshots}
\end{table}

\begin{table}[htbp]
\centering
\small
\setlength{\tabcolsep}{4pt}
\begin{tabular}{rrrrl}
\hline
iter. & $\HV^\ell$ $(H=4)$ & layers $(H=4)$ & $\HV^\ell$ $(H=5)$ & layers $(H=5)$ \\
\hline
0 & 0.32132 & $8+5+2$ & 0.37968 & $11+8+2$ \\
20 & 0.48307 & $11+4$ & 0.51546 & $16+5$ \\
40 & 0.52525 & $13+1+1$ & 0.54353 & $18+2+1$ \\
60 & 0.53459 & $14+1$ & 0.55058 & $20+1$ \\
80 & 0.53498 & $14+1$ & 0.55344 & $20+1$ \\
100 & 0.53202 & $14+1$ & 0.5501 & $21$ \\
120 & 0.53646 & $14+1$ & 0.55649 & $21$ \\
140 & 0.53524 & $14+1$ & 0.55579 & $21$ \\
160 & 0.53309 & $14+1$ & 0.55405 & $21$ \\
180 & 0.52111 & $14+1$ & 0.55025 & $21$ \\
200 & 0.53586 & $14+1$ & 0.55556 & $21$ \\
220 & 0.53452 & $14+1$ & 0.55642 & $21$ \\
240 & 0.5333 & $14+1$ & 0.55101 & $21$ \\
260 & 0.52854 & $14+1$ & 0.55684 & $21$ \\
280 & 0.53602 & $14+1$ & 0.55497 & $21$ \\
300 & 0.53687 & $15$ & 0.55312 & $21$ \\
320 & 0.52813 & $15$ & 0.55602 & $21$ \\
340 & 0.53963 & $15$ & 0.55538 & $21$ \\
360 & 0.53749 & $15$ & 0.55553 & $21$ \\
380 & 0.53408 & $15$ & 0.55206 & $21$ \\
400 & 0.53894 & $15$ & 0.55613 & $21$ \\
420 & 0.53933 & $15$ & 0.55642 & $21$ \\
440 & 0.53831 & $15$ & 0.55489 & $21$ \\
460 & 0.52916 & $15$ & 0.54747 & $21$ \\
480 & 0.53985 & $15$ & 0.55593 & $21$ \\
500 & 0.54011 & $15$ & 0.55632 & $21$ \\
\hline
\end{tabular}
\caption{Layer-size snapshots for the 500-episode hypervolume-optimized recovery runs, sampled every twentieth episode.  The $H=4$ run uses $\mu=15$ points and the $H=5$ run uses $\mu=21$ points.  These runs use the same starting clouds and perturbation/backtracking strategy as the magnitude-optimized recovery runs, but optimize layered hypervolume.}
\label{tab:convergence-hv-recovery-layer-snapshots}
\end{table}

The traces illustrate the same layer mechanism under both indicators.  In two objectives the first phase is dominated by promotion through nondomination layers: the profile $4+5+1$ has become a single first layer by the first sampled checkpoint at iteration $20$.  In the simplex-constrained supersphere run, all points are mutually nondominated from the beginning and the magnitude-indicator increase reflects redistribution on the surface.  In the layered-start box run, several points begin just outside the nondominated range.  The layer profile remains multi-layered initially and reaches a single nondominated layer by the first sampled checkpoint at iteration $20$.  The overall layered magnitude indicator grows by about $15.8\%$ and then continues with a shorter redistribution phase on the front.

The 500-episode variants use the same initial layer structures but replace the small accept-only restart by a larger backtracking perturbation.  The convergence curves therefore show small nonmonotone drops after some stagnation windows; these drops are deliberate and are followed by gradient-only recovery phases.  In the magnitude-indicator-optimized $H=4$ run, the final layered magnitude-indicator value is $3.13047$, slightly above the 45-episode deterministic run.  In the magnitude-indicator-optimized $H=5$ run, the larger budget reaches a single nondominated layer and a final magnitude-indicator value $3.14948$.  The hypervolume-optimized runs show the corresponding behavior for the volume term: the final layered hypervolume values are $0.54011$ for $H=4$ and $0.55632$ for $H=5$.  The absence of perturbations in the last ten episodes ensures that each reported endpoint is produced by a final deterministic polishing phase rather than by a last random move.

Figure~\ref{fig:h5-iter30-magnitude-vectorfield} complements the sampled convergence traces with a single instantaneous vector-field snapshot. It shows the $H=5$ magnitude-indicator-optimized run at iteration~30, where the layer profile is still $18+3$. The left panel displays the decision-space ascent directions obtained by pulling the layered magnitude-indicator gradient back through the objective Jacobians and normalizing rowwise before projection. The right panel shows the corresponding objective-space gradient directions. Thus the figure separates two geometries of the same update: the objective-space arrows show the local layered magnitude-indicator ascent direction, while the decision-space arrows show the feasible pullback direction available to the projected optimizer on the supersphere decision domain.

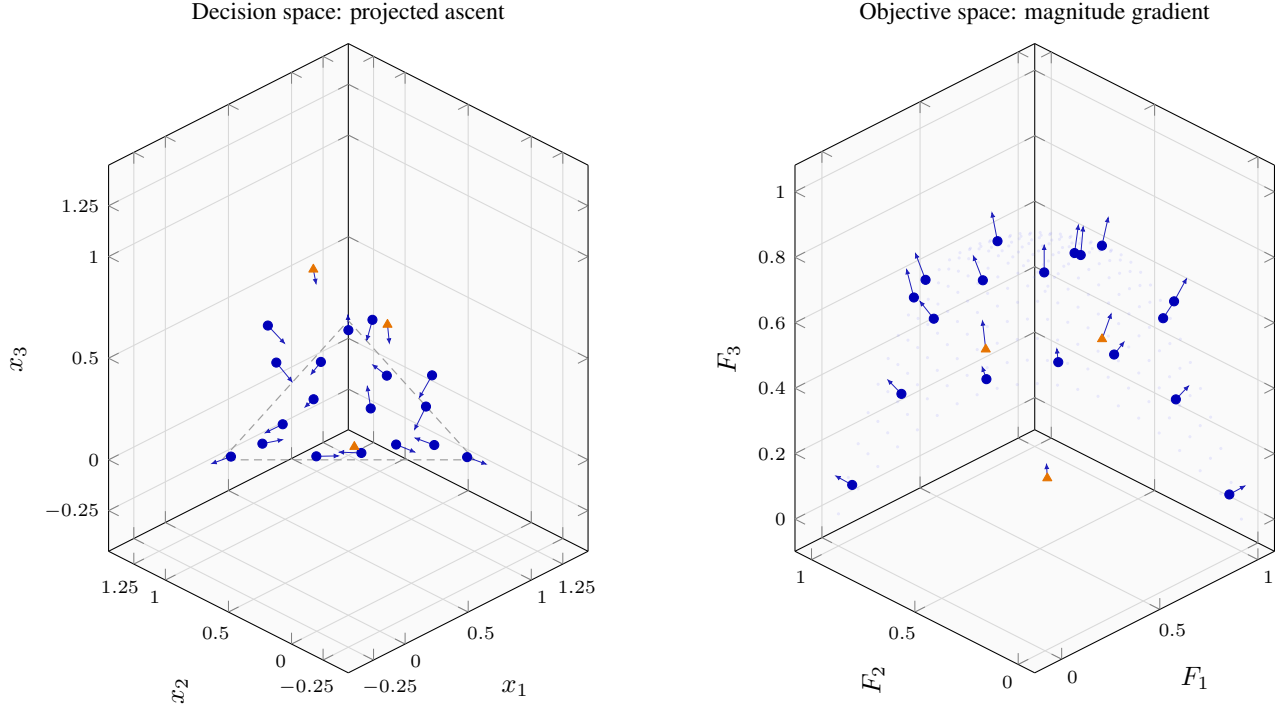
\begin{figure}[htbp]
\centering
\begin{tikzpicture}
\pgfplotsset{
  vectorfield axis/.style={
    width=0.48\linewidth,
    height=0.60\linewidth,
    view={315}{24},
    tick label style={font=\scriptsize},
    label style={font=\scriptsize},
    title style={font=\small, yshift=-0.5ex},
    grid=major,
    grid style={gray!30},
    axis background/.style={fill=gray!4},
    axis lines=box,
    clip=false,
    every axis plot/.append style={line join=round},
  }
}
\begin{axis}[
  vectorfield axis,
  name=decaxis,
  at={(0cm,0cm)},
  xlabel={$x_1$}, ylabel={$x_2$}, zlabel={$x_3$},
  title={Decision space: projected ascent},
  xmin=-0.45,xmax=1.45,
  ymin=-0.45,ymax=1.45,
  zmin=-0.45,zmax=1.45,
  xtick={-0.25,0,0.5,1.0,1.25},
  ytick={-0.25,0,0.5,1.0,1.25},
  ztick={-0.25,0,0.5,1.0,1.25},
]
\addplot3[densely dashed, gray!75, line width=0.45pt] coordinates {(1,0,0) (0,1,0) (0,0,1) (1,0,0)};
\addplot3[only marks, mark=*, mark size=1.70pt, blue!72!black] coordinates {(0.78748,0.10626,0.10626) (-0.0541,0.58359,0.8089) (0.25675,0.53212,0.36454) (0.63873,-0.02496,0.53736) (0.10616,0.78724,0.11306) (0.36032,0.61385,0.02539) (0.03957,0.257,0.70328) (0.31026,0.11901,0.86861) (0.63453,0.25569,0.1098) (0.52739,0.42378,0.04918) (0.11227,0.63226,0.25547) (0.35023,0.04496,0.60462) (0.66823,0.05269,0.34998) (0.00968,0.58036,0.60756) (0.02329,0.95349,0.02322) (0.03348,0.033,0.93199) (0.96104,0.01948,0.01948) (0.35947,0.18197,0.39718)};
\addplot3[only marks, mark=triangle*, mark size=1.95pt, orange!90!black] coordinates {(0.7056,0.65793,-0.05024) (0.7654,0.45499,0.59597) (0.46513,0.74231,0.87153)};
\draw[-{Latex[length=0.75mm,width=0.60mm]}, blue!70!black, line width=0.24pt, opacity=0.88] (axis cs:0.78748,0.10626,0.10626) -- (axis cs:0.68134,0.15933,0.15934);
\draw[-{Latex[length=0.75mm,width=0.60mm]}, blue!70!black, line width=0.24pt, opacity=0.88] (axis cs:-0.0541,0.58359,0.8089) -- (axis cs:-0.02349,0.47662,0.74167);
\draw[-{Latex[length=0.75mm,width=0.60mm]}, blue!70!black, line width=0.24pt, opacity=0.88] (axis cs:0.25675,0.53212,0.36454) -- (axis cs:0.13595,0.48531,0.3753);
\draw[-{Latex[length=0.75mm,width=0.60mm]}, blue!70!black, line width=0.24pt, opacity=0.88] (axis cs:0.7056,0.65793,-0.05024) -- (axis cs:0.60913,0.58374,-0.00452);
\draw[-{Latex[length=0.75mm,width=0.60mm]}, blue!70!black, line width=0.24pt, opacity=0.88] (axis cs:0.63873,-0.02496,0.53736) -- (axis cs:0.57663,0.01634,0.43089);
\draw[-{Latex[length=0.75mm,width=0.60mm]}, blue!70!black, line width=0.24pt, opacity=0.88] (axis cs:0.10616,0.78724,0.11306) -- (axis cs:0.16263,0.67641,0.15083);
\draw[-{Latex[length=0.75mm,width=0.60mm]}, blue!70!black, line width=0.24pt, opacity=0.88] (axis cs:0.36032,0.61385,0.02539) -- (axis cs:0.45128,0.52102,0.02815);
\draw[-{Latex[length=0.75mm,width=0.60mm]}, blue!70!black, line width=0.24pt, opacity=0.88] (axis cs:0.03957,0.257,0.70328) -- (axis cs:0.04594,0.34587,0.60862);
\draw[-{Latex[length=0.75mm,width=0.60mm]}, blue!70!black, line width=0.24pt, opacity=0.88] (axis cs:0.31026,0.11901,0.86861) -- (axis cs:0.20667,0.06241,0.81415);
\draw[-{Latex[length=0.75mm,width=0.60mm]}, blue!70!black, line width=0.24pt, opacity=0.88] (axis cs:0.7654,0.45499,0.59597) -- (axis cs:0.68664,0.36036,0.55424);
\draw[-{Latex[length=0.75mm,width=0.60mm]}, blue!70!black, line width=0.24pt, opacity=0.88] (axis cs:0.63453,0.25569,0.1098) -- (axis cs:0.74056,0.20675,0.05268);
\draw[-{Latex[length=0.75mm,width=0.60mm]}, blue!70!black, line width=0.24pt, opacity=0.88] (axis cs:0.52739,0.42378,0.04918) -- (axis cs:0.43231,0.51221,0.05538);
\draw[-{Latex[length=0.75mm,width=0.60mm]}, blue!70!black, line width=0.24pt, opacity=0.88] (axis cs:0.46513,0.74231,0.87153) -- (axis cs:0.38636,0.64051,0.85337);
\draw[-{Latex[length=0.75mm,width=0.60mm]}, blue!70!black, line width=0.24pt, opacity=0.88] (axis cs:0.11227,0.63226,0.25547) -- (axis cs:0.07197,0.73745,0.19058);
\draw[-{Latex[length=0.75mm,width=0.60mm]}, blue!70!black, line width=0.24pt, opacity=0.88] (axis cs:0.35023,0.04496,0.60462) -- (axis cs:0.25073,0.06377,0.68615);
\draw[-{Latex[length=0.75mm,width=0.60mm]}, blue!70!black, line width=0.24pt, opacity=0.88] (axis cs:0.66823,0.05269,0.34998) -- (axis cs:0.56988,0.048,0.2651);
\draw[-{Latex[length=0.75mm,width=0.60mm]}, blue!70!black, line width=0.24pt, opacity=0.88] (axis cs:0.00968,0.58036,0.60756) -- (axis cs:0.03572,0.4765,0.53384);
\draw[-{Latex[length=0.75mm,width=0.60mm]}, blue!70!black, line width=0.24pt, opacity=0.88] (axis cs:0.02329,0.95349,0.02322) -- (axis cs:-0.03,1.05964,-0.02963);
\draw[-{Latex[length=0.75mm,width=0.60mm]}, blue!70!black, line width=0.24pt, opacity=0.88] (axis cs:0.03348,0.033,0.93199) -- (axis cs:-0.01732,-0.01388,1.04209);
\draw[-{Latex[length=0.75mm,width=0.60mm]}, blue!70!black, line width=0.24pt, opacity=0.88] (axis cs:0.96104,0.01948,0.01948) -- (axis cs:1.06719,-0.03359,-0.03359);
\draw[-{Latex[length=0.75mm,width=0.60mm]}, blue!70!black, line width=0.24pt, opacity=0.88] (axis cs:0.35947,0.18197,0.39718) -- (axis cs:0.32629,0.17707,0.52278);
\end{axis}
\begin{axis}[
  vectorfield axis,
  at={(0.55\linewidth,0cm)},
  anchor=south west,
  xlabel={$F_1$}, ylabel={$F_2$}, zlabel={$F_3$},
  title={Objective space: magnitude gradient},
  xmin=-0.13726, xmax=1.08424,
  ymin=-0.08, ymax=1.08,
  zmin=-0.09823, zmax=1.08135,
]
\addplot3[only marks, mark=*, mark size=0.38pt, blue!45, opacity=0.16] coordinates {(0,0,1) (0.0475,0.0975,0.9975) (0.09,0.19,0.99) (0.1275,0.2775,0.9775) (0.16,0.36,0.96) (0.1875,0.4375,0.9375) (0.21,0.51,0.91) (0.2275,0.5775,0.8775) (0.24,0.64,0.84) (0.2475,0.6975,0.7975) (0.25,0.75,0.75) (0.2475,0.7975,0.6975) (0.24,0.84,0.64) (0.2275,0.8775,0.5775) (0.21,0.91,0.51) (0.1875,0.9375,0.4375) (0.16,0.96,0.36) (0.1275,0.9775,0.2775) (0.09,0.99,0.19) (0.0475,0.9975,0.0975) (0,1,0) (0.0975,0.0475,0.9975) (0.1425,0.1425,0.9925) (0.1825,0.2325,0.9825) (0.2175,0.3175,0.9675) (0.2475,0.3975,0.9475) (0.2725,0.4725,0.9225) (0.2925,0.5425,0.8925) (0.3075,0.6075,0.8575) (0.3175,0.6675,0.8175) (0.3225,0.7225,0.7725) (0.3225,0.7725,0.7225) (0.3175,0.8175,0.6675) (0.3075,0.8575,0.6075) (0.2925,0.8925,0.5425) (0.2725,0.9225,0.4725) (0.2475,0.9475,0.3975) (0.2175,0.9675,0.3175) (0.1825,0.9825,0.2325) (0.1425,0.9925,0.1425) (0.0975,0.9975,0.0475) (0.19,0.09,0.99) (0.2325,0.1825,0.9825) (0.27,0.27,0.97) (0.3025,0.3525,0.9525) (0.33,0.43,0.93) (0.3525,0.5025,0.9025) (0.37,0.57,0.87) (0.3825,0.6325,0.8325) (0.39,0.69,0.79) (0.3925,0.7425,0.7425) (0.39,0.79,0.69) (0.3825,0.8325,0.6325) (0.37,0.87,0.57) (0.3525,0.9025,0.5025) (0.33,0.93,0.43) (0.3025,0.9525,0.3525) (0.27,0.97,0.27) (0.2325,0.9825,0.1825) (0.19,0.99,0.09) (0.2775,0.1275,0.9775) (0.3175,0.2175,0.9675) (0.3525,0.3025,0.9525) (0.3825,0.3825,0.9325) (0.4075,0.4575,0.9075) (0.4275,0.5275,0.8775) (0.4425,0.5925,0.8425) (0.4525,0.6525,0.8025) (0.4575,0.7075,0.7575) (0.4575,0.7575,0.7075) (0.4525,0.8025,0.6525) (0.4425,0.8425,0.5925) (0.4275,0.8775,0.5275) (0.4075,0.9075,0.4575) (0.3825,0.9325,0.3825) (0.3525,0.9525,0.3025) (0.3175,0.9675,0.2175) (0.2775,0.9775,0.1275) (0.36,0.16,0.96) (0.3975,0.2475,0.9475) (0.43,0.33,0.93) (0.4575,0.4075,0.9075) (0.48,0.48,0.88) (0.4975,0.5475,0.8475) (0.51,0.61,0.81) (0.5175,0.6675,0.7675) (0.52,0.72,0.72) (0.5175,0.7675,0.6675) (0.51,0.81,0.61) (0.4975,0.8475,0.5475) (0.48,0.88,0.48) (0.4575,0.9075,0.4075) (0.43,0.93,0.33) (0.3975,0.9475,0.2475) (0.36,0.96,0.16) (0.4375,0.1875,0.9375) (0.4725,0.2725,0.9225) (0.5025,0.3525,0.9025) (0.5275,0.4275,0.8775) (0.5475,0.4975,0.8475) (0.5625,0.5625,0.8125) (0.5725,0.6225,0.7725) (0.5775,0.6775,0.7275) (0.5775,0.7275,0.6775) (0.5725,0.7725,0.6225) (0.5625,0.8125,0.5625) (0.5475,0.8475,0.4975) (0.5275,0.8775,0.4275) (0.5025,0.9025,0.3525) (0.4725,0.9225,0.2725) (0.4375,0.9375,0.1875) (0.51,0.21,0.91) (0.5425,0.2925,0.8925) (0.57,0.37,0.87) (0.5925,0.4425,0.8425) (0.61,0.51,0.81) (0.6225,0.5725,0.7725) (0.63,0.63,0.73) (0.6325,0.6825,0.6825) (0.63,0.73,0.63) (0.6225,0.7725,0.5725) (0.61,0.81,0.51) (0.5925,0.8425,0.4425) (0.57,0.87,0.37) (0.5425,0.8925,0.2925) (0.51,0.91,0.21) (0.5775,0.2275,0.8775) (0.6075,0.3075,0.8575) (0.6325,0.3825,0.8325) (0.6525,0.4525,0.8025) (0.6675,0.5175,0.7675) (0.6775,0.5775,0.7275) (0.6825,0.6325,0.6825) (0.6825,0.6825,0.6325) (0.6775,0.7275,0.5775) (0.6675,0.7675,0.5175) (0.6525,0.8025,0.4525) (0.6325,0.8325,0.3825) (0.6075,0.8575,0.3075) (0.5775,0.8775,0.2275) (0.64,0.24,0.84) (0.6675,0.3175,0.8175) (0.69,0.39,0.79) (0.7075,0.4575,0.7575) (0.72,0.52,0.72) (0.7275,0.5775,0.6775) (0.73,0.63,0.63) (0.7275,0.6775,0.5775) (0.72,0.72,0.52) (0.7075,0.7575,0.4575) (0.69,0.79,0.39) (0.6675,0.8175,0.3175) (0.64,0.84,0.24) (0.6975,0.2475,0.7975) (0.7225,0.3225,0.7725) (0.7425,0.3925,0.7425) (0.7575,0.4575,0.7075) (0.7675,0.5175,0.6675) (0.7725,0.5725,0.6225) (0.7725,0.6225,0.5725) (0.7675,0.6675,0.5175) (0.7575,0.7075,0.4575) (0.7425,0.7425,0.3925) (0.7225,0.7725,0.3225) (0.6975,0.7975,0.2475) (0.75,0.25,0.75) (0.7725,0.3225,0.7225) (0.79,0.39,0.69) (0.8025,0.4525,0.6525) (0.81,0.51,0.61) (0.8125,0.5625,0.5625) (0.81,0.61,0.51) (0.8025,0.6525,0.4525) (0.79,0.69,0.39) (0.7725,0.7225,0.3225) (0.75,0.75,0.25) (0.7975,0.2475,0.6975) (0.8175,0.3175,0.6675) (0.8325,0.3825,0.6325) (0.8425,0.4425,0.5925) (0.8475,0.4975,0.5475) (0.8475,0.5475,0.4975) (0.8425,0.5925,0.4425) (0.8325,0.6325,0.3825) (0.8175,0.6675,0.3175) (0.7975,0.6975,0.2475) (0.84,0.24,0.64) (0.8575,0.3075,0.6075) (0.87,0.37,0.57) (0.8775,0.4275,0.5275) (0.88,0.48,0.48) (0.8775,0.5275,0.4275) (0.87,0.57,0.37) (0.8575,0.6075,0.3075) (0.84,0.64,0.24) (0.8775,0.2275,0.5775) (0.8925,0.2925,0.5425) (0.9025,0.3525,0.5025) (0.9075,0.4075,0.4575) (0.9075,0.4575,0.4075) (0.9025,0.5025,0.3525) (0.8925,0.5425,0.2925) (0.8775,0.5775,0.2275) (0.91,0.21,0.51) (0.9225,0.2725,0.4725) (0.93,0.33,0.43) (0.9325,0.3825,0.3825) (0.93,0.43,0.33) (0.9225,0.4725,0.2725) (0.91,0.51,0.21) (0.9375,0.1875,0.4375) (0.9475,0.2475,0.3975) (0.9525,0.3025,0.3525) (0.9525,0.3525,0.3025) (0.9475,0.3975,0.2475) (0.9375,0.4375,0.1875) (0.96,0.16,0.36) (0.9675,0.2175,0.3175) (0.97,0.27,0.27) (0.9675,0.3175,0.2175) (0.96,0.36,0.16) (0.9775,0.1275,0.2775) (0.9825,0.1825,0.2325) (0.9825,0.2325,0.1825) (0.9775,0.2775,0.1275) (0.99,0.09,0.19) (0.9925,0.1425,0.1425) (0.99,0.19,0.09) (0.9975,0.0475,0.0975) (0.9975,0.0975,0.0475) (1,0,0)};
\addplot3[only marks, mark=*, mark size=1.70pt, blue!72!black] coordinates {(0.96613,0.28491,0.2849) (-0.05302,0.58468,0.80999) (0.51577,0.79114,0.62356) (0.79005,0.12636,0.68868) (0.28426,0.96534,0.29115) (0.60667,0.86021,0.27175) (0.25847,0.47589,0.92217) (0.37781,0.18656,0.93616) (0.8945,0.51566,0.36976) (0.79732,0.69371,0.3191) (0.37347,0.89345,0.51666) (0.60511,0.29983,0.8595) (0.88233,0.2668,0.56408) (0.15666,0.72734,0.75454) (0.06817,0.99838,0.06811) (0.09807,0.0976,0.99658) (0.99886,0.0573,0.0573) (0.69943,0.52193,0.73714)};
\addplot3[only marks, mark=triangle*, mark size=1.95pt, orange!90!black] coordinates {(0.73897,0.6913,-0.01688) (0.69138,0.38098,0.52195) (0.20166,0.47884,0.60806)};
\draw[-{Latex[length=0.75mm,width=0.60mm]}, blue!70!black, line width=0.24pt, opacity=0.88] (axis cs:0.96613,0.28491,0.2849) -- (axis cs:1.04889,0.29861,0.29861);
\draw[-{Latex[length=0.75mm,width=0.60mm]}, blue!70!black, line width=0.24pt, opacity=0.88] (axis cs:-0.05302,0.58468,0.80999) -- (axis cs:-0.05092,0.62326,0.8857);
\draw[-{Latex[length=0.75mm,width=0.60mm]}, blue!70!black, line width=0.24pt, opacity=0.88] (axis cs:0.51577,0.79114,0.62356) -- (axis cs:0.53448,0.8568,0.67419);
\draw[-{Latex[length=0.75mm,width=0.60mm]}, blue!70!black, line width=0.24pt, opacity=0.88] (axis cs:0.73897,0.6913,-0.01688) -- (axis cs:0.79826,0.75164,-0.00864);
\draw[-{Latex[length=0.75mm,width=0.60mm]}, blue!70!black, line width=0.24pt, opacity=0.88] (axis cs:0.79005,0.12636,0.68868) -- (axis cs:0.85875,0.12929,0.73865);
\draw[-{Latex[length=0.75mm,width=0.60mm]}, blue!70!black, line width=0.24pt, opacity=0.88] (axis cs:0.28426,0.96534,0.29115) -- (axis cs:0.29849,1.04793,0.30534);
\draw[-{Latex[length=0.75mm,width=0.60mm]}, blue!70!black, line width=0.24pt, opacity=0.88] (axis cs:0.60667,0.86021,0.27175) -- (axis cs:0.65603,0.92933,0.275);
\draw[-{Latex[length=0.75mm,width=0.60mm]}, blue!70!black, line width=0.24pt, opacity=0.88] (axis cs:0.25847,0.47589,0.92217) -- (axis cs:0.26323,0.50824,1.00063);
\draw[-{Latex[length=0.75mm,width=0.60mm]}, blue!70!black, line width=0.24pt, opacity=0.88] (axis cs:0.37781,0.18656,0.93616) -- (axis cs:0.39525,0.1908,1.01924);
\draw[-{Latex[length=0.75mm,width=0.60mm]}, blue!70!black, line width=0.24pt, opacity=0.88] (axis cs:0.69138,0.38098,0.52195) -- (axis cs:0.75269,0.39358,0.57946);
\draw[-{Latex[length=0.75mm,width=0.60mm]}, blue!70!black, line width=0.24pt, opacity=0.88] (axis cs:0.8945,0.51566,0.36976) -- (axis cs:0.97513,0.54127,0.37805);
\draw[-{Latex[length=0.75mm,width=0.60mm]}, blue!70!black, line width=0.24pt, opacity=0.88] (axis cs:0.79732,0.69371,0.3191) -- (axis cs:0.85485,0.7559,0.32592);
\draw[-{Latex[length=0.75mm,width=0.60mm]}, blue!70!black, line width=0.24pt, opacity=0.88] (axis cs:0.20166,0.47884,0.60806) -- (axis cs:0.20565,0.49907,0.69052);
\draw[-{Latex[length=0.75mm,width=0.60mm]}, blue!70!black, line width=0.24pt, opacity=0.88] (axis cs:0.37347,0.89345,0.51666) -- (axis cs:0.38436,0.97354,0.54296);
\draw[-{Latex[length=0.75mm,width=0.60mm]}, blue!70!black, line width=0.24pt, opacity=0.88] (axis cs:0.60511,0.29983,0.8595) -- (axis cs:0.6445,0.30577,0.93458);
\draw[-{Latex[length=0.75mm,width=0.60mm]}, blue!70!black, line width=0.24pt, opacity=0.88] (axis cs:0.88233,0.2668,0.56408) -- (axis cs:0.95805,0.27291,0.60223);
\draw[-{Latex[length=0.75mm,width=0.60mm]}, blue!70!black, line width=0.24pt, opacity=0.88] (axis cs:0.15666,0.72734,0.75454) -- (axis cs:0.16212,0.78301,0.81854);
\draw[-{Latex[length=0.75mm,width=0.60mm]}, blue!70!black, line width=0.24pt, opacity=0.88] (axis cs:0.06817,0.99838,0.06811) -- (axis cs:0.06953,1.08336,0.06946);
\draw[-{Latex[length=0.75mm,width=0.60mm]}, blue!70!black, line width=0.24pt, opacity=0.88] (axis cs:0.09807,0.0976,0.99658) -- (axis cs:0.10052,0.10004,1.08151);
\draw[-{Latex[length=0.75mm,width=0.60mm]}, blue!70!black, line width=0.24pt, opacity=0.88] (axis cs:0.99886,0.0573,0.0573) -- (axis cs:1.08384,0.05865,0.05865);
\draw[-{Latex[length=0.75mm,width=0.60mm]}, blue!70!black, line width=0.24pt, opacity=0.88] (axis cs:0.69943,0.52193,0.73714) -- (axis cs:0.74576,0.54653,0.80402);
\end{axis}
\end{tikzpicture}

\caption{Instantaneous vector-field snapshot for the $H=5$ layered magnitude-indicator run at iteration~30. The left panel shows the projected ascent directions in decision space for the 21 decision vectors, and the right panel shows the corresponding layered magnitude-indicator gradient directions in objective space. Blue circular markers indicate points in the first nondomination layer and orange triangular markers indicate the second layer. In the decision-space panel, the dashed triangle indicates the simplex used as reference for the supersphere construction; in the objective-space panel, the pale blue point cloud indicates the image of a coarse simplex grid on the supersphere front.}
\label{fig:h5-iter30-magnitude-vectorfield}
\end{figure}

\section{Discussion}
The examples illustrate a consistent computational picture for nonsmooth set-gradient ascent of layered indicators.  In the biobjective examples the leading first-layer term drives points toward the Pareto front.  Once the points have reached that front, the lower layers disappear and the remaining task is to distribute the points along it.  In the perturbed ten-point quadratic example this behavior is especially visible: in decision space the trajectories bend toward the efficient diagonal, while in objective space their images bend toward the curved front.  The short-range repulsion term is numerically useful because it resolves temporary near-collisions and produces separated front configurations without changing the leading-order preference for first-front quality.

The regularity results clarify the mathematical status of this procedure.  On chambers with fixed layer membership one has standard Lipschitz control and a piecewise-smooth structure, so Clarke-type language is natural there.  Across layer switches, however, the hard-layer finite-$\varepsilon$ scalarization can jump, and the counterexample shows that one should not expect a global locally Lipschitz theory for this exact scalar surrogate.  For that reason, the projected finite-difference and indicator-gradient methods used in the experiments are best interpreted as practical computational schemes informed by nonsmooth geometry rather than as fully analyzed subgradient algorithms.

The comparison between hypervolume and magnitude is also informative.  Layered hypervolume keeps the familiar volume-driven interpretation but extends it to deeper nondomination layers.  Layered magnitude contains hypervolume as its highest-dimensional term and adds lower-dimensional geometric contributions such as extent and projected area.  In the three-objective supersphere experiments this difference is visible in the final approximation sets: hypervolume-optimized runs emphasize balanced dominated volume, while magnitude-optimized runs retain more boundary and face information.  Both indicators fit the same nonsmooth set-gradient framework.

From a numerical perspective, the method is already effective on small but nontrivial approximation sets.  The ten-point quadratic example is useful because the front is genuinely curved and the trajectories are visibly nonlinear in both decision and objective space.  The supersphere benchmark is useful because the Pareto front is known analytically, the curvature can be tuned by the bulge parameter, and both hypervolume and magnitude can be evaluated in the same coordinate system.  Projection also makes constraint handling transparent: for the triangle problem the feasible region is treated directly in objective space, for the quadratic problem the box constraint is enforced by coordinatewise clipping, and for the supersphere examples the decision set is either the simplex or a box.

\paragraph{Availability and reproducibility.}
The reproducibility material for this paper is maintained in the public GitHub repository
\url{https://github.com/emmerichmtm/NonSmoothAscentFinal}.
The repository contains the code, parameter settings, generated data files, and a README describing how to regenerate the figures, tables, convergence traces, and stochastic reference runs.  A lightweight interactive version of the examples is also available as a Trinket at
\url{https://trinket.io/library/trinkets/992765676638}.
Script names and file organization may evolve, so the repository README is the authoritative entry point for reproduction.

\section{Future work}
Several extensions of the present layered-ascent framework are natural.  First, the present construction handles dominated points by assigning them to deeper nondomination layers with decreasing weights.  This should be compared systematically with earlier mechanisms for steering dominated points in hypervolume-indicator gradient ascent, especially the biobjective study of dominated-point steering in HV-gradient methods~\cite{wangrendeutzemmerich2016}.  Such a comparison may clarify when explicit dominated-layer ascent is preferable to steering terms that act within a conventional hypervolume-gradient framework.

Second, the numerical treatment of step sizes deserves a more systematic analysis.  The implementation used here combines normalization, projection, accept--reject tests, and stagnation recovery.  A more principled version could build directly on step-size control strategies developed for hypervolume-indicator gradient ascent~\cite{wangdeutzbackemmerich2017}, adapted to the discontinuities introduced by hard nondomination layers and to the lower-dimensional terms in the magnitude indicator.

Third, higher-order information is a promising direction.  Efficient computation of hypervolume derivatives and Hessian entries has been studied analytically, including complexity and sparsity properties of the hypervolume-indicator Hessian~\cite{deutzemmerichwang2023hessian}.  Combining those derivative formulae with the layered construction could lead to active-set-aware second-order approximations for each fixed layer chamber.  A full second-order layered method would also connect naturally to the hypervolume Newton method for constrained multiobjective optimization~\cite{wangemmerichdeutzhernandezschutze2023newton}.  In particular, one could investigate a layered Newton or quasi-Newton ascent method in which first-layer curvature information is dominant and deeper layers provide controlled correction directions for dominated points.

\section{Conclusion}
Layered set indicators give a simple way to incorporate dominated points into a unary quality objective without allowing them to outweigh the first nondominated front.  This paper has studied the construction for a general base indicator $\mathcal I$, instantiated by two anchored dominated-set indicators: hypervolume, the top-dimensional Lebesgue measure, and magnitude, the \(\ell_1\)-magnitude.  The lexicographic and infinitesimal-style formulations encode the same hierarchy for any finite-valued base indicator, while the finite-$\varepsilon$ surrogate makes both hypervolume and magnitude-indicator versions numerically accessible.

In the biobjective setting, the layered hypervolume and magnitude-indicator scalarizations are naturally piecewise smooth: on chambers with fixed layer membership one has Lipschitz control on bounded sets, whereas across layer switches the hard-layer scalarization may be discontinuous.  This regularity picture explains why a nonsmooth set-gradient viewpoint is appropriate and also why the numerical methods used in the experiments should be described as projected finite-difference or indicator-gradient ascent schemes rather than as fully analyzed generalized-gradient methods.

The computational illustrations show that this simple framework is already effective on representative small examples, including a ten-point problem with a genuinely curved Pareto front and three-objective supersphere examples comparing layered hypervolume and the layered magnitude indicator.  Taken together, the analytic observations and numerical experiments suggest that nonsmooth set-gradient ascent of layered hypervolume and magnitude indicators is a useful framework for studying Pareto-front approximation beyond the first nondominated front.

\appendix

\section{Reference pseudocode for projected set-gradient ascent}
\label{app:algorithm}

This appendix records the numerical core used in the paper without any interactive or display-related components.  The same template applies to the layered magnitude indicator and the layered hypervolume indicator; only the base layer indicator is changed. In particular, it omits plotting, keyboard control, initialization menus, and demo-specific bookkeeping, and keeps only the ascent iteration itself.

\subsection*{Symbols and default implementation constants}
We use the following notation.
\begin{description}[leftmargin=2.8cm,style=nextline]
\item[$\mu$] number of points in the approximation set.
\item[$d$] dimension of the state space of a single point. In the biobjective examples one has $d=2$; in the supersphere benchmark one has $d=3$.
\item[$\Xi\subseteq \mathbb{R}^{d}$] feasible region for one point, with product feasible set $\Xi^{\mu}$.
\item[$\Pi_{\Xi^{\mu}}$] Euclidean projection onto $\Xi^{\mu}$. For a box constraint this is coordinatewise clipping.
\item[$X=(x^{(1)},\dots,x^{(\mu)})\in \Xi^{\mu}$] current state configuration.
\item[$\Phi:\Xi\to\mathbb{R}^{m}$] map from the state space to objective space. For objective-space experiments, $\Phi$ is the identity; for decision-space experiments, $\Phi=F$.
\item[$Y(X)=(\Phi(x^{(1)}),\dots,\Phi(x^{(\mu)}))$] objective-space image of the current configuration.
\item[$\front{\ell}(Y)$] the $\ell$th nondomination layer of $Y$.
\item[$\mathcal I(\front{\ell}(Y))$] base indicator of the $\ell$th layer; in this paper $\mathcal I$ is either $\Mag(\Dom(\cdot))$ or $\HV(\cdot)$.
\item[$R_{\sigma}(Y)$] short-range repulsion term in objective space.
\item[$J_{\varepsilon,\tau}(X)$] scalar surrogate
\[
J^{\mathcal I}_{\varepsilon,\tau}(X):=\sum_{\ell=1}^{L(X)} \varepsilon^{\ell-1}\mathcal I\!\bigl(\front{\ell}(Y(X))\bigr)-\tau R_{\sigma}(Y(X)).
\]
\item[$h$] central finite-difference radius.
\item[$\alpha$] ascent step size.
\item[$K_{\max}$] maximum number of iterations.
\item[$\delta_{\mathrm{val}}$] stopping tolerance for objective-value improvement.
\item[$\texttt{normalize\_per\_point}$] Boolean switch that normalizes each point-direction separately.
\end{description}

For the reference implementation attached to the curved-front example, the default constants are
\[
\alpha=0.005,\qquad h=10^{-6},\qquad \varepsilon=10^{-3},\qquad \tau=2\cdot 10^{-4},\qquad \sigma=0.03,
\]
with \texttt{normalize\_per\_point}=\texttt{True} and, in that specific demo, box bounds $\mathrm{BOX\_LO}=0$ and $\mathrm{BOX\_HI}=1$. Example-specific runs in the main text may override these defaults.

\begin{algorithm}[htbp]
\caption{Projected set-gradient ascent for a layered indicator surrogate}
\label{alg:layered-projected-fd}
\begin{algorithmic}[1]
\Require Feasible set $\Xi$, projection $\Pi_{\Xi^{\mu}}$, state-to-objective map $\Phi$, number of points $\mu$, parameters $\varepsilon,\tau,\sigma,h,\alpha,K_{\max},\delta_{\mathrm{val}}$, normalization flag \texttt{normalize\_per\_point}, initial configuration $X^{0}\in \Xi^{\mu}$
\Ensure Final iterate $X^{k}$ and objective-space configuration $Y(X^{k})$
\State $X^{0} \gets \Pi_{\Xi^{\mu}}(X^{0})$ \Comment{Project the initial guess to feasibility if needed}
\State $V^{0} \gets J_{\varepsilon,\tau}(X^{0})$
\For{$k=0,1,\dots,K_{\max}-1$}
    \State $G^{k} \gets 0 \in \mathbb{R}^{\mu\times d}$
    \For{$i=1,2,\dots,\mu$}
        \For{$r=1,2,\dots,d$}
            \State $E \gets 0 \in \mathbb{R}^{\mu\times d}$
            \State $E_{i,r} \gets h$ \Comment{Perturb coordinate $r$ of point $i$ only}
            \State $X^{+} \gets \Pi_{\Xi^{\mu}}(X^{k}+E)$
            \State $X^{-} \gets \Pi_{\Xi^{\mu}}(X^{k}-E)$
            \State $G^{k}_{i,r} \gets \dfrac{J_{\varepsilon,\tau}(X^{+})-J_{\varepsilon,\tau}(X^{-})}{2h}$
        \EndFor
    \EndFor
    \If{\texttt{normalize\_per\_point}}
        \For{$i=1,2,\dots,\mu$}
            \State $n_i \gets \lVert G^{k}_{i,\bullet}\rVert_2$
            \If{$n_i > 10^{-12}$}
                \State $G^{k}_{i,\bullet} \gets G^{k}_{i,\bullet}/n_i$ \Comment{Equalize the step length of the active points}
            \EndIf
        \EndFor
    \EndIf
    \If{$\lVert G^{k}\rVert_F \le 10^{-12}$}
        \State \Return $X^{k},Y(X^{k})$ \Comment{No numerically meaningful ascent direction found}
    \EndIf
    \State $\widetilde X^{k+1} \gets X^{k}+\alpha G^{k}$
    \State $X^{k+1} \gets \Pi_{\Xi^{\mu}}(\widetilde X^{k+1})$ \Comment{Respect constraints after the ascent step}
    \State $V^{k+1} \gets J_{\varepsilon,\tau}(X^{k+1})$
    \If{$|V^{k+1}-V^{k}| \le \delta_{\mathrm{val}}$}
        \State \Return $X^{k+1},Y(X^{k+1})$ \Comment{Objective improvement has become negligible}
    \EndIf
\EndFor
\State \Return $X^{K_{\max}},Y(X^{K_{\max}})$
\end{algorithmic}
\end{algorithm}

\subsection*{Walkthrough of one iteration}
The logic of Algorithm~\ref{alg:layered-projected-fd} is easiest to read from the inside out.

\paragraph{Step 1: map the current state to objective space.}
The iterate is a configuration $X^{k}=(x^{(1)},\dots,x^{(\mu)})$. If the optimization is carried out directly in objective space, then $\Phi$ is the identity and $Y(X^{k})=X^{k}$. If the optimization is carried out in decision space, each point is first mapped by the objective function, so that $Y(X^{k})=(\Phi(x^{(1)}),\dots,\Phi(x^{(\mu)}))$.

\paragraph{Step 2: evaluate the layered surrogate.}
From $Y(X^{k})$ one computes the nondomination layers $\front{1},\front{2},\dots,\front{L(X^{k})}$. Each layer contributes the chosen base indicator $\mathcal I$, with weight $1,\varepsilon,\varepsilon^{2},\dots$, and the repulsion term is subtracted.  The two choices used in the experiments are $\mathcal I(B)=\Mag(\Dom(B))$ and $\mathcal I(B)=\HV(B)$, with the origin anchor suppressed. This gives the current scalar value $J_{\varepsilon,\tau}(X^{k})$.

\paragraph{Step 3: build a coordinatewise direction field.}
For each point $i$ and each coordinate $r$, the algorithm perturbs only the entry $(i,r)$ by $\pm h$, projects the perturbed configurations back to feasibility if necessary, and then forms the central finite difference
\[
\frac{J_{\varepsilon,\tau}(X^{k}+E)-J_{\varepsilon,\tau}(X^{k}-E)}{2h}.
\]
Collecting all these values yields the matrix $G^{k}\in\mathbb{R}^{\mu\times d}$. In smooth regions this approximates the ordinary gradient of the scalar surrogate with respect to the point coordinates; near switching configurations it should be interpreted as a finite-difference numerical direction field motivated by the neighboring smooth chambers.

\paragraph{Step 4: optionally normalize pointwise.}
If \texttt{normalize\_per\_point} is enabled, then each row $G^{k}_{i,\bullet}$ is divided by its Euclidean norm whenever that norm is nonzero. This produces a diffusion-like motion in which every active point takes a step of comparable length. In the curved-front experiments this is helpful because otherwise a few points with larger local finite-difference values can dominate the entire update.

\paragraph{Step 5: take a projected ascent step.}
The tentative update is $\widetilde X^{k+1}=X^{k}+\alpha G^{k}$. Because this tentative point may violate constraints, it is projected back to the feasible set, giving $X^{k+1}=\Pi_{\Xi^{\mu}}(\widetilde X^{k+1})$. For a box constraint this is simply coordinatewise clipping; for more general convex sets it is ordinary Euclidean projection.

\paragraph{Step 6: update the objective value and stop if appropriate.}
After projection, the new scalar value $V^{k+1}=J_{\varepsilon,\tau}(X^{k+1})$ is computed. The iteration terminates if either the full direction field is numerically zero or the objective improvement $|V^{k+1}-V^{k}|$ falls below the prescribed tolerance $\delta_{\mathrm{val}}$. Otherwise the method continues.

\paragraph{How this matches the main text.}
The algorithm in the appendix is the formal version of the scheme described in Section~5. Its main ingredients are exactly the same: evaluation of the layered surrogate, symmetric finite differences taken point coordinate by point coordinate, optional pointwise normalization, a small ascent step, and projection onto the feasible set. The appendix merely makes explicit two details that were only implicit in the main text: first, the distinction between state space and objective space through the map $\Phi$; and second, the use of a concrete stopping rule and default implementation constants.

\section{The supersphere benchmark}
\label{app:supersphere-benchmark}

This appendix records the analytic form of the bulged three-objective benchmark used in Section~\ref{sec:three-objective-supersphere}.  Let
\[
\Delta_2:=\{x\in\mathbb{R}^3_{\ge 0}:x_1+x_2+x_3=1\}
\]
be the standard two-simplex and let $e_1,e_2,e_3$ denote the coordinate unit vectors.  For a curvature parameter $\gamma>0$, define
\[
F_\gamma(x)=\bigl(f_1(x),f_2(x),f_3(x)\bigr),\qquad
f_i(x)=1-\left(\frac{\|x-e_i\|_2^2}{2}\right)^\gamma,
\qquad i=1,2,3.
\]
The normalization by $2$ makes the largest squared vertex-to-vertex distance equal to one after scaling.  Hence $0\le f_i(x)\le 1$ on the simplex, $F_\gamma(e_1)=(1,0,0)$, $F_\gamma(e_2)=(0,1,0)$, and $F_\gamma(e_3)=(0,0,1)$.  At the barycenter $c=(1/3,1/3,1/3)$, one obtains
\[
F_\gamma(c)=\bigl(1-3^{-\gamma},1-3^{-\gamma},1-3^{-\gamma}\bigr),
\]
which gives a convenient scalar check on the amount of curvature introduced by $\gamma$.

A useful analytic parameterization of the Pareto-front surface is obtained by writing
\[
x_1=u,
\qquad
x_2=(1-u)v,
\qquad
x_3=(1-u)(1-v),
\qquad 0\le u,v\le 1.
\]
Then
\[
\mathcal{P}_\gamma
=
\left\{F_\gamma\bigl(u,(1-u)v,(1-u)(1-v)\bigr):0\le u,v\le 1\right\}
\]
is the analytic supersphere front used in the plots.  This parameterization covers the simplex exactly, although not uniformly with respect to surface area.  The efficient set in decision space is the simplex itself: improving one objective means moving closer to one vertex, which necessarily increases the normalized distance to at least one of the other vertices.

For reproducibility, the script also contains a box-constrained variant with the same objective formula but with $x\in[-2,2]^3$.  The simplex version is the cleaner analytic benchmark because the front admits the explicit two-parameter description above.  The box-constrained version is useful as a stress test because projection no longer enforces the simplex geometry.

\section{Reference stochastic optimizer}
\label{app:stochastic-optimizer}

The stochastic reference algorithm is related to earlier cooperative swarm-based approaches to set-oriented multiobjective optimization.  Verhoef, Deutz, and Emmerich compared deterministic set-oriented hypervolume-gradient algorithms with cooperative particle-swarm algorithms for bicriteria approximation-set construction~\cite{verhoefdeutzemmerich2018swarm}.  In such swarm-based approaches, the moving particles are interpreted collectively as an approximation set rather than as independent scalar optimizers, so their motion is assessed through a set quality criterion.  The stochastic hillclimber below is deliberately simpler: it keeps the same layered scalar objective as the gradient runs, but replaces gradient information by local random accept--reject moves.

The stochastic reference algorithm uses the same layered scalar objective as the gradient method, but it does not compute an indicator gradient.  At each iteration it perturbs one approximation-set point in a random unit direction, projects the perturbed configuration back to the feasible set, and accepts the move if the selected layered indicator value does not decrease.  It is therefore a simple projected stochastic hillclimber for the nonsmooth scalar surrogate.

\begin{algorithm}[htbp]
\caption{Projected stochastic hillclimbing reference optimizer}
\label{alg:stochastic-reference}
\begin{algorithmic}[1]
\Require Feasible set $\Xi$, projection $\Pi_{\Xi^\mu}$, state-to-objective map $\Phi$, initial configuration $X^0$, initial step size $\alpha_0$, shrink factor $\rho\in(0,1)$, floor $\alpha_{\min}$, maximum retries $R$, maximum iterations $K_{\max}$, layered indicator choice $I\in\{\Mag_3,\HV_3\}$
\State $X^0\gets \Pi_{\Xi^\mu}(X^0)$
\State $V^0\gets J_I(X^0)$ and $\alpha\gets \alpha_0$
\For{$k=0,1,\ldots,K_{\max}-1$}
    \State accepted $\gets$ false, $\alpha_{\rm trial}\gets\alpha$
    \For{$r=0,1,\ldots,R$}
        \State Choose an index $i\in\{1,\ldots,\mu\}$ uniformly at random
        \State Draw $d$ from a standard normal distribution and normalize it to $\|d\|_2=1$
        \State $\widetilde X\gets X^k$ and replace $\widetilde x^{(i)}$ by $x^{k,(i)}+\alpha_{\rm trial}d$
        \State $X^{\rm trial}\gets \Pi_{\Xi^\mu}(\widetilde X)$
        \State $V^{\rm trial}\gets J_I(X^{\rm trial})$
        \If{$V^{\rm trial}\ge V^k$}
            \State $X^{k+1}\gets X^{\rm trial}$, $V^{k+1}\gets V^{\rm trial}$, $\alpha\gets\alpha_{\rm trial}$
            \State accepted $\gets$ true
            \State \textbf{break}
        \Else
            \State $\alpha_{\rm trial}\gets\max\{\rho\alpha_{\rm trial},\alpha_{\min}\}$
        \EndIf
    \EndFor
    \If{accepted is false}
        \State $X^{k+1}\gets X^k$, $V^{k+1}\gets V^k$, and keep or recover the step size according to the stagnation rule
    \EndIf
\EndFor
\State \Return $X^{K_{\max}}$
\end{algorithmic}
\end{algorithm}

In the implementation used for the three-dimensional runs, the stochastic method shares the same projection operators, nondomination-layer computation, repulsion term, retry loop, and step-size recovery logic as the gradient method.  Its source code is part of the reproducibility package referenced in the main text, and it implements the projected stochastic hillclimber summarized in Algorithm~\ref{alg:stochastic-reference}.  Its role is mainly diagnostic: it checks whether the chosen indicator can be improved by a very simple accept--reject search even in regions where indicator gradients are expensive, unreliable, or affected by many active-set switches.


\begin{thebibliography}{99}

\bibitem{clarke1990}
F.~H.~Clarke,
\emph{Optimization and Nonsmooth Analysis},
Classics in Applied Mathematics~5, SIAM, Philadelphia, 1990.

\bibitem{augerbaderbrockhoffzitzler2012}
A.~Auger, J.~Bader, D.~Brockhoff, and E.~Zitzler,
``Hypervolume-Based Multiobjective Optimization: Theoretical Foundations and Practical Implications,''
\emph{Theoretical Computer Science} \textbf{425} (2012), 75--103.
doi:10.1016/j.tcs.2011.03.012.

\bibitem{custodioetal2011}
A.~L.~Cust{\'o}dio, J.~F.~A.~Madeira, A.~I.~F.~Vaz, and L.~N.~Vicente,
``Direct Multisearch for Multiobjective Optimization,''
\emph{SIAM Journal on Optimization} \textbf{21}(3) (2011), 1109--1140.
doi:10.1137/10079731X.

\bibitem{deb2001multiobjective}
K.~Deb,
\emph{Multi-Objective Optimization Using Evolutionary Algorithms},
Wiley, Chichester, 2001.

\bibitem{deistetal2020}
T.~M.~Deist, S.~C.~Maree, T.~Alderliesten, and P.~A.~N.~Bosman,
``Multi-objective Optimization by Uncrowded Hypervolume Gradient Ascent,''
in \emph{Parallel Problem Solving from Nature -- PPSN XVI},
Lecture Notes in Computer Science~12270, 2020, pp.~186--200.
doi:10.1007/978-3-030-58115-2\_13.

\bibitem{deutzemmerichwang2023hessian}
A.~Deutz, M.~Emmerich, and H.~Wang,
``The Hypervolume Indicator Hessian Matrix: Analytical Expression, Computational Time Complexity, and Sparsity,''
in \emph{Evolutionary Multi-Criterion Optimization},
Lecture Notes in Computer Science~13970, 2023, pp.~405--418.
doi:10.1007/978-3-031-27250-9\_29.

\bibitem{emmerichdeutzbeume2007}
M.~Emmerich, A.~Deutz, and N.~Beume,
``Gradient-Based/Evolutionary Relay Hybrid for Computing Pareto Front Approximations Maximizing the S-Metric,''
in \emph{Hybrid Metaheuristics},
Lecture Notes in Computer Science~4771, 2007, pp.~140--156.
doi:10.1007/978-3-540-75514-2\_11.

\bibitem{emmerich2026magnitudedominatedsetspareto}
M.~T.~M.~Emmerich,
\emph{The Magnitude of Dominated Sets: A Pareto Compliant Indicator Grounded in Metric Geometry},
arXiv:2604.18147 [math.OC], 2026.
doi:10.48550/arXiv.2604.18147.

\bibitem{sergeyev2017numerical}
Y.~D.~Sergeyev,
``Numerical Infinities and Infinitesimals: Methodology, Applications, and Repercussions on Two Hilbert Problems,''
\emph{EMS Surveys in Mathematical Sciences} \textbf{4}(2) (2017), 219--320.
doi:10.4171/EMSS/4-2-3.

\bibitem{miettinenmakela1993}
K.~Miettinen and M.~M.~M{\"a}kel{\"a},
``An Interactive Method for Nonsmooth Multiobjective Optimization with an Application to Optimal Control,''
\emph{Optimization Methods and Software} \textbf{2}(1) (1993), 31--44.
doi:10.1080/10556789308805533.

\bibitem{leinster2013}
T.~Leinster,
``The Magnitude of Metric Spaces,''
\emph{Documenta Mathematica} \textbf{18} (2013), 857--905.

\bibitem{lopesklamrothpaquete2025}
G.~Lopes, K.~Klamroth, and L.~Paquete,
``A Greedy Hypervolume Polychotomic Scheme for Multiobjective Combinatorial Optimization,''
\emph{Computers \& Operations Research} \textbf{183} (2025), Article~107140.
doi:10.1016/j.cor.2025.107140.

\bibitem{paqueteschulzestiglmayrlourenco2022}
L.~Paquete, B.~Schulze, M.~Stiglmayr, and A.~C.~Louren\c{c}o,
``Computing Representations Using Hypervolume Scalarizations,''
\emph{Computers \& Operations Research} \textbf{137} (2022), Article~105349.
doi:10.1016/j.cor.2021.105349.

\bibitem{emmerichdeutz2014}
M.~Emmerich and A.~Deutz,
``Time Complexity and Zeros of the Hypervolume Indicator Gradient Field,''
in \emph{EVOLVE -- A Bridge between Probability, Set Oriented Numerics, and Evolutionary Computation III},
Studies in Computational Intelligence~500, 2014, pp.~169--193.
doi:10.1007/978-3-319-01460-9\_8.

\bibitem{guerreirofonsecapaquete2021}
A.~P.~Guerreiro, C.~M.~Fonseca, and L.~Paquete,
``The Hypervolume Indicator: Computational Problems and Algorithms,''
\emph{ACM Computing Surveys} \textbf{54}(6) (2021), Article~119, 42~pp.
doi:10.1145/3453474.

\bibitem{uribeetal2020hausdorffnewton}
L.~Uribe, J.~M.~Bogoya, A.~Vargas, A.~Lara, G.~Rudolph, and O.~Sch{\"u}tze,
``A Set Based Newton Method for the Averaged Hausdorff Distance for Multi-Objective Reference Set Problems,''
\emph{Mathematics} \textbf{8}(10) (2020), Article~1822.
doi:10.3390/math8101822.

\bibitem{zitzlerthiele1999}
E.~Zitzler and L.~Thiele,
``Multiobjective Evolutionary Algorithms: A Comparative Case Study and the Strength Pareto Approach,''
\emph{IEEE Transactions on Evolutionary Computation} \textbf{3}(4) (1999), 257--271.
doi:10.1109/4235.797969.

\bibitem{verhoefdeutzemmerich2018swarm}
W.~Verhoef, A.~H.~Deutz, and M.~T.~M.~Emmerich,
``On Gradient-Based and Swarm-Based Algorithms for Set-Oriented Bicriteria Optimization,''
in \emph{EVOLVE -- A Bridge between Probability, Set Oriented Numerics, and Evolutionary Computation VI},
Advances in Intelligent Systems and Computing~674, 2018, pp.~18--36.
doi:10.1007/978-3-319-69710-9\_2.

\bibitem{wangrendeutzemmerich2016}
H.~Wang, Y.~Ren, A.~Deutz, and M.~Emmerich,
``On Steering Dominated Points in Hypervolume Indicator Gradient Ascent for Bi-Objective Optimization,''
in \emph{NEO 2015},
Studies in Computational Intelligence~663, 2017, pp.~175--203.
doi:10.1007/978-3-319-44003-3\_8.

\bibitem{wangdeutzbackemmerich2017}
H.~Wang, A.~Deutz, T.~B{\"a}ck, and M.~Emmerich,
``Hypervolume Indicator Gradient Ascent Multi-Objective Optimization,''
in \emph{Evolutionary Multi-Criterion Optimization},
Lecture Notes in Computer Science~10173, 2017, pp.~654--669.
doi:10.1007/978-3-319-54157-0\_44.

\bibitem{wangemmerichdeutzhernandezschutze2023newton}
H.~Wang, M.~Emmerich, A.~Deutz, V.~A.~S.~Hern{\'a}ndez, and O.~Sch{\"u}tze,
``The Hypervolume Newton Method for Constrained Multi-Objective Optimization Problems,''
\emph{Mathematical and Computational Applications} \textbf{28}(1) (2023), Article~10.
doi:10.3390/mca28010010.

\end{thebibliography}
\end{document}